\documentclass[11pt]{amsart}
\usepackage[all]{xy}
\usepackage{amsmath,amsfonts,amssymb,amsthm,epsfig,amscd,psfrag,framed,comment,color}
\usepackage{hyperref}

\allowdisplaybreaks
\usepackage{xargs}
\usepackage{environ}
\usepackage{graphicx}

\usepackage{tikz}
\usetikzlibrary{shapes}
\usetikzlibrary{backgrounds}
\usetikzlibrary{decorations,decorations.pathreplacing,decorations.markings}
\usetikzlibrary{fit,calc,through}
\usetikzlibrary{external}

\tikzstyle{mid>}=[decoration={markings, mark=at position 0.5 with {\arrow{>}}}, postaction={decorate}]
\tikzstyle{mid<}=[decoration={markings, mark=at position 0.5 with {\arrow{<}}}, postaction={decorate}]
\tikzstyle{upper>}=[decoration={markings, mark=at position 0.8 with {\arrow{>}}}, postaction={decorate}]
\tikzstyle{upper<}=[decoration={markings, mark=at position 0.8 with {\arrow{<}}}, postaction={decorate}]
\tikzstyle{lower>}=[decoration={markings, mark=at position 0.2 with {\arrow{>}}}, postaction={decorate}]
\tikzstyle{lower<}=[decoration={markings, mark=at position 0.2 with {\arrow{<}}}, postaction={decorate}]
\newcommand{\incircle}{\draw [very thick] (0,0) circle [radius=1]}

\newcommand{\outcircle}{\draw [very thick] (0,0) circle [radius=4]}

\renewcommand{\moveleft}[1]{
\draw[mid>] (45+15*#1:1) -- (45+15*#1:2);
\draw (45 + 15*#1:2) to [out=45 + 15*#1,in=240+15*#1] (60+15*#1:3);
\draw (60+15*#1:3) to (60 + 15*#1 :4);
}

\newcommandx{\NewEnvironx}[5][2,3]{%
  \expandafter\newcommandx\csname start#1\endcsname[#2][#3]{#4}%
  \NewEnviron{#1}{\csname start#1\expandafter\endcsname\BODY #5}}

\newcommand{\ladderX}{1.5}
\newcommand{\ladderY}{1.5}
\newcommand{\ladderR}{0.6}
\newcommand{\laddercoordinates}[2]{
\foreach \x in {0,...,#1} {
	\foreach \y in {0,...,#2} {
		\coordinate (l\x\y) at (\x * \ladderX, \y * \ladderY);
		\coordinate (u\x\y) at ($(l\x\y)+\ladderR*(0,\ladderY)$);
		\coordinate (d\x\y) at ($(l\x\y)+(0,\ladderY)-\ladderR*(0,\ladderY)$);
	}
}
}
\newcommand{\ladderEn}[5]{
\draw[mid>] (l#1#2) -- (d#1#2);
\draw[mid>] (d#1#2) -- ($(l#1#2)+(0,\ladderY)$) node[left] {#3};
\draw[mid>] ($(l#1#2)+(\ladderX,0)$) -- ($(u#1#2)+(\ladderX,0)$);
\draw[mid>] ($(u#1#2)+(\ladderX,0)$) -- ($(l#1#2)+(\ladderX,\ladderY)$) node[right] {#4};
\draw[mid>] (d#1#2) --node[above]{#5} ($(u#1#2)+(\ladderX,0)$);
}

\newcommand{\ladderFn}[5]{
\draw[mid>] (l#1#2) -- (u#1#2);
\draw[mid>] (u#1#2) -- ($(l#1#2)+(0,\ladderY)$) node[left] {#3};
\draw[mid>] ($(l#1#2)+(\ladderX,0)$) -- ($(d#1#2)+(\ladderX,0)$);
\draw[mid>] ($(d#1#2)+(\ladderX,0)$) -- ($(l#1#2)+(\ladderX,\ladderY)$) node[right] {#4};
\draw[mid>] ($(d#1#2)+(\ladderX,0)$) --node[above]{#5} (u#1#2);
}

\NewEnvironx{ladder}[2]{%
  \begin{tikzpicture}[baseline=13*\ladderY*#2]\laddercoordinates{#1}{#2}}
{\end{tikzpicture}}

\newcommand{\fuse}[3]{\tikz[baseline=0.5cm]{
\coordinate (z1) at (0,0);
\coordinate (z2) at (1,0);
\coordinate (c) at (0.5,0.5);
\coordinate (e) at (0.5,1);
\draw[mid>] (z1) node[below] {$#1$} -- (c);
\draw[mid>] (z2) node[below] {$#2$} -- (c);
\draw[mid>] (c) -- (e) node[above] {$#3$};
}}
\newcommand{\fork}[3]{\tikz[baseline=0.5cm]{
\coordinate (z1) at (0,1);
\coordinate (z2) at (1,1);
\coordinate (c) at (0.5,0.5);
\coordinate (e) at (0.5,0);
\draw[mid<] (z1) node[above] {$#1$} -- (c);
\draw[mid<] (z2) node[above] {$#2$} -- (c);
\draw[mid<] (c) -- (e) node[below] {$#3$};
}}

\newcommand{\scs}{\scriptstyle}
\newcommand{\qbin}[2]{
\left[
 \begin{array}{c}
 #1 \\
 #2 \\
 \end{array}
 \right]
}
\newcommand{\qbins}[2]{
\left[ \;
\vcenter{\xy
(0,3)*{\scs #1}; (0,-1)*{\scs #2};
\endxy} \;
 \right]
}

\def\coker{{\rm coker}}
\def\supp{{\rm supp}}
\def\Spec{{\rm Spec}}
\def\Proj{{\rm Proj}}

\def\ochi{{\overline{\chi}}}
\def\oE{{\overline{E}}}
\def\oR{{\overline{R}}}
\def\oZ{{\overline{Z}}}
\def\sC{{\mathcal{C}}}

\def\sF{{\mathcal{F}}}

\def\ux{{\underline{x}}}
\def\sN{{\mathcal N}}
\def\rk{{\rm rk}}
\def\G{{\mathbb{G}}}
\def\bhH{\widehat{{\mathbb{H}}}}
\def\B{\mathcal{B}}
\def\Sym{{\rm Sym}}
\def\sP{{\mathcal{P}}}
\def\sF{{\mathcal{F}}}
\def\sL{{\mathcal{L}}}
\def\la{\langle}
\def\ra{\rangle}
\def\l{\lambda}
\def\A{{SL_n \times \Cx}}
\def\P{{\mathbb{P}}}

\def\AB{{\mathcal{AB}}}
\def\B{{\mathcal{B}}}
\def\Beta{\beta}
\def\APsi{A\Psi}

\newcommand{\g}{\mathfrak{g}}
\newcommand{\dU}{\dot{U}}
\renewcommand{\sl}{\mathfrak{sl}}
\newcommand{\gl}{\mathfrak{gl}}
\newcommand{\C}{\mathbb{C}}
\newcommand{\cC}{\mathcal{C}}
\newcommand{\Cx}{\mathbb{C}^\times}
\newcommand{\K}{\mathcal{K}}
\newcommand{\N}{\mathbb{N}}
\newcommand{\ulam}{{\underline{\lambda}}}
\newcommand{\umu}{{\underline{\mu}}}
\newcommand{\dUL}{\dot{U}_q L \gl_m}
\newcommand{\dULq}{\dot{U}_{[q^\pm]} L \gl_m}
\newcommand{\dsU}{\dot{{\mathcal U}}_q \gl_m}
\newcommand{\dsUL}{\dot{{\mathcal U}}_q L \gl_m}
\newcommand{\Z}{\mathbb Z}

\newcommand{\pt}{{\rm{pt}}}

\newcommand{\RqS}{\mathcal{R}ep_q(SL_n)}

\def\RqSm{\mathcal{R}ep^{{min}}_{q}(SL_n)}
\def\R[q]Sm{\mathcal{R}ep^{{min}}_{[q^\pm]}(SL_n)}

\newcommand{\OgS}{\cO(\frgS)\mod}
\newcommand{\OSS}{\cO(\frSS)\mod}
\newcommand{\OqSS}{\cO_q(\frSS)\mod}

\def\OqSSm{\cO_{q}^{{min}}(\frSS)\mod}
\def\cO{\mathcal{O}}

\newcommand{\OS}{\cO(SL_n)}
\newcommand{\OqS}{\cO_q(SL_n)}

\newcommand{\RG}{\mathcal{R}ep(G)}
\newcommand{\RqG}{\mathcal{R}ep_q(G)}

\newcommand{\RqGm}{\mathcal{R}ep^{{min}}_q(G)}

\newcommand{\AqSp}{\mathcal{AS}p_n(q)}
\newcommand{\ASp}{\mathcal{AS}p_n}
\newcommand{\Sp}{\mathcal{S}p_n}
\newcommand{\qSp}{\mathcal{S}p_n(q)}

\newcommand{\OgG}{\cO(\frgG)\mod}
\newcommand{\OGG}{\cO(\frGG)\mod}
\newcommand{\OqGG}{\cO_q(\frGG)\mod}

\newcommand{\OGGm}{\cO^{min}(\frGG)\mod}
\newcommand{\OqGGm}{\cO_q^{min}(\frGG)\mod}

\newcommand{\frGG}{\tfrac{G}{G}}
\newcommand{\frgG}{\tfrac{\fg}{G}}
\newcommand{\frSS}{\tfrac{SL_n}{SL_n}}
\newcommand{\frgS}{\tfrac{{\mathfrak gl}_n}{GL_n}}

\newcommand{\Uq}{U_q \g}
\newcommand{\Oq}{\cO_q(G)}

\newcommand{\Altq}[1]{{\textstyle\bigwedge_q^{#1}}}

\newcommand{\Alt}[1]{{\textstyle \bigwedge^{#1}}}

\renewcommand{\mod}{\operatorname{-mod}}

\newcommand{\Gv}{{G^\vee}}
\newcommand{\Gc}{{\widetilde{\Gv}}}
\newcommand{\fg}{\mathfrak g}
\newcommand{\ft}{\mathfrak t}

\newcommand{\Uqn}{U_q \sl_n}

\newcommand{\uk}{{\underline{k}}}
\newcommand{\ul}{\underline{l}}

\newcommand{\ur}{\underline{r}}

\def\1{{1}}
\def\hH{\widehat{\mathcal{H}}}

\def\H{{\mathcal{H}}}

\DeclareMathOperator{\Hom}{Hom}
\DeclareMathOperator{\Ext}{Ext}
\DeclareMathOperator{\End}{End}
\DeclareMathOperator{\Maps}{Maps}

\newtheorem{Theorem}{Theorem}[section]
\newtheorem{Proposition}[Theorem]{Proposition}
\newtheorem{Lemma}[Theorem]{Lemma}

\newtheorem{Corollary}[Theorem]{Corollary}
\newtheorem{Conjecture}[Theorem]{Conjecture}
\newtheorem{Remark}[Theorem]{Remark}

\begin{document}
\setcounter{tocdepth}{1}

\title{Quantum K-theoretic geometric Satake: $ SL_n$ case}


\author{Sabin Cautis}
\email{cautis@math.ubc.ca}
\address{Department of Mathematics\\ University of British Columbia \\ Vancouver BC, Canada}

\author{Joel Kamnitzer}
\email{jkamnitz@math.utoronto.ca}
\address{Department of Mathematics \\ University of Toronto \\ Toronto ON, Canada}

\begin{abstract}
 The geometric Satake correspondence gives an equivalence of categories between the representations of a semisimple group $ G $ and the spherical perverse sheaves on the affine Grassmannian $Gr$ of its Langlands dual group.  Bezrukavnikov-Finkelberg developed a derived version of this equivalence which relates the derived category of $ G^\vee$-equivariant constructible sheaves on $ Gr $ with the category of $G$-equivariant ${\mathcal O}(\mathfrak g)$-modules.

In this paper, we develop a K-theoretic version of the derived geometric Satake which involves the quantum group $ U_q \mathfrak g $.  We define a convolution category $ KConv(Gr) $ whose morphism spaces are given by the $ G^\vee \times \mathbb C^\times $-equivariant algebraic K-theory of certain fibre products. We conjecture that $KConv(Gr)$ is equivalent to a full subcategory of the category of $ U_q \mathfrak g $-equivariant $ \mathcal O_q(G) $-modules.

We prove this conjecture when $G = SL_n$. A key tool in our proof is the $SL_n$ spider, which is a combinatorial description of the category of $U_q \mathfrak{sl}_n$ representations.  By applying horizontal trace, we show that the annular $SL_n$ spider describes the category of  $ U_q \mathfrak{sl}_n $-equivariant $ \mathcal O_q(SL_n) $-modules.  Then we use quantum loop algebras to relate the annular $SL_n $ spider to  $ KConv(Gr) $.  This gives a combinatorial/diagrammatic description of both categories and proves our conjecture.
\end{abstract}

\maketitle

\tableofcontents

\section{Introduction}

\subsection{The Geometric Satake correspondence}
Let $ G $ be a semisimple complex group.  Let $G^\vee $ be its Langlands dual group and let $ Gr = G^\vee((t))/G^\vee[[t]] $ be the affine Grassmannian.  Let $ P(Gr) $ denote the abelian category of perverse sheaves on $ Gr $ which are constructible with respect to the stratification by $ G^\vee[[t]] $-orbits.  This category is naturally equipped with a monoidal structure using convolution.  The geometric Satake correspondence of Mirkovic-Vybornov \cite{MV} gives the following result.

\begin{Theorem}
There is an equivalence of monoidal categories  $$ P(Gr) \cong \RG.$$
\end{Theorem}

 Let $ D_{G^\vee}(Gr) $ denote the $ G^\vee$-equivariant derived category of constructible sheaves on $ Gr$ and $ D_{G^\vee \times \Cx}(Gr) $ the $G^\vee \times \Cx $-equivariant derived category.  Here $ G^\vee $ acts naturally on $ Gr $ by left multiplication and $ \Cx $ acts by loop rotation.

The derived geometric Satake correspondence describes these categories of sheaves in terms of representation theory/geometry of the group $ G $. Let $ Coh^{G \times \Cx}(\fg) = \OgG$ denote the category of $G$-equivariant $\Z$-graded $ \cO(\fg)$-modules and likewise let $ U_\hbar(\tfrac{\fg}{G})\mod $ denote the category of $G$-equivariant graded $ U_\hbar(\fg)$-modules. Here $ U_\hbar(\fg) $ denotes the Rees algebra of the universal enveloping algebra $ U \fg $ with respect to the PBW filtration ($U_\hbar(\fg) $ is sometimes called the asymptotic universal enveloping algebra -- it should not be confused with the quantum group $U_q \g$). The $\Cx$ action on $\fg$ is by scaling, which means that in the induced $\Z$-grading on $ \cO(\fg) $ and $ U_\hbar(\fg)$, elements of $ \fg $ as well as $\hbar$ have degree $1$.  As usual, we will write $ \{1\} $ for a shift in the $\Z$-grading. Typical objects of $ \OgG $ and $  U_\hbar(\tfrac{\fg}{G})\mod $ are $ \cO(\fg) \otimes V $ and $ U_\hbar(\fg) \otimes V $ where $ V $ is a finite-dimensional representation of $ G $.

The following result is due to Bezrukavnikov-Finkelberg \cite[Corollary 1]{BF}.

\begin{Theorem} \label{th:DerivedSatake}
There are monoidal functors
\begin{align*}
  F : D(\OgG) \rightarrow D_{G^\vee}(Gr) \\
F: D(U_\hbar(\tfrac{\fg}{G})\mod) \rightarrow D_{G^\vee \times \Cx}(Gr)
\end{align*}
such that for any two objects $ A, B \in \OgG$, $ F$ gives an isomorphism
$$
\bigoplus_{n,m} \Ext^m(A, B\{n\}) \cong \bigoplus_k \Ext^k(F(A),F(B))
$$
(and similarly for $A, B \in U_\hbar(\tfrac{\fg}{G})\mod$).
\end{Theorem}

Using a standard result of Ginzburg \cite[Theorem 8.6.7]{CG} (see section 3.4 below), we can translate this into a statement about homology of certain fibre products.   Suppose that $ \ulam = (\lambda_1, \dots, \lambda_m) $ and $ \umu = ( \mu_1, \dots, \mu_{m'}) $ are two sequences of minuscule dominant weights for $ G $.  Then the geometric Satake equivalence and its derived version (Theorem \ref{th:DerivedSatake}) imply the following isomorphisms.

\begin{Theorem} \label{th:DerivedSatake2}
There are canonical isomorphisms
\begin{align*}
H_{top}(Z(\ulam, \umu)) &\cong \Hom_G(V(\ulam), V(\umu)) \\
H^{G^\vee}_*(Z(\ulam, \umu)) &\cong \Hom_{\OgG}(\cO(\fg) \otimes V(\ulam), \cO(\fg) \otimes V(\umu)) \\
H^{G^\vee \times \Cx}_*(Z(\ulam, \umu)) &\cong \Hom_{ U_\hbar(\tfrac{\fg}{G})\mod}(U_\hbar(\fg) \otimes V(\ulam), U_\hbar(\fg) \otimes V(\umu)).
\end{align*}
\end{Theorem}
Here $ Z(\ulam, \umu) $ is a certain fibre product defined using $ Gr $ (see section \ref{sec:homconv}), and $ V(\ulam) = V(\lambda_1) \otimes \cdots \otimes V(\lambda_m) $ is a tensor product of irreducible representations.

\subsection{A conjectural K-theoretic version}

One goal in this paper is to extend the geometric Satake correspondence and the isomorphisms from Theorem \ref{th:DerivedSatake2} to the quantum group $ U_q \g $. We will introduce the quantum parameter $q$ by replacing homology with $\mathbb C^\times $-equivariant $K$-theory\footnote{It is clear from Theorem \ref{th:DerivedSatake2} that working with $ \Cx$-equivariant homology does not introduce the quantum group.}.  We begin with the ansatz that passing from homology to K-theory on the left hand sides on Theorem \ref{th:DerivedSatake2} corresponds to passing from $ \cO(\fg) $ to $ \cO(G) $ on the right hand sides.  Let $ \cO_q(G) $ be the braided quantum function algebra, which is a quantization of $ \cO(G) $.  This leads us to the following conjecture.

\begin{Conjecture}\label{conj:main}
Let $ \ulam $ and $ \umu $ be two sequences of minuscule dominant weights. Up to appropriate localizations, there are canonical isomorphisms
\begin{align*}
K^{G^\vee}(Z(\ulam, \umu)) \cong \Hom_{\OGG}(\cO(G)  \otimes V(\ulam), \cO(G) \otimes V(\umu)) \\
K^{G^\vee \times \Cx}(Z(\ulam, \umu)) \cong \Hom_{\OqGG}(\Oq \otimes V(\ulam), \Oq \otimes V(\umu))
\end{align*}
where $\OGG$ is the category of $G$-equivariant $\cO(G)$-modules and $ \OqGG $ the category of $\Uq$-equivariant $\Oq$-modules.
\end{Conjecture}
Notice that we do not make any statement concerning $ \Hom_{\Uq}(V(\ulam), V(\umu)) $.  This is because we do not have an analog of ``top homology'' within the world of $ \Cx$-equivariant K-theory.

In section \ref{sec:kconv}, we define a category $ KConv^{G^\vee \times \Cx}(Gr) $ whose objects are sequences of minuscule dominant weights and whose morphisms are given by $K^{G^\vee \times \Cx}(Z(\ulam, \umu)) $ (and similarly without the $\Cx $). Then Conjecture \ref{conj:main} can be reformulated as the existence of full embeddings
\begin{align*}
KConv^{G^\vee}(Gr) &\rightarrow \OGG \\
KConv^{G^\vee \times \Cx}(Gr) &\rightarrow \OqGG
\end{align*}
(see Conjecture \ref{co:main}).

There are two explanations for the appearances of the categories $\OGG $ and $ \OqGG $.  On the one hand, they are the natural ``multiplicative'' analogs of the categories  $ \OgG$ and $ U_\hbar(\tfrac{\fg}{G})\mod$ from the work of Bezrukavnikov-Finkelberg \cite{BF}.  On the other hand, as we will recall in section \ref{se:hortrace}, they are the horizontal traces of the monoidal categories $ \RG $ and $ \RqG$.  Both perspectives will play a role in this paper.

The algebras $ \cO(\g), U_\hbar(\fg), \cO(G), \cO_q(G) $ fit into a natural ``diamond'', which we learnt from Jordan \cite{J}\footnote{We emphasize here that $ U_\hbar(\fg) $ is not the quantum group, just the asymptotic enveloping algebra, and the parameters $ \hbar $ and $ q $ are unrelated.}.
\begin{equation*}
\xymatrix{
& \cO_q(G) \ar@{-}[dr] & \\
U_\hbar(\fg) \ar@{-}[ru] & & \cO(G) \\
& \cO(\g) \ar@{-}[lu] \ar@{-}[ru] &
}
\end{equation*}
This leads to a diagram of categories.
\begin{equation*}\label{diamond1}
\xymatrix{
& \OqGG \ar@{-}[ld] \ar@{-}[rd] & \\
U_\hbar(\tfrac{\fg}{G})\mod & & \OGG \\
& \OgG \ar@{-}[ru] \ar@{-}[lu]&
}
\end{equation*}
Conjecture \ref{conj:main} can then be formulated as saying that this diagram of ``algebraic'' categories is related to the following diagram of ``topological'' categories.
\begin{equation*}\label{diamond2}
\xymatrix{
& KConv^{G^\vee \times \Cx}(Gr) \ar@{-}[ld] \ar@{-}[rd] & \\
D_{G^\vee \times \Cx}(Gr) & & KConv^{G^\vee}(Gr) \\
& D_{G^\vee}(Gr) \ar@{-}[lu] \ar@{-}[ru] &
}
\end{equation*}

\subsection{Proof in type A}
Most of this paper is devoted to proving the above conjecture for $G = SL_n$.

Let $ \OqSSm $ be the subcategory of $\Uqn$-equivariant $\OqS$-modules of the form $ \OqS \otimes V $, where $ V $ is a tensor product of fundamental representations.  Our strategy is to give a combinatorial/diagrammatic presentation of both categories $ \OqSSm $ and $ KConv^{\A}(Gr) $ using the annular $ SL_n$ spider.  We believe that these presentations are of their own independent interest.

Our strategy involves a series of functors and categories summarized in the commutative diagram (\ref{fig:summary}). We now briefly discuss these categories and functors.

\begin{equation}\label{fig:summary}
\xymatrix{
(\dUL)^n \ar[r]^{\APsi_m} \ar[dr]^{\Phi_m} & \ASp(q) \ar[r]^{A\Gamma} \ar[d]^\Phi & \OqSSm \\
 & KConv^{\A}(Gr) &
}
\end{equation}

\subsubsection{Annular spiders}
The $SL_n$ spider category $\Sp$ is a combinatorial/diagrammatic presentation of the representation category of $SL_n$ due to the authors and Morrison \cite{CKM}\footnote{In \cite{CKM}, we worked over $\C(q)$, as we (mostly) do in this paper.  Later, Elias \cite{E2} extended the equivalence $\Sp(q) \cong \RqSm$ to $ \Z[q,q^{-1}] $.  Thus it is possible that the results of the current paper hold over $ \Z[q,q^{-1}] $.  But we have not explored this in depth.}.   In this paper, we use the annular version $\ASp$ of this category, where morphisms are linear combination of webs (with boundary) in an annulus modulo the same local relations as in $\Sp$. Using the machinery of horizontal trace and the main result in \cite{CKM}, we obtain the following.

\begin{Theorem}
There is an equivalence of categories
$$ A\Gamma : \AqSp \xrightarrow{\sim} \OqSSm.$$
\end{Theorem}

\subsubsection{Quantum loop algebras}

The main tool in \cite{CKM} was the realization that via skew-Howe duality, $ \RqS $ (and thus $ \qSp $) is the limit as $m \rightarrow \infty$ of truncations $ (\dU_q \gl_m)^n $ of the idempotented version of quantum groups.  In this paper, we will use that $\AqSp$ is the limit of truncations $(\dUL)^n$ of idempotented quantum loop algebras.  This is achieved by defining an annular ``ladder-formation'' functor $ \APsi_m : (\dUL)^n \rightarrow \AqSp $.

On the other hand, Nakajima \cite{Na1} (following Ginzburg and Vasserot) constructed actions of quantum loop algebras on the K-theory of quiver varieties.  In type A, the varieties $ Z(\ulam, \umu) $ are closely related to Steinberg varieties of quiver varieties.  Motivated by this connection and by our earlier work \cite{CKL,Ca,CK3}, we construct functors
$$ \Phi_m : (\dUL)^n \rightarrow KConv^{\A}(Gr) $$
for each $m$. These functors are compatible with the inclusion of $(\dUL)^n$ into $(\dot{U}_q L \gl_{m+1})^n$ and thus induce a functor
$$ \Phi : \AqSp \rightarrow KConv^{\A}(Gr).$$

\subsubsection{The proof}

Since we have an equivalence $ A\Gamma : \AqSp \xrightarrow{\sim} \OqSSm $ it suffices to show that $\Phi : \AqSp \rightarrow KConv^{\A}(Gr)$ is an equivalence. Using a simple trick, this can be reduced to showing that the map $ \Phi : \End_{\AqSp}(1^m) \rightarrow K^{\A}(Z(1^m, 1^m)) $ is an isomorphism (where $ 1^m = (1,\dots, 1) $). To do this we consider the composition
$$\End_{\AqSp}(1^m) \rightarrow K^{\A}(Z(1^m, 1^m)) \rightarrow \End(K^{\A}(Y(1^m)))$$
and compute explicitly the images of everything inside the right hand term.  

Here $ Y(1^m) $ is an iterated fibre product of $ m $ copies of $ \P^{n-1} $ which is defined using the affine Grassmannian. To do this computation, we study a certain ring $R_{n,[q^\pm]}^m$ (see section \ref{sec:qDefR} for the definition) which we identify with the equivariant K-theory of $Y(1^m)$ (Lemma \ref{lem:ktheory}).  We also study certain endomorphisms of this ring which give us an algebra closely related to cyclotomic affine Hecke algebras (we discuss this in section \ref{se:cyclotomic}). All these rings are quite explicit and may be of independent interest. 

\subsection{Relationship to other work}

\subsubsection{Betti geometric Langlands}

The work of Kapustin-Witten \cite{KW}, and more recent work of Ben-Zvi and Nadler (in progress), describes a Betti version of the geometric Langlands program. In this setup one has a (partially defined) 4-dimensional topological field theory. To a surface this theory associates a category. These categories were studied recently in \cite{BBJ}. In the case when the surface is an annulus the theory gives $\OqGG$. This category along with the category associated to a pair of pants form the building blocks of this TFT.

The results in this paper relate $\OqGG$ with the $K$-theoretic convolution category of the affine Grassmannian for the Langlands dual group. It would be interesting to obtain a similar result for the category associated to a pair of pants.  We hope that our results will be an important step towards the Betti geometric Langlands program.

\subsubsection{Gaitsgory's work}

There is already a quantum version of the geometric Satake correspondence due to Gaitsgory \cite{Ga}.  In his work, the category of representations of the quantum group is realized as the category of twisted Whittaker sheaves on the affine Grassmannian. The relationship between Gaitsgory's result and the present work is not at all clear.  In particular, it follows that the horizontal trace of the category of twisted Whittaker sheaves is equivalent to our category $ KConv^\A(Gr)$ --- we do not have a geometric explanation for this fact.

\subsubsection{Elias' work}
Our work can also be compared with a recent paper by Elias \cite{E}.  Analogous to the category $KConv^{G^\vee}(Gr)$, one may also define $ hConv^{G^\vee}(Gr) $ (see section \ref{sec:homconv}), which is equivalent to a full subcategory of $ D_{G^\vee}(Gr) $.  Via the work of Soergel and others (see section 6 of \cite{E}), the category $ hConv^{G^\vee}(Gr) $ is equivalent to the category of maximally singular Bott-Samelson bimodules $ mSBSBim $ for the corresponding affine Weyl group.

Assume now $ G = SL_n$.  By combinatorial/algebraic methods, Elias has constructed an equivalence between $ \Sp $ and the degree 0 part of $ mSBSBim $ (in \cite{E} this is done for $ n = 2, 3$ and the general case will appear in a future work).  This reproves the original abelian geometric Satake equivalence in this case.  This should be compared with our equivalence $ \ASp \cong KConv^{G^\vee}(Gr) $.  (Our work gives an equivalence between $ \Sp$ and a subcategory of $ KConv^{G^\vee}(Gr) $, but it is not clear whether this reproves the original abelian geometric Satake equivalence.)

Moreover, Elias gives a $q$-deformation $ mSBSBim_q $ which he proves is degree 0 equivalent to $ \qSp$.  We do not know how this $q$-deformation is related to our equivalence $ \AqSp \cong KConv^{\A}(Gr) $.

\subsubsection{Categorification}
In our previous work \cite{CKM}, we suggested that the equivalence $ \qSp \cong \RqS $ could be used to define a categorification of $\RqS$, called $Foam_n$, as the a limit of categorifications of $(\dU_q \gl_m)^n$
$$\lim_{m \rightarrow \infty} (\dsU)^n \cong Foam_n.$$
This program was carried out by Queffelec-Rose \cite{QR}.  The foam category $Foam_n$ has objects sequences from $\{1, \dots, n-1\}$, 1-morphisms given by webs, and 2-morphisms given by certain bordisms with seams, called foams.

In our current paper, we have functors
$$ (\dUL)^n \rightarrow \ASp \rightarrow KConv^{\A}(Gr) $$
and it is natural to look for categorifications of each of these categories. One expects that $ (\dUL)^n $ can be categorified by some 2-category $ (\dsUL)^n $ analogous to $(\dsU)^n$ (though this has not been done as far as we know).  $ \ASp $ can be categorified by  $AFoam_n$, the category whose objects are the same as $ Foam_n$, whose 1-morphisms are webs in the annulus (with boundary) and whose 2-morphisms are annular foams. In a recent paper \cite{QR2}, Queffelec-Rose study the endomorphisms of the trivial object in this category.  Finally, $ KConv^{\A}(Gr) $ has a natural categorification $ CohConv^{\A}(Gr) $, where the morphism categories are given by derived categories of coherent sheaves on the fibre products $ Z(\ulam, \umu) $.

Thus it is natural to expect that we should have 2-functors
$$ \lim_{m \rightarrow \infty} (\dsUL)^n \rightarrow AFoam_n \rightarrow CohConv^{\A}(Gr).$$
The composition of these 2-functors is understood to some extent, based on this paper and earlier work by the authors. We expect the first 2-functor to be an equivalence while the second 2-functor is more mysterious.

\subsubsection{Knot invariants}
The spider category $\qSp$ can be used to define Reshetikhin-Turaev (RT) invariants of type A for links in the ball. Similarly, the categorification $Foam_n$ can be used to define homological link invariants (this was done in \cite{QR} following the original approach due to Khovanov \cite{Ksl3}).

The results of this paper show that $ \qSp $ embeds into $KConv^\A(Gr) $ and thus $KConv^\A(Gr) $ can be used to define Reshtikhin-Turaev invariants. This suggests that the 2-category $CohConv^{\A}(Gr)$ can be used to define homological knot invariants. This is essentially the approach used in our previous papers \cite{CK1,CK2,Ca}. Our current work gives a little more perspective to these papers.

Moreover, the results of this paper show that $ KConv^\A(Gr) $ can be used to define RT-invariants for links in the annulus. This suggests that $CohConv^{\A}(Gr)$ can be used to obtain homological invariants of links in the annulus. The case $G = SL_2$ should correspond to sutured Khovanov homology. Using this approach, we immediately obtain an action of $SL_n$ on the $SL_n$-homology of links in the annulus. This is because the resulting invariant lies in the $SL_n \times \C^\times$-equivariant derived category of a point (in the case $n=2$, such an action was studied by Grigsby-Licata-Wehrli \cite{GLW}). This construction is compatible with the recent work of Queffelec-Rose \cite{QR2}. More details will appear in a future paper.

\subsection*{Acknowledgements}
We would like to thank Pierre Baumann, David Ben-Zvi, Alexander Braverman, Adrien Brochier, Roman Bezrukavnikov, Chris Dodd, Ben Elias, Pavel Etingof, Michael Finkelberg, Dennis Gaitsgory, David Jordan, Chia-Cheng Liu, Peter McNamara, Scott Morrison, Hoel Queffelec, David Rose, Peter Samuelson, Xinwen Zhu for helpful discussions.  We particularly thank Michael Finkelberg for helping us to formulate the main conjecture.  We also thank the anonymous referee for many helpful suggestions. The first author was supported by NSERC and thanks MSRI for their support and hospitality. The second author was supported by NSERC, Sloan Foundation, Simons Foundation, and SwissMAP, and he thanks the geometry group at EPFL for their hospitality.

\section{Representation categories}

\subsection{Notation}
Let $G$ be a simply-connected semisimple group.  Let $ T $ be a maximal torus of $ G $, let $ \Lambda $ denote its weight lattice and let $ W $ denote the Weyl group.  Let $ \Lambda_+ $ denote the set of dominant weights and $ \rho^\vee $ half the sum of the fundamental coroots.  For $ \lambda \in \Lambda_+ $, let $ V(\lambda) $ denote the irreducible representation of highest weight $ \lambda $.

Recall that we have an isomorphism $ \cO(G)^G \cong \cO(T)^W $ where $ G $ acts on $ G $ by the adjoint action.  Let $ E = \cO(T)^W \cong \C[\Lambda]^W $.  Recall that $ E $ is a polynomial ring in $ r $ variables, where $ r $ is the rank of $ G $.  We also have an isomorphism $ R(G) \cong E $, where $ R(G) $ denotes the (complexified) representation ring of $ G $.

Let $ G^\vee $ denote the Langlands dual group to $ G$ and let $ \Gc $ be its simply-connected cover.  Let $ \Lambda^\vee $ be the weight lattice of $ \Gc $ and let $ E^\vee = R(\Gc) = \C[\Lambda^\vee]^W $.

Let $(,)$ be the $ W$-invariant bilinear form on $ \ft^* $ (the dual of the Lie algebra of $T$), such that $(\alpha, \alpha) = 2 $ for all short roots $ \alpha $.  This bilinear form gives an isomorphism $ \iota : \ft^* \rightarrow \ft $.  From the $ W$-invariance, it follows that $ \iota(\Lambda) \subset \Lambda^\vee $ (with equality if $ G $ is simply-laced).  Thus $ \iota $ gives us an inclusion $ E \hookrightarrow E^\vee $ (which is an isomorphism when $ G $ is simply laced).

\subsection{Classical representation categories}
Let $\RG$ denote the usual category of finite-dimensional representations of $ G$.  We will be interested in various enhancements/modifications of $ \RG $.

We have the adjoint actions of $ G $ on $ \fg $ and $ G $.  This makes $ \cO(\fg) $ and $ \cO(G) $ into (infinite-dimensional) $ G $ representations.  Let $ \OgG$ denote the full-subcategory of $ G $-equivariant $ \cO(\fg) $-modules which are of the form $ \cO(\fg) \otimes V$, for $ V $ in $ \RG $.  Here $ G $ acts diagonally on $ \cO(\fg) \otimes V$ and $ \cO(\fg) $ acts on the left tensor factor.    Similarly we define $ \OGG$ to be the full-subcategory of $ G$-equivariant $ \cO(G) $-modules which are of the form $ \cO(G) \otimes V$.  Note that we can think of $ \OgG $ as the full subcategory of $ G$-equivariant coherent sheaves of $ \fg $ consisting of trivial vector bundles with $G$-action (and similarly for $ \OGG$).

\subsection{Quantum representation categories} \label{sec:quantumgroups}

\subsubsection{Quantum group and quantum function algebra}
We will denote by $[n]$ the quantum integer $q^{n-1} + q^{n-3} + \dots + q^{-n+3} + q^{-n+1}$. More generally, we have quantum binomial coefficients
$$\qbin{n}{k} := \frac{[n] \dots [1]}{([n-k] \dots [1])([k] \dots [1])}.$$

Let $ \Uq $ denote the quantum group over $ \C(q) $. Let $ \RqG $ denote the category of finite-dimensional representations of $ \Uq $.  For each $ \lambda \in \Lambda_+ $, we have an irreducible representation $ V(\lambda) $.  Recall that the representation ring $R_q(G) $ of $ \Uq $ coincides with the representation ring of $ G $, so $ R_q(G) \cong E $. Also recall that for any two representations $ V,W $ of $ \Uq $, we have a braiding map $ \beta_{V,W} : V \otimes W \rightarrow W \otimes V $.

We consider the quantum function algebra $ \Oq $.  There are (at least) two versions of $ \Oq $ in the literature.  We will need the version introduced by Majid, which sometimes goes under the name ``reflection equation algebra''. We define $ \Oq $ to be the subspace of $ (\Uq)^* $ spanned by matrix coefficients $ V \otimes V^* $ of finite-dimensional representations $ V $.  From this we see that as a vector space $ \Oq = \oplus_{\lambda} V(\lambda) \otimes V(\lambda)^* $ and we have a map $ r_V : V \otimes V^* \rightarrow \Oq $.

There is an action of $ \Uq $ on $ \Oq $ such that for all $ V$, $ r_V : V \otimes V^* \rightarrow \Oq $ is $\Uq$-equivariant.

The multiplication on $ \Oq $ is defined so that for any $ V, W $ the diagram
$$ \xymatrix{
(V \otimes V^*) \otimes (W \otimes W^*) \ar[rr]_>>>>>>>>{I_V \otimes \beta^{-1}_{W \otimes W^*, V^*}} \ar[d]^{r_V \otimes r_W} & &V \otimes W \otimes W^* \otimes V^* = (V \otimes W) \otimes (V \otimes W)^* \ar[d]^{r_{V \otimes W}} \\
\Oq \otimes \Oq \ar[rr] & & \Oq }
$$
commutes.

Since the upper horizontal line in this diagram is a map of $ \Uq $ modules, we see that multiplication in $\Oq $ is equivariant for this action.

Let $ \OqGG$ denote the full subcategory of $ \Uq $-equivariant $\Oq $-modules which are of the form $ \Oq \otimes V $, where $ V $ is a representation $ \Uq $.  Here $ \Uq $ acts on $ \Oq \otimes V $ by the diagonal action and $ \Oq $ acts by left multiplication.

\subsubsection{Central elements and enrichment} \label{se:enrich}
For any representation $ V$, let $ t_V \in \Oq $ denote the image of the canonical element of $ V \otimes V^* $ under the map $ r_V $ (so interpreted as an element of $ (\Uq)^*$, $ t_V $ is the trace along the representation $ V $).
\begin{Proposition} \label{pr:OqUq}
The elements $ t_V $ are central and the map $ [V] \mapsto t_V $ defines an algebra isomorphism between $ E(q) \cong R_q(G) $ and $ (\Oq)^{\Uq} $.
\end{Proposition}
\begin{proof}
To prove that $ t_V $ is central, let us write the canonical element of $ V \otimes V^* $ as $ \sum_i x_i \otimes x^i $ where $ x_i $ is a basis for $ V $ and $ x^i $ is the dual basis.  Consider some element $ r_W(w \otimes w') \in \Oq $ for some $w \in W, w' \in W^*$.  Then we have
\begin{align*}
t_V r_W(w \otimes w') &= \sum_i r_{V \otimes W}\bigl( x_i \otimes \beta_{W \otimes W^*,V^*}^{-1}(x^i \otimes (w \otimes w')) \bigr) \\
&= \sum_i r_{V \otimes W}\bigl( ( 1 \otimes 1 \otimes \beta_{W^*,V^*}^{-1})(x_i \otimes \beta_{W,V^*}^{-1}(x^i \otimes w) \otimes w') \bigr) \\
&= \sum_i r_{V \otimes W}\bigl((\beta_{W,V} \otimes \beta_{W^*,V^*}^{-1})( w \otimes x_i \otimes x^i \otimes w')\bigr) \\
&= \sum_i r_{V \otimes W}\bigl( x_i \otimes w \otimes w' \otimes x^i\bigr)
\end{align*}
where the last line follows from the fact that $ \beta_{W,V} $ and $ \beta_{W^*, V^*} $ are dual to each other.  A similar (but simpler) calculation shows that $  r_W(w \otimes w') t_V $ is given by the same formula.  Thus $ t_V $ is central.

Next note that $ t_V $ is clearly $ \Uq$-invariant. Since $ t_{V\otimes W} = t_V t_W $ we see that the map is an algebra morphism.  It is clearly injective.  Finally, to show that it is surjective, we just use the $ \Uq $-equivariant isomorphism $ \Oq \cong \oplus_\lambda V(\lambda) \otimes V(\lambda)^* $.
\end{proof}

\begin{Remark}
We believe that these $ t_V $ generate the centre of $\Oq$ and thus the centre is isomorphic to $E(q)$.  We were not able to find a reference for this fact.
\end{Remark}

Using these elements $ t_V $, we can see that the category $ \OqGG $ is enriched over $ E(q)$ (the hom spaces are $E(q)$-modules).  Equivalently, we have a map from $ E(q) $ to endomorphisms of the identity functor in $ \OqGG$ given by sending an element of $ E(q) \cong (\Oq)^{\Uq} $ to its action on $ \Oq \otimes V $ given by multiplication on the $\Oq$ factor.

\subsubsection{Coaction and canonical automorphism} \label{sec:coact}
For any representation $ V $ of $ \Uq $, we have a coaction map $ C_V: V \rightarrow \Oq \otimes V $ defined by
$$ V \rightarrow V \otimes V^* \otimes V \xrightarrow{r_V \otimes I} \Oq \otimes V , $$
where the first map is the canonical coevaluation (creating the second two tensor factors).  This coaction map is a $ \Uq $-module map.

Using the coaction, the category $ \OqGG $  comes with an automorphism $ X $ of the identity functor.  For $ V \in \RqG $, we define $X_V $ as the composition
\begin{equation} \label{eq:phidefq}
 \Oq \otimes V  \xrightarrow{ I \otimes C_V}  \Oq \otimes \Oq \otimes V  \xrightarrow{m \otimes \theta_V}  \Oq \otimes V
\end{equation}
where $ m : \Oq \otimes \Oq \rightarrow \Oq $ is the multiplication, and where $ \theta_V $ is the balancing isomorphism (which acts by a power of $ q $ on each irreducible representation, see section 2.2 of \cite{BK}).  Since each arrow is a map of $ \Uq$-modules, so is the composite.

\begin{Remark}
If we work with $ \OGG $, this definition also makes sense.  In this case, $ X_V$  is the automorphism of the trivial vector bundle over $ G $ with fibre $ V $, given by acting by $ g $ in the fibre over $ g $.
\end{Remark}

\subsection{Horizontal trace and $ \OqGG $} \label{se:hortrace}
One reason why these categories $ \OqGG $ are important for us is because they arise as horizontal traces.  We now explain this construction.  The results in this section seem to be known to experts, but we were not able to find an adequate reference.

\subsubsection{Definition of horizontal trace}
Given a skeletally small monoidal category $ \cC $, we define its horizontal trace $ \cC(S^1) $ to have the some objects as $\cC $ but with morphisms defined as follows.   Fix two objects $A, B $ in $ \cC$.  We consider pairs
 $(Z, \phi) $ where $ Z $ is an object of $ \cC $ and $\phi $ is a morphism
$$ \phi: A \otimes Z \rightarrow Z \otimes B $$
in $ \cC$.

We define an equivalence relation on these pairs as follows.  Suppose that $ Z, Z' $ are objects of $ \cC $, $ \alpha : Z \rightarrow Z' $ is a morphism in $ \cC$, and $ \psi : A \otimes Z' \rightarrow Z \otimes B $ is a morphism in $ \cC$, then we can form two pairs $ (Z, \psi \circ 1 \otimes \alpha) $ and $(Z', \alpha \otimes 1 \circ \psi) $.  Under this circumstance, we say $  (Z, \psi \circ 1 \otimes \alpha) $ and $(Z', \alpha \otimes 1 \circ \psi) $ are related and we consider the equivalence relation generated by these relations.  Then we define the set $ \Hom_{\cC(S^1)}(A, B) $ to be the set of equivalence classes of these pairs.  Composition of morphisms is defined in the natural way.

\begin{Remark}
We learned of the definition of horizontal trace from Scott Morrison.  This definition was introduced by Walker \cite{Wal} under the name ``annularization''.  The name ``horizontal trace'' appears in the paper \cite{BHLZ}. Finally, in the paper \cite{BBJ}, the same concept is defined under the name ``zeroth Hochschild homology''.
 \end{Remark}

\subsubsection{Monoidal structure} \label{sec:monstructure}
In general the category $ \cC(S^1) $ has no monoidal structure. However, if $ \cC$ is a braided monoidal category with braiding $ \beta $, then $\cC(S^1) $ has a natural monoidal structure, where tensor product of objects is the same as in $ \cC $ and tensor product of morphism is given by
$$ [Z_1, \phi_1] \otimes [Z_2, \phi_2] = [Z_1 \otimes Z_2, I \otimes \beta_{B_1, Z_2} \otimes I \circ \phi_1 \otimes \phi_2 \circ I \otimes \beta_{A_2, Z_1} \otimes I ] $$
where $ [Z_1, \phi_1] \in \Hom_{\cC(S^1)}(A_1, B_1) $ and $(Z_2, \phi_2) \in \Hom_{\cC(S^1)}(A_2, B_2)$.

\subsubsection{Semisimple case}
If $ \cC $ is semisimple, then we can get a simpler description of morphisms in $ \cC(S^1)$.

\begin{Lemma}
Suppose that $ \cC $ is a semisimple abelian category, with simple objects $ \{V_i \}_{i \in I} $.  Then we have an isomorphism
$$
\Hom_{\cC(S^1)}(A, B) \cong \oplus_{i \in I} \Hom_{\cC}(A \otimes V_i, V_i \otimes B).
$$
\end{Lemma}

\begin{proof}
Fix $ A, B $ objects of $ \cC $. First, we define an abelian group structure on $ \Hom_{\cC(S^1)}(A, B) $ as follows.

Given $ [Z_1, \phi_1]$ and $ [Z_2, \phi_2] $ in $  \Hom_{\cC(S^1)}(A, B) $, we define
$$ [Z_1, \phi_1] +  [Z_2, \phi_2] = [Z_1 \oplus Z_2, \left[ \begin{smallmatrix} \phi_1 & 0 \\ 0 & \phi_2 \end{smallmatrix} \right] ] $$
where the matrix represents a morphism $ A \otimes (Z_1 \oplus Z_2) = (A \otimes Z_1) \oplus (A \otimes Z_2) \rightarrow (Z_1 \otimes B) \oplus (Z_2 \otimes B) = (Z_1 \oplus Z_2) \otimes B $.  This extends the usual addition of morphisms; if $ Z_1 = Z_2 = Z $ then the diagonal map $ Z \rightarrow Z \oplus Z $ gives an equivalence between $ (Z, \phi_1) + (Z, \phi_2) $ and $ (Z, \phi_1 + \phi_2)$.

Together with the usual scaling of morphisms, this results in a $ \C$-vector space structure on $ \Hom_{\cC(S^1)}(A, B) $.

This gives us a map
\begin{equation}\label{eq:RHStoLHS}
\oplus_{i \in I} \Hom_{\cC}(A \otimes V_i, V_i \otimes B) \rightarrow \Hom_{\cC(S^1)}(A, B).
\end{equation}

On the other hand, suppose that $ Z = Z_1 \oplus Z_2 $ and $ [Z, \phi] \in \Hom_{\cC(S^1)}(A, B)$.  Then we can write
$$ \phi = \left[ \begin{smallmatrix} \phi_{11} & \phi_{12} \\ \phi_{21} & \phi_{22} \end{smallmatrix} \right]  =    \left[ \begin{smallmatrix} \phi_{11} & \phi_{12} \\ 0 & 0 \end{smallmatrix} \right] + \left[ \begin{smallmatrix} 0 & 0 \\ \phi_{21} & \phi_{22} \end{smallmatrix} \right].
$$
Consider the inclusion map $ \alpha : Z_1 \rightarrow Z_1 \oplus Z_2$.  Using this map, we can see that $ [Z, \left[ \begin{smallmatrix} \phi_{11} & \phi_{12} \\ 0 & 0 \end{smallmatrix} \right] ] = [Z_1, \phi_{11}] $.  Similarly, $ [Z, \left[ \begin{smallmatrix} 0 & 0 \\ \phi_{21} & \phi_{22} \end{smallmatrix} \right]] = [Z_2, \phi_{22}] $.  Thus $ [Z, \phi] = [Z_1, \phi_{11}] + [Z_2, \phi_{22}]$. Applying this repeatedly leads to an inverse to (\ref{eq:RHStoLHS}) .
\end{proof}

When $ \cC = \RqG $, then $ \cC(S^1)$ is quite familiar.  (This result also appears as \cite[Theorem 4.16]{BBJ}.)
\begin{Proposition} \label{pr:OCmod}
We have an equivalence of categories $ \RqG(S^1) \cong \OqGG $.
\end{Proposition}

\begin{proof}
Let $ V, W \in \RqG $.  Then,
\begin{align*}
\Hom_{\RqG(S^1)}(A, B) &\cong \oplus_{\lambda \in \Lambda_+} \Hom_{\RqG}(A \otimes V(\lambda), V(\lambda) \otimes B) \\
&\cong \oplus_{\lambda \in \Lambda_+} \Hom_{\RqG}(A,  V(\lambda) \otimes B \otimes V(\lambda)^*) \\
&\cong \oplus_{\lambda \in \Lambda_+} \Hom_{\RqG}(A,  V(\lambda) \otimes V(\lambda)^* \otimes B) \\
&\cong \Hom_{\RqG}(A, \Oq \otimes B) \\
&\cong \Hom_{\OqGG}(\Oq \otimes A, \Oq \otimes B).
\end{align*}
Here in the third line we use the inverse braiding $\beta^{-1}_{V(\lambda)^*, B}$ and in the last line, we use that $ \Oq \otimes A $ is free as an $ \Oq $-module.

Moreover, these isomorphisms are compatible with composition on both sides.
\end{proof}

\subsubsection{Idempotent completion}
Finally, we will need to examine the behaviour of horizontal trace under idempotent completion.

\begin{Proposition} \label{pr:Idemp}
Let $ \cC $ be a monoidal category and let $ \overline{\cC} $ be its idempotent completion.  Then the functor $ \cC(S^1) \rightarrow \overline{\cC}(S^1) $ is fully-faithful.
\end{Proposition}

\begin{proof}
By general results, we know that the functor $ \cC \rightarrow \overline{\cC} $ is fully-faithful.

Let $ A, B $ be objects of $ \cC $.  We would like to show that the map
$$
\Hom_{\cC(S^1)}(A, B) \rightarrow \Hom_{\overline{\cC}(S^1)}(A, B)
$$
is an isomorphism.

To show surjectivity, let $ [Z, \phi] $ be a morphism in $ \overline{\cC}(S^1) $.  Then there is an object $ Y $ of $ \cC$ and an isomorphism $ Y = Z \oplus Z' $ in $ \overline{\cC} $.  Then consider  $ \psi : A \otimes (Z \oplus Z') \rightarrow (Z \oplus Z') \otimes B $ defined by the matrix $ \left[ \begin{smallmatrix}  \phi & 0 \\ 0 & 0 \end{smallmatrix} \right]$ --- it is a morphism in $ \cC $.  Then $[Y, \psi] $ is a morphism in $ \cC(S^1) $ and the image of $ [Y, \psi] $ in $\Hom_{\overline{\cC}(S^1)}(A,B) $ will be $ [Z, \phi] $.

The proof of injectivity is similar.
\end{proof}

\subsection{Action of annular braids} \label{sec:AnnularBraids}

As mentioned above, the procedure of horizontal trace is closely related to the circle or annulus.  In particular, if $ \mathcal C $ is braided monoidal, then the annular braid group acts on tensor products in $\mathcal C(S^1)$. Let us formulate this fact in a precise fashion in the case of $ \RqG$.

\subsubsection{The annular braid group}
Consider the annular braid group $AB_m$ (sometimes it is also called the extended affine braid group). It has generators $ X_1, \dots, X_m $ and $ T_1, \dots, T_{m-1} $ subject to the following relations:
\begin{enumerate}
\item $T_i T_j = T_j T_i$ if $|i-j| > 1$ and $T_i T_j T_i = T_j T_i T_j$ if $|i-j|=1$,
\item $T_i X_j = X_j T_i$ if $j \ne i,i+1$,
\item $T_i X_i T_i = X_{i+1}$ for $i=1, \dots, m-1$,
\item $X_i X_j = X_j X_i$.
\end{enumerate}
These can be viewed as braids with $m$ strands lying on an annulus.  More precisely $ AB_m = \pi_1(A^m \smallsetminus \Delta / S_m) $, where $ A $ is the annulus.  The $T_i$ corresponds to strand $ i $ crossing over strand $ i+1 $, while $X_i$ is the braid where the $i$th strand curls itself once counterclockwise around the annulus, passing under strands $ 1, \dots, i-1$ and over strands $ i+1, \dots, m $ (from this description it is clear why $X_iX_j=X_jX_i$).

Another common description of $AB_m$ is using generators $T_0, T_1, \dots, T_{m-1}$ and $R$ where the $T_i$ satisfy the same relations as above while $R$ satisfies $R T_i = T_{i-1} R$ (modulo $m$). Here $R$ is the braid where each strand moves to the counterclockwise one position.  From this description it is clear that $R^m = X_1 \dots X_m$ while an easy exercise shows that $R = T_{m-1} \cdots T_1 X_1 $.

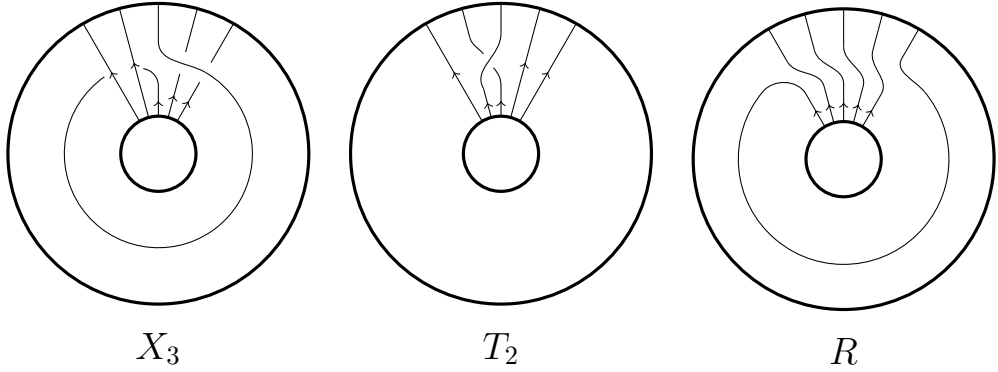
\begin{figure}
\centering
\begin{tikzpicture}[scale=0.5]
\incircle;
\outcircle;
\draw[mid>] (105:1) -- (105:4);
\draw[mid>] (120:1) -- (120:4);
\draw[mid>] (60:1) -- (60:2.2);
\draw (60:2.8) -- (60:4);
\draw[mid>] (75:1) -- (75:2.2);
\draw (75:2.8) -- (75:4);
\draw[mid>] (90:1) to (90:1.8);
\draw (90:1.8) to [out=90, in=350] (100:2.3);
\draw (125:2.5) arc [radius=2.5, start angle=125, end angle=420];
\draw (60:2.5) to [out=150, in=270] (90:3);
\draw (90:3) -- (90:4);
\node [below] at (270:4.5) {\Large $ X_3$};
\end{tikzpicture} \quad
\begin{tikzpicture}[scale=0.5]
\incircle;
\outcircle;
\draw[mid>] (105:1) -- (105:1.8);
\draw[mid>] (120:1) -- (120:4);
\draw[mid>] (60:1) -- (60:4);
\draw[mid>] (75:1) -- (75:4);
\draw[mid>] (90:1) to (90:1.8);
\draw (90:1.8) to [out=90, in=310] (95:2.3);
\draw (100:2.7) to [out=130, in=285] (105:3.2);
\draw (105:3.2) to (105:4);

\draw (105:1.8) to [out=120, in=270] (90:3.2);
\draw (90:3.2) to (90:4);

\node [below] at (270:4.5) {\Large $ T_2$};
\end{tikzpicture} \quad
\begin{tikzpicture}[scale=0.5]
\incircle;
\outcircle;
\draw[mid>] (120:1) -- (120:2);
\draw (120:2) to [out=120, in=50]  (140:2.8);
\draw (140:2.8) arc [radius=2.8, start angle=140, end angle=400];
\draw (40:2.8) to [out=130, in=240] (60:3.1);
\draw (60:3.1) -- (60:4);

\moveleft{1}
\moveleft{2}
\moveleft{3}
\moveleft{4}

\node [below] at (270:4.5) {\Large $ R$};
\end{tikzpicture}
\caption{Some generators of $AB_5$. }
\end{figure}

Given any set $S$, we can form the groupoid of annular braids labelled by $ S$, $\AB_m(S) $.  This is a groupoid whose set of objects is $ S^m $ with generators
$$ T_i : (s_1, \dots, s_m) \rightarrow (s_1, \dots, s_{i+1}, s_i, \dots, s_m), \quad X_i : (s_1,\dots, s_m) \rightarrow (s_1, \dots, s_m)$$
for all $ s_1, \dots, s_m \in S $.  These generators satisfy the same relations as in $AB_m $.  In other words, $ \AB_m(S) $ is the fundamental groupoid of the space $(A \times S)^m \setminus \Delta / S_m $ (where $ \Delta $ is the pullback of the usual fat diagonal in $ A^m $).

We can also consider the usual braid group $B_m$ (the subgroup generated by the $ T_i$'s) and the groupoid $\B_m(S) $.

\subsubsection{Action on $\OqGG$}

Since $ \RqG $ is a braided monoidal category, then we have a functor $ \B_m(\RqG) \rightarrow \RqG $ where the $ T_i$'s act by the usual braiding of adjacent tensor factors.  This can be extended to the annular braid groupoid as follows.

Now, suppose that we have representations $V_1, \dots, V_m $ of $ \Uq $.  Then we can form $ \Oq \otimes V_1 \otimes \cdots \otimes V_m  $.  We can braid the representations $ V_1, \dots, V_m $ using the braiding in $ \RqG $ which gives us isomorphisms $ T_1, \dots, T_{m-1} $ satisfying the braid relations.

We define the automorphism $X_j $ by placing the automorphism $ X $ (from section \ref{sec:coact}) in the $j$th position using the monoidal structure defined in section \ref{sec:monstructure}. In other words, $ X_j $ is given by the composition
\begin{gather*}
\Oq \otimes V_1 \otimes \cdots \otimes V_m  \xrightarrow{I \otimes \beta_{V_1 \otimes \cdots \otimes V_{j-1}, V_j} \otimes I} \Oq \otimes V_j \otimes  V_1 \cdots V_m \\
\xrightarrow{X \otimes I} \Oq \otimes V_j \otimes  V_1 \cdots V_m \xrightarrow{I \otimes \beta_{V_j, V_1 \otimes \cdots \otimes V_{j-1}} \otimes I} \Oq \otimes V_1 \otimes \cdots \otimes V_m
\end{gather*}
Because they come from the monoidal structure, $X_1, \dots, X_m $ commute. The following result is due to Lyubashenko-Majid \cite{LM}.  The proof is straightforward.
\begin{Proposition} \label{pr:ABaction}
The above $ T_i, X_j $ satisfy the annular braid group relations and thus define a functor $ AB(\RqG) \rightarrow \OqGG $.
\end{Proposition}

An object of $ \OqGG $ of the form $  \Oq \otimes V_1 \otimes \cdots \otimes V_m  $ will be called an object of length $ m $.  As a special case of Proposition \ref{pr:ABaction}, we see that $ \Z^m $ acts on objects of length $ m$ in $ \OqGG $.

\subsection{Kostant slice functor} \label{sec:Kostant}
In \cite{BF}, an important role is played by the Kostant slice functor.  In their setting of $ \OgG$, they considered the Kostant slice $ \mathfrak{t}/W \hookrightarrow \fg $.  Pulling back along the Kostant slice gives rise to a fully faithful  functor
$ S : \OgG \rightarrow Z_{\fg}\mod $ where  $ Z_\fg $ is the group scheme of regular centralizers regarded as a group scheme over $ \mathfrak{t}/ W$.

In our situation, we have a Steinberg slice $ T/W \hookrightarrow G $.  As in the Lie algebra case, the image of this Steinberg slice meets every regular $ G$-orbit once.  Moreover, the set $G^{reg} $ of regular elements of $ G$ is dense and its complement has codimension 2.  Thus we see that the resulting functor $ S : \OGG \rightarrow Z_G\mod $ given by $ M \mapsto M \otimes_{\cO(G)} \cO(T/W) $ is an equivalence, where $ Z_G $ is the group scheme of group regular centralizers regarded as a group scheme over $ T/W $.  This functor (as in the work of \cite{BF}) can and will be used to give a more algebraic description of the category $ \OGG $.

Unfortunately, in the quantum group situation we do not seem to have an analogous map $ \Oq \rightarrow E(q) $ (nor an analog of the group scheme of regular centralizers).  However, it should be possible to construct a version of the Kostant slice functor by relating $ \OqGG $ to a category of Harish-Chandra modules and using the work of Sevostyanov \cite{Sev}.

\subsection{Minuscule versions}
A non-trivial irreducible representation $ V $ of $ G $ is called minuscule, if all weights of $ V $ lie in single Weyl orbit.  Equivalently $ V(\lambda) $ is minuscule if and only if $ \lambda $ is a non-zero minimal element of $ \Lambda_+ $ (with respect to the usual partial order on $ \Lambda_+$).

We consider the full subcategory $ \RqGm$ of $ \RqG $ consisting of tensor products of minuscule representations.  Similarly, we consider $ \OqGGm $ to be the subcategory of $ \OqGG $ consisting of objects of the form $ \Oq \otimes V $ for $ V \in \RqGm$.

When $ G = SL_n $, the minuscule representations coincide with the fundamental representations $ \bigwedge^k \C^n $ for $ k = 1, \dots, n-1 $.  In this case, the idempotent completion of $ \RqSm $ is $ \RqS $.   From Propositions \ref{pr:OCmod} and \ref{pr:Idemp}, we deduce the following.

\begin{Corollary} \label{cor:OCmod}
There is an equivalence $ \RqSm(S^1) \cong \OqSSm $.
\end{Corollary}

\section{Geometry of affine Grassmannians}

\subsection{The affine Grassmannian}
Let $ \K = \C((t)), \cO = \C[[t]] $ and let $ Gr = \Gv(\K)/\Gv(\cO) $ denote the affine Grassmannian of the Langlands dual group.  See Zhu \cite{Z} for a thorough discussion about affine Grassmannians.

For every weight $ \lambda \in \Lambda $, there is a point $ t^\lambda \in Gr $.  For each $ \lambda \in \Lambda_+ $, let $ Gr^\lambda = \Gv(\cO)t^\lambda $.  We have $ \dim Gr^\lambda = 2\rho^\vee(\lambda) $.  We have the stratification $ Gr = \cup_{\lambda \in \Lambda_+} Gr^\lambda $.  Moreover, given any two points $ L_1, L_2 \in Gr$, there exists $ g \in \Gv(\K) $ such that $ (g L_1, gL_2) = (t^0, t^\lambda) $ for a unique dominant weight $ \lambda \in \Lambda_+ $.  In this case, we say that $ L_2 $ is distance $ \lambda $ from $ L_1$ and we write $ d(L_1, L_2) = \lambda $.  With this definition, $ Gr^\lambda $ is the set of points of distance $ \lambda $ from $ t^0$.

The following well known result explains the geometric significance of minuscule weights.
\begin{Lemma}
The following are equivalent.
\begin{enumerate}
\item $Gr^\lambda $ is projective.
\item $ \overline{Gr^\lambda} $ is smooth.
\item $ \lambda $ is minuscule.
\end{enumerate}
\end{Lemma}

Given a sequence $ \ulam = (\lambda_1, \dots, \lambda_m) $ of minuscule dominant weights, we can consider the variety
$$ Gr^{\ulam} := \{ (t^0 = L_0, L_1, \dots, L_m) \in Gr^m : d(L_{i-1}, L_i) = \lambda_i \} $$
which comes with a morphism $ p_{\ulam} : Gr^{\ulam} \rightarrow Gr $ given by $ (L_1, \dots, L_m) \mapsto L_m$.  Note that $ Gr^{\ulam} $ is an iterated bundle of $ Gr^{\lambda_i}$; there is a obvious map $ Gr^{\ulam} \rightarrow Gr^{(\lambda_1, \dots, \lambda_{m-1})} $ which is a $Gr^{\lambda_m} $-bundle.

There is an action of $ \Gv$ on $ Gr $ by left multiplication.  Also there is an action of $ \C^\times $ on $ Gr $ by ``loop rotation''; this comes from the action of $ \Cx $ on $ \cO $ given by $ s \cdot p(t) = p(st) $.

\subsection{Line bundles on the affine Grassmannian} \label{sec:linebundles}

Recall that the connected components of $ Gr $ are labelled by the finite set $ \Lambda / \Z \Delta $, where $ \Delta $ is the set of roots of $ G $.   Each connected component has Picard group $ \Z $ and we define a line bundle $ \sL $ on $ Gr $ to be the positive generator of the Picard group on each connected component (see for example section 2.4 of \cite{Z}).

Unfortunately, the line bundle $ \sL $ is not always $ \Gv$-equivariant.  It is however $ \Gc $-equivariant (since each connected component is isomorphic to the quotient of $ \widehat{\Gc} $ (the affine Kac-Moody group) by a hyperspecial parahoric subgroup).  In fact $ \sL $ carries a $ \Gc \times \C^\times $-equivariant structure where the action of $ \widetilde{T^\vee}$ on the fibre over $ t^\mu $ is given by the weight $ \iota(\mu) $ and the action of $\C^\times $ on the same fibre over $ t^\mu $ is given by the pairing $ (\mu, \mu)  $.

\subsection{The equivariant Satake category}
Consider the Satake category $P(Gr)$ of perverse sheaves on $ Gr $ constructible with respect to the stratification by $ \Gv(\cO) $ orbits.

We will need an equivariant version of this category. Note that we have an action of $ \Gv $ on $ Gr $ which preserves the $\Gv(\cO) $-orbits.  Using the $W$-invariant bilinear form, we may identify $H_\Gv(\pt) = \cO(\ft^\vee)^W \cong \cO(\ft)^W $.

Let $ D_{\Gv}(Gr) $ denote the equivariant derived category of $Gr$ for this action.   We consider $D_{\Gv}(Gr) $ with morphisms given by total $\mathrm{Ext} $.  Thus the category $ D_{\Gv}(Gr) $ is enriched over graded $ \cO(\ft)^W$-modules.

We have a global section functor
$$ D_{\Gv}(Gr) \rightarrow H^*_\Gv(Gr)\mod.$$
By a theorem of Ginzburg (see Lemma 13 from \cite{BF}), this functor is fully faithful.

Since $ Gr^\ulam $ is smooth and $ p_\ulam $ is semismall, we can consider the perverse sheaf $ {p_\ulam}_* \C_{Gr^\ulam} [2\rho^\vee(\sum_k \lambda_k)] \in P(Gr) $.  It is the tensor product of the objects $ \C_{Gr^{\lambda_k}}[2\rho^\vee(\lambda_k)] $ with respect to the usual monoidal structure on $ P(Gr) $.  We let $ D_{\Gv}^{min}(Gr)$ denote the full subcategory of $ D_{\Gv}(Gr)$ consisting of these objects.

\subsection{Homology convolution categories}\label{sec:homconv}

Recall convolution in homology following Chriss-Ginzburg \cite{CG}. Let $ I $ be an index set and let $ \{Y_i\}_{i \in I} $ be a collection of smooth varieties along with proper maps $ p_i : Y_i \rightarrow Y $ to some fixed variety $ Y $.  Form the fibre products $ Z_{ij} = Y_i \times_Y Y_j $.

We can use this data to define a category $ hConv(Y) $ (this category depends on more than just $ Y $, but we suppress the rest of the data in the notation). The objects in $hConv(Y)$ consist of the elements of the set $ I $. The morphisms are defined as
$$ \Hom_{hConv(Y)}(i,j) = H_*(Z_{ij})$$
where $ H_*(Z_{ij}) $ denotes the total Borel-Moore homology of $ Z_{ij} $ (with coefficients in $ \C $). The composition operation in $hConv(Y)$ is defined by the convolution product
$$ H_*(Z_{ij}) \otimes H_*(Z_{jk}) \rightarrow H_*(Z_{ik})$$
which is given by the formula
$$c_1 * c_2 = (\pi_{13})_* (\pi_{12}^*(c_1) \cap \pi_{23}^*(c_2)),$$
where ``$\cap$'' denotes the intersection product (with support), relative to the ambient manifold $Y_i \times Y_j \times Y_k$ and $ \pi_{12} : Y_i \times Y_j \times Y_k \rightarrow Y_i \times Y_j$, etc.  For more details about this construction see \cite[Sec. 2.7]{CG}.

The following result of Chriss-Ginzburg shows the importance of these categories. Consider the derived category of constructible sheaves on $Y$, $D_c(Y)$, which we regard as a category enriched over graded vector spaces by taking total $Ext$.

\begin{Theorem}\label{th:hConv}\cite[Theorem 8.6.7]{CG}
The category $ hConv(Y) $ is equivalent to the full subcategory of $ D_c(Y) $ whose objects are $ {p_i}_* \C_{Y_i} $.
\end{Theorem}

If the varietes $ Y_i, Y $ all carry compatible actions of some group $H$, we may consider the corresponding equivariant category $hConv^H(Y)$ where the morphisms are defined by equivariant homology.

We now apply this framework to obtain the category $ hConv(Gr) $ where the base variety is $ Gr $ and the mapping varieties are given by $ p_\ulam : Gr^\ulam \rightarrow Gr $, where $\ulam $ ranges over all sequence of minuscule dominant weights. The category $hConv(Gr)$ was first considered by the second author with Fontaine and Kuperberg in \cite{FKK}. The fibre products $ Z(\ulam, \umu) $ appearing in $hConv(Gr)$ are the ``Steinberg-like'' varieties
$$ Z(\ulam, \umu) = \{ (L_1, \dots, L_m), (L'_1, \dots, L'_{m'}) \in Gr^\ulam \times Gr^\umu : L_m = L'_{m'} \}. $$
Using the action of $ \Gv $ we may also define $ hConv^{\Gv}(Gr) $.  Applying Theorem \ref{th:hConv} we deduce the following.
\begin{Theorem} \label{th:BFConv}
We have an equivalence
$$ D^{min}_{\Gv}(Gr) \cong hConv^{\Gv}(Gr).$$
\end{Theorem}

In other words, we can model $ D_{\Gv}^{min}(Gr) $ using this homology convolution category.  This motivates us to introduce a K-theory convolution category.

\subsection{The category $KConv(Gr)$}\label{sec:kconv}

The convolution formalism in homology from section \ref{sec:homconv} can be repeated in K-theory. More precisely, given a collection $ Y_i \rightarrow Y $ as above, the objects in $KConv(Y)$ are still indexed by the set $I$ but the morphisms are now defined as
$$Hom_{KConv(Y)}(i,j) = K(Z_{ij})$$
where $K(Z_{ij}) $ denotes the complexified Grothendieck group of coherent sheaves on $Z_{ij}$.  The composition operation in $ KConv(Y) $ is defined by a similar formula, namely
\begin{align*}
K(Z_{ij}) \otimes K(Z_{jk}) &\rightarrow K(Z_{ik}) \\
[\sF_1] \otimes [\sF_2] &\mapsto [\pi_{13*} (\pi_{12}^*(\sF_1) \otimes \pi_{23}^*(\sF_2))]
\end{align*}
For more details about this construction see \cite[Sec. 5.2]{CG}. As before, given a compatible action of a group $H$ we can work with equivariant K-theory $K^H(\cdot)$ to define the category $ KConv^H(Y)$ .

Applying this construction to $Gr^{\ulam} \rightarrow Gr$ gives us the categories $ KConv^{\Gc}(Gr) $ and $ KConv^{\Gc \times \Cx}(Gr) $ where $\Gc$ acts as before and $\Cx$ acts by loop rotation.  (Here we are using $ \Gc $ rather than $ \Gv$, since our line bundles are only $ \Gc$-equivariant.)  In these categories, the objects are sequences $ (\lambda_1, \dots, \lambda_m) $ (we will say that this object has length $ m $).

On the variety $Gr^\ulam$ we have $ m$ line bundles $ \sL_1, \dots, \sL_m $, where $\sL_k $ is defined as $ p_k^* \sL \otimes p_{k-1}^* \sL^\vee $, where $ p_k : Gr^\ulam \rightarrow Gr $ is given by $ p_k(L_1, \dots, L_m) = L_k $. For $i=1, \dots, m$ we define
$$X_i \in \End_{KConv^{\Gc\times \Cx}}(\ulam) = K^{\Gc \times \Cx}(Z(\ulam, \ulam))$$
as $ X_i := \Delta_* [\sL_i] $ where $ \Delta : Gr^\ulam \rightarrow Z(\ulam, \ulam) $ is the diagonal inclusion. These $ X_i $ are invertible and so they define a map $ \Z^m \rightarrow \End_{KConv^{\Gc \times \Cx}}(\ulam)$.  Thus, we see that $ \Z^m $ acts on objects of length $ m $ in $ KConv^{\Gc}(Gr) $ and $ KConv^{\Gc \times \Cx}(Gr) $.

\subsection{Main conjecture}
 Motivated by Theorem \ref{th:BFConv} and the discussion in the introduction, we now formulate the following main conjecture.

\begin{Conjecture} \label{co:main}
We have equivalences
\begin{align}
\OGGm \otimes_E E^\vee &\cong KConv^{\Gc}(Gr)  \\
\OqGGm \otimes_E E^\vee &\cong KConv^{\Gc \times \Cx}(Gr) \otimes_{\C[q^\pm]} \C(q)
\end{align}
of $E^\vee\mod$ and $E^\vee(q)\mod$ enriched categories respectively.

Moreover, these equivalences should be compatible with the actions of $ \Z^m $ on objects of length $ m $ on each side of the equivalence.
\end{Conjecture}
By an equivalence of $E^\vee\mod$ enriched categories, we mean that the map on Hom spaces should be an $ E^\vee $-module morphism map.  On the left hand side the $ E^\vee$-module structure comes from tensoring (and the $ E $-module structure was explained in section \ref{sec:Kostant}), while on the right hand side it comes from $ K^\Gc(pt) = E^\vee $.

The compatibility between the actions of $ \Z^m $ should be viewed as a partial substitute for the commutative diagram appearing in \cite{BF} which involves the Kostant slice functor.  The action of $ \Z^m $ on the left hand side is defined in section \ref{sec:AnnularBraids} and the action of $ \Z^m $ on the right hand side is defined in section \ref{sec:kconv}.

\begin{Remark}
The base change from $ E $ to $ E^\vee $ on the left hand side seems to be necessary to get the endomorphisms of the trivial objects to match.  Note that in the simply-laced case, we have an isomorphism $ E \cong E^\vee $, so the base change is unneccesary.
\end{Remark}

\section{Spiders and annular spiders}\label{sec:spider-Oq}

\subsection{Some $ SL_n $ notation}

The remainder of the paper will be devoted to the proof of Conjecture \ref{co:main} for the case $ G = SL_n $.

The weight lattice of $ SL_n$ is $ \Lambda = \Z^n / \Z (1, \dots, 1) $ and $ \Lambda_+ $ is the subset consisting of those $ \lambda = (a_1, \dots, a_n) $ such that $ a_1 \ge \dots \ge a_n $.  The minuscule dominant weights are $ \omega_k = (1, \dots, 1, 0, \dots, 0) $, for $ k = 1, \dots, n-1 $.  The corresponding representation of $U_q \sl_n$ is $ V(\omega_k) = \Altq{k} \C_q^n $.

A sequence of minuscule dominant weights $ \ulam = (\l_1, \dots, \l_m) $ will be encoded by a sequence $ \uk = (k_1, \dots, k_m)$, where $ \lambda_j = \omega_{k_j}$.  Note that $ \Gv = PGL_n $ and $ \Gc = SL_n $.  We write
$$ E = ( \C[t_1^{\pm}, \dots, t_n^{\pm}]/(t_1 \dots t_n - 1))^{S_n} = \C[e_1, \dots, e_{n-1}], $$
where $e_k$ is the usual $k$th elementary symmetric function.

\subsection{The spider} \label{se:spider}
Our main tool for proving the conjecture in the $ SL_n$ case will be a combinatorially defined category called the annular spider.

We begin by recalling the definition of the spider category $\Sp[q^\pm]$ from \cite{CKM}. $\Sp[q^\pm]$ has as objects sequences $\uk$ in $\{1^\pm,\ldots,(n-1)^\pm\}$, and as morphisms $\C[q^\pm]$-linear combinations of oriented planar graphs by the following four types of vertices:
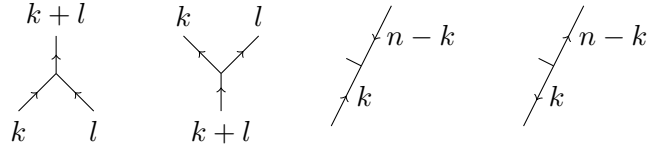
\begin{align*}
\fuse{k}{l}{k+l}
\qquad
\fork{k}{l}{k+l}
\qquad
\tikz[baseline=0.7cm]{
\foreach \n in {0,1,2} {
	\coordinate (a\n) at (0.4*\n, 0.8*\n);
}
\draw[mid>] (a0) -- node[right] {$k$} (a1);
\draw[mid<] (a1) -- node[right] {$n-k$} (a2);
\draw (a1) -- +(-0.2,0.1);
}
\qquad
\tikz[baseline=0.7cm]{
\foreach \n in {0,1,2} {
	\coordinate (a\n) at (0.4*\n, 0.8*\n);
}
\draw[mid<] (a0) -- node[right] {$k$} (a1);
\draw[mid>] (a1) -- node[right] {$n-k$} (a2);
\draw (a1) -- +(-0.2,0.1);
}
\end{align*}
with all labels drawn from the set $\{1,\ldots,n-1\}$. The third and fourth graphs depict bivalent vertices, called `tags', which are not rotationally symmetric, meaning that the tag provides a distinguished side. The bottom boundary of any planar graph in $\Hom(\uk, \ul)$ is $\uk$ with the strand oriented up for each positive entry, and the strand oriented down for each negative entry. Similarly, the top boundary is determined by $\ul$ in the same way.

On these diagrams, we impose the following relations, together with the mirror reflections and the arrow reversals of these.

\begin{align}
\tikz[baseline=0.4cm]{
\foreach \n in {0,1,2} {
	\coordinate (a\n) at (0.4*\n, 0.8*\n);
}
\draw[mid>] (a0) -- node[right] {$k$} (a1);
\draw[mid<] (a1) -- node[right] {$n-k$} (a2);
\draw (a1) -- +(-0.2,0.1);
}
& = (-1)^{k(n-k)}
\tikz[baseline=0.4cm]{
\foreach \n in {0,1,2} {
	\coordinate (a\n) at (0.4*\n, 0.8*\n);
}
\draw[mid>] (a0) -- node[right] {$k$} (a1);
\draw[mid<] (a1) -- node[right] {$n-k$} (a2);
\draw (a1) -- +(0.2,-0.1);
}
\displaybreak[1]
\label{eq:switch}
\\
\begin{tikzpicture}[baseline=20]
\foreach \n in {0,...,3} {
	\coordinate (z\n) at (0.4*\n, 0.8*\n);
}
\draw[mid>] (z0) -- node[right] {\small $k+l$} (z1);
\draw[mid>] (z2) -- node[right] {\small $k+l$} (z3);
\draw[mid>] (z1) to[out=150,in=-190] node[left] {\small $k$} (z2);
\draw[mid>] (z1) to[out=-30,in=0] node[right] {\small $l$} (z2);
\end{tikzpicture}
& = \qbin{k+l}{l}
\tikz[baseline=20]{\draw[mid>] (0,0) -- node[right] {\small $k+l$} (1,2);}
\label{eq:bigon1}
\displaybreak[1] \\
\begin{tikzpicture}[baseline=20]
\foreach \n in {0,...,3} {
	\coordinate (z\n) at (0.4*\n, 0.8*\n);
}
\draw[mid>] (z0) -- node[right] {\small $k$} (z1);
\draw[mid>] (z2) -- node[right] {\small $k$} (z3);
\draw[mid<] (z1) to[out=150,in=-190] node[left] {\small $l$} (z2);
\draw[mid>] (z1) to[out=-30,in=0] node[right] {\small $k+l$} (z2);
\end{tikzpicture}
& = \qbin{n-k}{l}
\tikz[baseline=20]{\draw[mid>] (0,0) -- node[right] {\small$k$} (1,2);}
\label{eq:bigon2}
\displaybreak[1] \\
\begin{tikzpicture}[baseline]
\foreach \x/\y in {0/0,1/0,2/0,0/1,1/1,0/2} {
	\coordinate(z\x\y) at (\x+\y/2,\y/1.5);
}
\coordinate (z03) at (1,2);
\draw[mid>] (z00) node[below] {\small $k$} --  (z01);
\draw[mid>] (z01) -- node[left] {\small $k+l$} (z02);
\draw[mid>] (z10) node[below] {\small $l$} -- (z01);
\draw[mid>] (z20) node[below] {\small $m$} -- (z02);
\draw[mid>](z02) -- node[left] {\small $k+l+m$} (z03);
\end{tikzpicture}
& =
\begin{tikzpicture}[baseline]
\foreach \x/\y in {0/0,1/0,2/0,0/1,1/1,0/2} {
	\coordinate(z\x\y) at (\x+\y/2,\y/1.5);
}
\coordinate (z03) at (1,2);
\draw[mid>] (z00) node[below] {\small $k$} --  (z02);
\draw[mid>] (z10) node[below] {\small $l$} -- (z11);
\draw[mid>] (z20) node[below] {\small $m$} -- (z11);
\draw[mid>] (z11) -- node[right] {\small $l+m$} (z02);
\draw[mid>](z02) -- node[left] {\small $k+l+m$} (z03);
\end{tikzpicture}
\label{eq:IH}
\displaybreak[1] \\
\label{eq:tag-migration}
\begin{tikzpicture}[baseline=20]
\coordinate (z1) at (0,0);
\coordinate (z2) at (1,0);
\coordinate (c) at (0.5,0.5);
\coordinate (ce) at (0.5,1);
\coordinate (e) at (0.5,1.45);
\draw[mid>] (z1) node[below] {\small $k$} -- (c);
\draw[mid>] (z2) node[below] {\small $l$} -- (c);
\draw[mid>] (c) -- node[right] {\small $k + l$} (ce);
\draw[mid<] (ce) -- (e) node[above] {\small $n-k-l$};
\draw (ce) -- +(0.2,0);
\end{tikzpicture}
&=
\begin{tikzpicture}[baseline=20]
\coordinate (z1) at (0,0);
\coordinate (z2) at (1.5,0);
\coordinate (cze) at (1.25,0.333);
\coordinate (c) at (0.75,1);
\coordinate (e) at (0.75,1.45);
\draw[mid>] (z1) node[below] {\small $k$} -- (c);
\draw[mid>] (z2) node[below] {\small $l$} -- (cze);
\draw[mid<] (cze) -- node[right] {\small $n-l$} (c);
\draw (cze) -- + (0.15,0.15);
\draw[mid<] (c) -- (e) node[above] {\small $n-k-l$};
\end{tikzpicture} \\
\label{eq:tag-migration2}
\begin{tikzpicture}[baseline=20]
\coordinate (z1) at (0,0);
\coordinate (z2) at (1,0);
\coordinate (c) at (0.5,0.5);
\coordinate (ce) at (0.5,1);
\coordinate (e) at (0.5,1.45);
\draw[mid>] (z1) node[below] {\small $k+l$} -- (c);
\draw[mid<] (z2) node[below] {\small $k$} -- (c);
\draw[mid>] (c) -- node[right] {\small $l$} (ce);
\draw[mid<] (ce) -- (e) node[above] {\small $n-l$};
\draw (ce) -- +(-0.2,0);
\end{tikzpicture}
&=
\begin{tikzpicture}[baseline=20,x=-1cm]
\coordinate (z1) at (0,0);
\coordinate (z2) at (1.5,0);
\coordinate (cze) at (1.25,0.333);
\coordinate (c) at (0.75,1);
\coordinate (e) at (0.75,1.45);
\draw[mid<] (z1) node[below] {\small $k$} -- (c);
\draw[mid>] (z2) node[below] {\small $k+l$} -- (cze);
\draw[mid<] (cze) -- node[left] {\small $n-k-l$} (c);
\draw (cze) -- + (0.15,0.15);
\draw[mid<] (c) -- (e) node[above] {\small $n-l$};
\end{tikzpicture} \\
\label{eq:id1b}
\tikz[baseline=40]{
\laddercoordinates{1}{2}
\ladderEn{0}{0}{$k-s$}{$l+s$}{$s$}
\ladderEn{0}{1}{$k-s-r$}{$l+s+r$}{$r$}
\node[left] at (l00) {$k$};
\node[right] at (l10) {$l$};
}
&=
\qbin{r+s}{r}
\tikz[baseline=20]{
\laddercoordinates{1}{1}
\ladderEn{0}{0}{$k-s-r$}{$l+s+r$}{$r+s$}
\node[left] at (l00) {$k$};
\node[right] at (l10) {$l$};
}
\displaybreak[1]
\\
\label{eq:commutation}
\begin{tikzpicture}[baseline=40]
\laddercoordinates{1}{2}
\node[left] at (l00) {$k$};
\node[right] at (l10) {$l$};
\ladderEn{0}{0}{$k{-}s$}{$l{+}s$}{$s$}
\ladderFn{0}{1}{$k{-}s{+}r$}{$l{+}s{-}r$}{$r$}
\end{tikzpicture}
&= \sum_t \qbin{k-l+r-s}{t}
\begin{tikzpicture}[baseline=40]
\laddercoordinates{1}{2}
\node[left] at (l00) {$k$};
\node[right] at (l10) {$l$};
\ladderFn{0}{0}{$k{+}r{-}t$}{$l{-}r{+}t$}{$r{-}t$}
\ladderEn{0}{1}{$k{-}s{+}r$}{$l{+}s{-}r$}{$s{-}t$}
\end{tikzpicture}
\end{align}
\begin{Remark}
In the relations above we allow strands to be labelled by $0$ and $n$. This means that $0$-strands should be deleted and $n$-strands replaced by tags.  On the other hand, any diagram containing a strand whose label is less than $0$ or greater than $ n $ should be interpreted as the 0 morphism.
\end{Remark}

The category $\Sp[q^\pm] $ is a monoidal category with tensor product given on objects by concatenation of sequences and on morphisms by concatenation in the vertical direction.  It also has a natural pivotal structure by the usual graphical calculus. We will denote by $\Sp(q)$ the same category but over $\C(q)$ rather than $\C[q^\pm]$. We also denote by $\Sp$ the category specialized to $q=1$.

It will be useful for us to add crossings to the spider.  As in \cite{CKM}, we define a crossing of a strand labelled $k$ with a strand
labelled $ l $ to be a sum of webs.
\begin{equation}\label{eq:braiding1}
\begin{tikzpicture}[baseline=40]
\draw [shift={+(0.8,0)}][->](-9,0.0) -- (-7,2.9);
\draw [shift={+(0.8,0)}][->](-8.2,1.8) -- (-9,2.9);
\draw [shift={+(0.8,0)}][-](-7,0.0) -- (-7.8,1.2);
\draw [shift={+(-8.7,0)}](0.3,0) node {\small $k$};
\draw [shift={+(-8.3,0)}](2.3,0) node {\small $l$};
\node[left] at (-1,1.2) {$= (-1)^{kl} (-q)^{k} \sum\limits_{\substack{a,b\geq 0 \\ b-a = k-l}} (-q)^{-b}$};
\laddercoordinates{1}{2};
\node[left] at (l00) {\small $k$};
\node[right] at (l10) {\small $l$};
\ladderEn{0}{0}{\small $k-b$}{\small $l+b$}{\small $b$}
\ladderFn{0}{1}{\small $l$}{\small $k$}{\small $a$}
\end{tikzpicture}
\end{equation}
This defines a functor $\B_m(\{1, \dots, n\}) \rightarrow \Sp$ which we can regard as giving a braiding on the category $\Sp$.

\subsection{The annular spider}
We now define the annular spider $\ASp[q^\pm]$. The objects are the same as in $\Sp[q^\pm]$.  For the morphisms, we consider webs with the same kinds of vertices, except that the webs are now embedded in an annulus $ A = S^1 \times [0,1] $.  We consider the same local generating relations as in $ \Sp[q^\pm]$.  The inner boundary of the annulus is regarded as the source of the morphism and the outer boundary is the target.  Let $ * \in S^1 $ denote the bottom of the the circle.  We do not allow strands to start or end at $ *$.  Given a web $ w $, its source (resp. target) is found by reading the inner (resp. outer) circle clockwise starting at the $ * $.

\begin{figure}
\centering
\begin{tikzpicture}[scale=0.5]
\incircle;
\outcircle;
\draw[mid>] (135:1) -- (135:2.2);
\node [above] at (135:1.5) {\small $5$};
\draw[mid>] (135:2.2)  -- (120:4);
\node [above right] at (128:3) {\small $3$};
\draw[mid>] (135:2.2) -- (150:4);
\node [below left] at (142:3) {\small $2$};
\draw[mid<] (45:1)  -- (45:4);
\node [above] at (45:2.5) {\small $4$};

\draw[mid>] (200:1) -- (200:2);
\node[above] at (200:1.5) {\small $2$};
\draw[mid>] (200:2) to  [out=280,in=110] (280:4);
\node at (240:2.6) {\small $1$};
\node [left] at (190:3) {\small $1$};
\draw[mid>] (200:2) to (180:4);

\draw[mid>] (250:1) to [out=250,in=120] (300:4);
\node[right] at (295:2.9) {\small $3$};

\node at (270:1) {\Large $*$};
\node at (270:4) {\Large $*$};
\end{tikzpicture}
\caption{An example of an annular web with source $(3,2,5,-4)$ and target $ (1,2,3,-4,3,1)$.}
\end{figure}
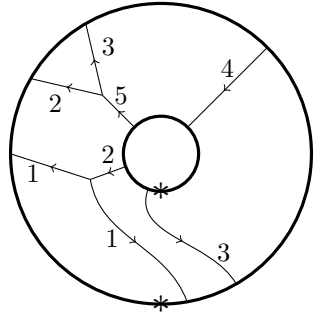

As before, we denote by $\ASp(q)$ the same category but with the space of morphisms defined over $\C(q)$ rather than $\C[q^\pm]$. We can also consider the $q = 1$ version of $\ASp[q^\pm]$, which we denote $\ASp$.

\subsection{The functor $\Gamma$}

We begin by defining a functor
$$ \Gamma : \Sp(q) \rightarrow \RqSm.$$
following \cite{CKM} (this functor was called $\Gamma_n$ in \cite{CKM}). At the level of objects we take
$$(k_1^{\varepsilon_1}, \dots, k_m^{\varepsilon_m}) \mapsto (\Altq{k_1} \C_q^n)^{\varepsilon_1} \otimes \dots \otimes (\Altq{k_m} \C_q^n)^{\varepsilon_m}$$
(where $ \varepsilon_i \in \{+, - \} $).

The morphisms are mapped as follows
\begin{equation}\label{eq:M}
\fuse{k}{l}{k+l} \mapsto M_{k,l} \ \ \text{ and } \ \ \fork{k}{l}{k+l} \mapsto M'_{k,l}
\end{equation}
where $M_{k,l} : \Altq{k} (\C^n_q) \otimes \Altq{l}(\C^n_q) \rightarrow \Altq{k+l}(\C^n_q)$ is defined by
\begin{equation*}
M_{k,l}(x_S \otimes x_T) = x_S \wedge_q x_T = \begin{cases} (-q)^{\ell(S, T)} x_{S \cup T} & \text{ if } S \cap T = \emptyset \\
 0 & \text{ otherwise }
 \end{cases}
\end{equation*}
while $M'_{k,l} :\Altq{k+l}(\C^n_q) \rightarrow \Altq{k} (\C^n_q) \otimes \Altq{l}(\C^n_q)$ is defined by
\begin{align*}
M'_{k,l}(x_S) = (-1)^{kl} \sum_{T \subset S} (-q)^{-\ell(S \smallsetminus T, T)} x_T \otimes x_{S \smallsetminus T}.
\end{align*}
Here, if $ S =\{k_1, \dots, k_a\} \subset \{1, \dots, n\} $, with $ k_1 > \dots > k_a $, we write $ x_S = x_{k_1} \wedge_q \cdots \wedge_q x_{k_a} \in \Altq{a}(\C_q^n)$ and likewise for $x_T$. Moreover, if $S, T$ are disjoint subsets of $ \{1, \dots, n\} $ we define
$$ \ell(S, T) = |\{ (i,j) : i \in S, j \in T \text{ and } i < j \}|.$$
The maps in (\ref{eq:M}) then force upon us the definition for tags. The main result from \cite{CKM} can be stated as follows.

\begin{Theorem} \label{th:CKM}
The functor $ \Gamma: \Sp(q) \rightarrow \RqSm$ is an equivalence of braided monoidal categories.
\end{Theorem}

\begin{Remark}
The argument in \cite{CKM} implies that the functor $\Gamma$ is also an equivalence when $q=1$.
\end{Remark}

\subsection{Horizontal trace and the annular spider category} \label{sec:TraceSpider}

We will now identify the annular spider category as the horizontal trace of the usual spider category. We define a functor $ \qSp(S^1) \rightarrow \AqSp $ as follows.  On objects it is the identity.

For morphisms, let $ \uk, \ul $ be two objects in $ \qSp $.  Suppose that we have a morphism $[\ur, \phi] $ in $ \qSp(S^1)$, where $\phi $ is a web with bottom boundary  $\uk \sqcup \ur $ and top boundary $\ur \sqcup \ul $.  From $ (\ur, \phi)$ we can get a morphism $ \overline{(r, \phi)}$ in $ \AqSp$ by sending $ \ur $ around the circle as follows
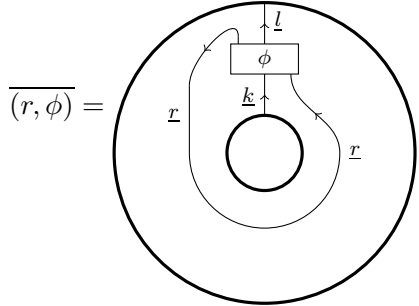
\begin{equation*}
\overline{(r,\phi)} = \begin{tikzpicture}[baseline=15,scale=0.5]
\incircle;
\outcircle;
\draw[mid>] (90:1) -- (90:2.1);
\node [left] at (90:1.5) {\scriptsize $\underline{k}$};

\draw[mid>] (2,0) to [out=90,in=270] (0.7,2.1);
\node [right] at (2,0) {\scriptsize $\underline{r}$};

\draw (0.9,2.1) rectangle (-0.9,2.9);
\node at (0, 2.5) {\scriptsize $\phi$};

\draw (2, 0) arc [radius=2, start angle = 360, end angle = 180];

\draw (-2,0) -- (-2,1.8);
\draw[mid<] (-2,1.8) to [out=90,in=90] (-0.7,3.1);
\draw (-0.7, 2.9) to (-0.7, 3.1);
\node [left] at (-2,1) {\scriptsize $\underline{r}$};

\draw[mid>] (90:2.9) -- (90:4);
\node [right] at (90:3.5) {\scriptsize $\underline{l}$};
\end{tikzpicture}
\end{equation*}

\begin{Proposition} \label{th:Annular}
For any two objects $ \uk, \ul $, this gives an isomorphism
$$ \Hom_{\qSp(S^1)}(\uk, \ul) \xrightarrow{\sim} \Hom_{\AqSp}(\uk, \ul). $$
This implies $ \AqSp \cong \qSp(S^1) $.
\end{Proposition}

\begin{proof}
First we check that the map is well-defined.   Suppose that we have a morphism $ \alpha : \ur \rightarrow \ur' $ and a morphism $ \psi : \uk \sqcup \ur' \rightarrow \ur \sqcup \ul $.  Then we need to know that $ (\ur, \psi \circ I \sqcup \alpha) $ and $ (\ur', \alpha \sqcup I \circ \psi) $ give rise to the same annular web.  This follows since the following two annular webs are isotopic:

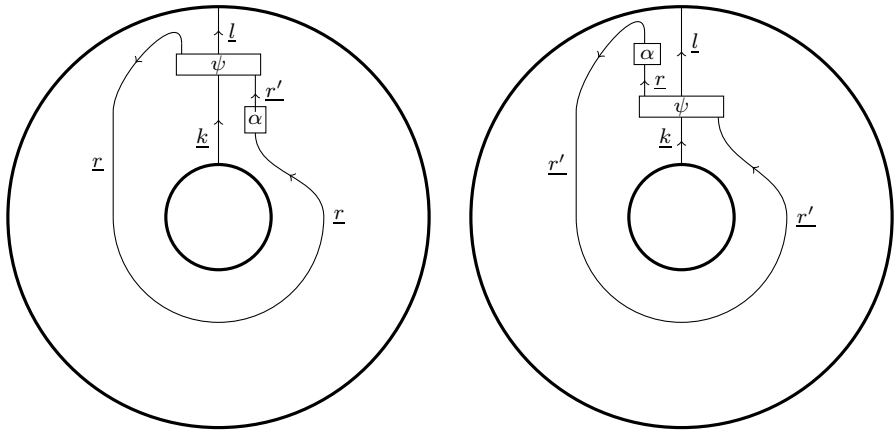
\begin{center}
\begin{tikzpicture}[scale=0.7]
\incircle;
\outcircle;
\draw[mid>] (90:1) -- (90:2.7);
\node [left] at (90:1.5) {\scriptsize $\underline{k}$};

\draw[mid>] (0.7,2) to (0.7,2.7);
\node [right] at (0.7,2.4) {\scriptsize $\underline{r'}$};

\draw (0.5,1.6) rectangle (0.9,2.1);
\node at (0.7,1.85) {\scriptsize $\alpha$};

\draw (0.8,2.7) rectangle (-0.8,3.1);
\node at (0, 2.9) {\scriptsize $\psi$};

\draw[mid>] (2,0) to [out=90,in=270] (0.7,1.6);

\draw (2, 0) arc [radius=2, start angle = 360, end angle = 180];
\node [right] at (2,0) {\scriptsize $\underline{r}$};

\draw (-2,0) -- (-2,2);
\draw[mid<] (-2,2) to [out=90,in=90] (-0.7,3.3);
\draw (-0.7, 3.1) to (-0.7, 3.3);
\node [left] at (-2,1) {\scriptsize $\underline{r}$};

\draw[mid>] (90:3.1) -- (90:4);
\node [right] at (90:3.5) {\scriptsize $\underline{l}$};
\end{tikzpicture}
\quad
\begin{tikzpicture}[scale=0.7]
\incircle;
\outcircle;
\draw[mid>] (90:1) -- (90:1.9);
\node [left] at (90:1.5) {\scriptsize $\underline{k}$};

\draw (0.8,1.9) rectangle (-0.8,2.3);
\node at (0, 2.1) {\scriptsize $\psi$};

\draw[mid>] (2,0) to [out=90,in=270] (0.7,1.9);

\draw (2, 0) arc [radius=2, start angle = 360, end angle = 180];
\node [right] at (2,0) {\scriptsize $\underline{r'}$};

\draw (-2,0) -- (-2,2);
\draw[mid<] (-2,2) to [out=90,in=90] (-0.7,3.5);
\node [left] at (-2,1) {\scriptsize $\underline{r'}$};

\draw (-0.7,3.3) to (-0.7,3.5);
\draw (-0.9,2.9) rectangle (-0.4,3.3);
\node at (-0.65,3.1) {\scriptsize $\alpha$};

\draw[mid>] (-0.7,2.3) to (-0.7,2.9);
\node [right] at (-0.7,2.55) {\scriptsize $\underline{r}$};

\draw[mid>] (90:2.3) -- (90:4);
\node [right] at (90:3.3) {\scriptsize $\underline{l}$};
\end{tikzpicture}
\end{center}

Let $ \uk, \ul $ be objects of $ \AqSp $.  Suppose that we have an annular web $\psi $ which gives a morphism between $ \uk $ and $ \ul $ in $ \AqSp$.  Draw a cut line from the inner boundary to the outer boundary (connecting the bottoms) and let $ \ur $ be the labels on the strands cutting this line.  Then unfold $ \psi $ to give a web $ \psi' $ with bottom boundary $ \uk \sqcup \ur $ and top boundary $ \ur \sqcup \ul $ (see Figure \ref{fig:unfold}).  It is easy to see that this map $ \psi \mapsto \psi' $ provides an inverse to the above construction and thus the map
$$ \Hom_{\qSp(S^1)}(\uk, \ul) \rightarrow \Hom_{\AqSp}(\uk, \ul) $$
is an isomorphism.
\end{proof}

\begin{figure}
\centering
\begin{tikzpicture}[scale=0.5]
\incircle;
\outcircle;
\draw[mid>] (135:1) -- (135:2.2);
\node [above] at (135:1.5) {\small $5$};
\draw[mid>] (135:2.2)  -- (120:4);
\node [above right] at (128:3) {\small $3$};
\draw[mid>] (135:2.2) -- (150:4);
\node [below left] at (142:3) { \small$2$};
\draw[mid<] (45:1)  -- (45:4);
\node [above] at (45:2.5) {\small $4$};

\draw[mid>] (200:1) -- (200:2);
\node[above] at (200:1.5) {\small $2$};
\draw[mid>] (200:2) to  [out=280,in=110] (280:4);
\node at (240:2.6) {\small $1$};
\node [left] at (190:3) {\small $1$};
\draw[mid>] (200:2) to (180:4);

\draw[mid>] (250:1) to [out=250,in=120] (300:4);
\node[right] at (295:2.9) {\small $3$};
\node at (270:1) {\Large $*$};
\node at (270:4) {\Large $*$};

\draw [dashed] (270:1) -- (270:4);
\end{tikzpicture} \quad
\begin{tikzpicture}[scale=0.8]
\draw[thick] (-2,0) to (6,0);
\draw[thick] (-2,4) to (6,4);
\draw[thick] (0,3.8) to (0,4.2);
\draw[thick] (4,-0.2) to (4,0.2);

\draw[mid>] (-1,0) to (-1,4);
\node[left] at (-1,1) {\small $3$};

\draw[mid>] (0,0) to (0,2);
\node [left] at (0,1) {\small $2$};
\draw[mid>] (0,2) to (-0.5, 4);
\node [left] at (-0.4,3.6) {\small $1$};
\draw[mid>] (0,2) to (0.5, 4);
\node [left] at (0.4,3.6) {\small $1$};

\draw[mid>] (2,0) to (2,2);
\node [left] at (2,1) {\small $5$};
\draw[mid>] (2,2) to (1.5, 4);
\node [left] at (1.6,3.6) {\small $2$};
\draw[mid>] (2,2) to (2.5, 4);
\node [left] at (2.4,3.6) {\small $3$};

\draw[mid<] (3.5,0) to (3.5,4);
\node[left] at (3.5,1) {\small $4$};

\draw[mid>] (4.5,0) to (4.5,4);
\node[left] at (4.5,1) {\small $3$};

\draw[mid>] (5.5,0) to (5.5, 4);
\node[left] at (5.5, 1) {\small $1$};
\end{tikzpicture}
\caption{An annular web $ \psi $ and its unfolding $ \psi'$.  In this example $ \uk = (3,2,5,4^-), \ul = (1,2,3,4^-,3,1) $ and $ \ur = (3,1) $.}\label{fig:unfold}
\end{figure}

Because $ \qSp $ is a braided monoidal category, $ \qSp(S^1) $ carries a monoidal  structure and thus $ \AqSp $  acquires a monoidal structure.  Diagrammatically, the tensor product of two morphisms $ \phi_1 $ and $ \phi_2 $ is given by placing the source and target of $ \phi_1 $ to the left of those of $ \phi_2 $ and having the strands of $\phi_1 $ cross over those of $ \phi_2 $.

\begin{Corollary}\label{cor:GamEquiv}
The functor $\Gamma$ induces an equivalence
$$A\Gamma: \ASp(q) \rightarrow \OqSSm.$$
\end{Corollary}
\begin{proof}
By Proposition \ref{th:Annular} and Theorem \ref{th:CKM}, we have
$$ \AqSp \cong \qSp(S^1) \cong \RqSm(S^1) $$
On the other hand, Corollary \ref{cor:OCmod} gives an equivalence between $ \RqSm(S^1) $ and $ \OqSSm$.
\end{proof}
\begin{Remark}
Just as Theorem \ref{th:CKM} holds at $q=1$, so does Corollary \ref{cor:GamEquiv}.
\end{Remark}

\subsection{Action of annular braids and loops}\label{sec:actionZ}
Suppose we have an annular braid with strands labelled from $ 1, \dots, n-1$. Since we have crossings in $ \AqSp $, we can interpret this annular braid as a morphism in $ \AqSp$. Thus we get a functor
$$\AB_m(\{\1, \dots, n-1 \}) \rightarrow \AqSp.$$

\begin{Proposition} \label{pr:ABfactor}
The functor $ \AB_m(\{1, \dots, n-1 \}) \rightarrow \OqSSm $ defined in section \ref{sec:AnnularBraids} factors as
$$ \AB_m(\{1, \dots, n-1\}) \rightarrow \AqSp \xrightarrow{A\Gamma} \OqSSm $$
up to a factor of $q$ as in Remark \ref{rem:factor}.
\end{Proposition}
\begin{proof}
Since annular braids are generated by the $ T_i $ and $ X_1 $ it suffices to check the result on these generators. For $T_i$ the result is proved in \cite[Cor. 6.2.3]{CKM}.

\begin{Remark}\label{rem:factor}
Unfortunately this is not entirely correct as the two images of $T_i$ actually differ by a factor of $q$. More precisely, the two $T_i$ acting on $\wedge^k(\C^n) \otimes \wedge^l(\C^n)$ differ by a factor of $q^{kl/n}$. This is because of the usual conventions for defining the braiding in $\RqG$. We chose to overlook this small discrepancy rather than adding this factor to the definition in (\ref{eq:braiding1}) which would also mean having to work over $\C(q^{1/n})$ throughout.
\end{Remark}

For $ X_1$, we consider some object $ \uk = (k_1, \dots, k_m) $ and then we use the construction from the proof of Proposition \ref{th:Annular} to unfold $ X_1 $ and produce a morphism $ \uk \sqcup \{ k_1 \} \rightarrow \{k_1 \} \sqcup \uk $.
\begin{figure}
\centering
\begin{tikzpicture}[scale=0.7]
\incircle;
\outcircle;
\draw[mid>] (60:1) -- (60:2.2);
\draw (60:2.8) -- (60:4);
\draw[mid>] (75:1) -- (75:2.2);
\draw (75:2.8) -- (75:4);

\draw[mid>] (90:1) -- (90:2.2);
\draw (90:2.8) -- (90:4);
\draw[mid>] (105:1) -- (105:2.2);
\draw (105:2.8) -- (105:4);

\draw[mid>] (120:1) to (120:1.8);
\draw[mid>] (120:1.8) to [out=120, in=410] (140:2.5);
\draw (140:2.5) arc [radius=2.5, start angle=140, end angle=470];
\draw (110:2.5) to [out=200, in=300] (120:3);
\draw (120:3) -- (120:4);

\node[right] at (120:3.5) {\small $k_1$};
\node[right] at (105:3.5) {\small $k_2$};
\node[right] at (90:3.5) {\small $k_3$};
\node[right] at (75:3.5) {\small $k_4$};
\node[right] at (60:3.5) {\small $k_5$};
\end{tikzpicture} \quad
\begin{tikzpicture}[scale=0.9]
\draw[thick] (0,0) to (6,0);
\draw[thick] (0,4) to (6,4);
\draw[thick] (1,3.8) to (1,4.2);
\draw[thick] (5,-0.2) to (5,0.2);

\draw[mid>] (0.5,0) to (0.5,4);
\node [right] at (0.5,3) {\small $k_1$};

\draw[mid>] (1.5,0) to (2.1,2.4);
\draw[mid>] (2.5,0) to (2.95, 1.8);
\draw[mid>] (3.5,0) to (3.85,1.4);
\draw[mid>] (4.5,0) to  (4.75,1);

\draw (2.3, 3.2) to (2.5, 4);
\draw (3.15, 2.6) to (3.5, 4);
\draw (4.05,2.2) to (4.5, 4);
\draw (4.95, 2) to (5.5, 4);

\draw[mid>] (5.5, 0) to [out=90,in=270] (1.5, 4);

\node[left] at (1.58,0.5)  {\small$k_2$};
\node[left] at (2.58,0.5)  {\small $k_3$};
\node[left] at (3.58,0.5)  {\small $k_4$};
\node[left] at (4.58,0.5)  {\small $k_5$};

\node[right] at (5.45, 0.5) {\small $k_1$};
\end{tikzpicture}
\caption{$X_1 $ and its unfolded version.}
\end{figure}
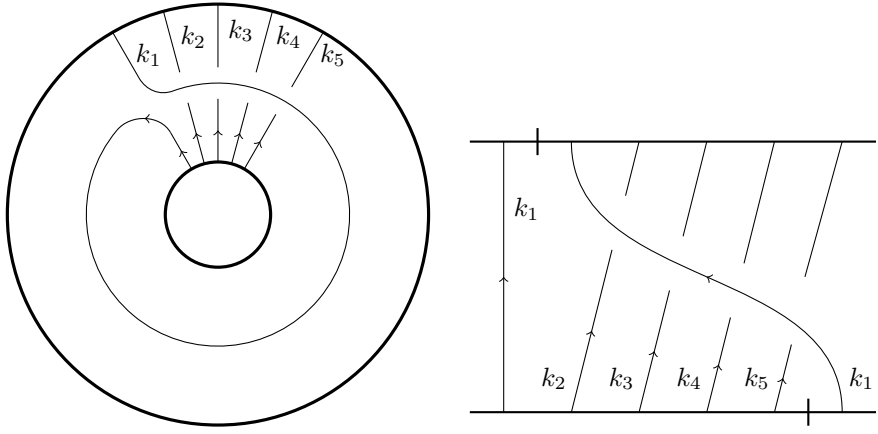

Applying $ \Gamma $ we get the map
$$ V \otimes W \otimes V \xrightarrow{I \otimes \beta^{-1}_{W, V}} V \otimes V \otimes W $$
where $ V =  \Altq{k_1} (\C^n_q) $ and $ W = \Altq{k_2} (\C^n_q) \otimes \cdots \otimes \Altq{k_m} (\C^n_q)$.

If we apply the identification between $ \RqS $ and $ \OqSSm $ given by Proposition \ref{pr:OCmod}, we see that this map corresponds to the map in $ \OqSSm $ given by
\begin{align*}
\OqS \otimes V \otimes W &\rightarrow \OqS \otimes V \otimes W \otimes V \otimes V^* \\
&\xrightarrow{\beta_{V, W}^{-1}} \OqS \otimes V \otimes V \otimes W \otimes V^* \\
&\xrightarrow{\beta^{-1}_{V^*, V \otimes W}} \OqS \otimes V \otimes V^* \otimes V \otimes W \rightarrow  \OqS \otimes V \otimes W
\end{align*}
which is the same as the map $ X_1 $ defined in section \ref{sec:AnnularBraids}, because the composition
$$ \C(q) \rightarrow V \otimes V^* \xrightarrow{\beta^{-1}_{V^*, V}} V^* \otimes V $$
agrees with the composition
$$ \C(q) \rightarrow V^* \otimes V \xrightarrow{I \otimes \theta_V} V^* \otimes V. $$

\end{proof}

\subsection{The action of $E$}\label{sec:actionE}

One can define a map $E[q^\pm] \rightarrow \End_{\ASp[q^\pm]}(\uk)$ by taking $e_k$ to a counterclockwise loop labelled $ k $ passing over all the strands:
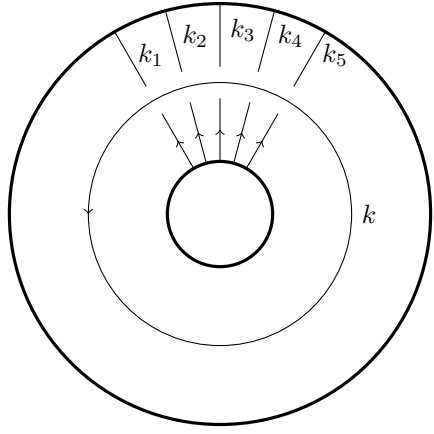
\begin{equation*}
\begin{tikzpicture}[scale=0.7]
\incircle;
\outcircle;
\draw[mid>] (60:1) -- (60:2.2);
\draw (60:2.8) -- (60:4);
\draw[mid>] (75:1) -- (75:2.2);
\draw (75:2.8) -- (75:4);

\draw[mid>] (90:1) -- (90:2.2);
\draw (90:2.8) -- (90:4);
\draw[mid>] (105:1) -- (105:2.2);
\draw (105:2.8) -- (105:4);

\draw[mid>] (120:1) -- (120:2.2);
\draw (120:2.8) -- (120:4);

\draw[mid>] (0:2.5) arc [radius=2.5, start angle=0, end angle=360];
\node[right] at (120:3.5) {\small $k_1$};
\node[right] at (105:3.5) {\small $k_2$};
\node[right] at (90:3.5) {\small $k_3$};
\node[right] at (75:3.5) {\small $k_4$};
\node[right] at (60:3.5) {\small $k_5$};

\node[right] at (0:2.5) {\small $k$};
\end{tikzpicture}
\end{equation*}

Such maps commute with all morphisms in $\ASp[q^\pm]$ and hence give a map $ E[q^\pm] \rightarrow \End_{\ASp[q^\pm]}(id) $ to the endomorphisms of the identity functor.

On the other hand, by the construction in section \ref{se:enrich}, we have a map $ E(q) \rightarrow \End_{\OqSSm}(id) $.  By a similar argument as in the proof of Proposition \ref{pr:ABfactor}, we deduce the following.

\begin{Proposition}
The composition
$$E(q) \rightarrow \End_{\ASp(q)}(id) \xrightarrow{A \Gamma} \End_{\OqSSm}(id) $$
equals the map defined in section \ref{se:enrich}. In other words if we have a counterclockwise loop labelled $ k $, then we get the morphism $ \OqS \otimes V \rightarrow \OqS \otimes V $ given by left multiplication by $ e_k $.
\end{Proposition}

\section{Quantum loop algebras and annular spiders}

It turns out that $\ASp[q^\pm]$ is closely related to quantum loop algebras of $ \gl_m $ for varying $ m$ (a similar relationship between $ \Sp[q^\pm] $ and usual quantum $ \gl_m $ was studied in \cite{CKM}).

\subsection{The quantum loop algebra}

We first recall the integral, idempotent form $\dULq$ of the quantum loop algebra of $\gl_m$. This is a category whose objects are indexed by $\uk = (k_1, \dots, k_m) \in \Z^m $. For $i=1, \dots, m-1$ we denote $\alpha_i = (0,\dots,1,-1,\dots,0)$ where the $1$ is in position $i$. Moreover we denote $\alpha_0 = (-1,0,\dots,0,1)$.

As usual, we will write $\1_\uk$ for the identity morphism of $\uk$. The morphisms in $\dULq$ are further generated over $\C[q^{\pm}]$ by $ E_i^{(s)}, F_i^{(s)} $, for $i = 0, \dots, m-1$, $s \in \N$ and by $R$. These satisfy the following relations.
\begin{enumerate}
\item $\1_{\uk+s\alpha_i} E_i^{(s)} = \1_{\uk+s\alpha_i} E_i^{(s)} \1_\uk = E_i^{(s)} \1_\uk$ and $\1_\uk F_i^{(s)} = \1_\uk F_i^{(s)} \1_{\uk+s\alpha_i} = F_i^{(s)} \1_{\uk+s\alpha_i}$,
\item $R \1_{\uk} = 1_{r \cdot \uk} R$ where $r \cdot (k_1,\dots,k_m) = (k_m,k_1,\dots,k_{m-1})$ is ``rotation'' of the weights,
\item $[E_i^{(s)}, F_i^{(t)}] \1_\uk = \sum_{j \ge 0} \qbins{\la \uk, \alpha_i \ra - t + s}{j} F_i^{(t-j)} E_i^{(s-j)} \1_\uk$ where $\la \cdot, \cdot \ra$ is the standard inner product (here by convention $ F^{(r)} = E^{(r)} = 0 $ if $ r < 0 $, so this is actually a finite sum),
\item $[E_i^{(s)}, F_j^{(t)}] = 0$ if $i \ne j$,
\item if $|i-j|=1$ (modulo $m$) then $E_iE_jE_i = E_i^{(2)} E_j + E_j E_i^{(2)}$ and similarly for $F$'s instead of $E$'s,
\item if $|i-j|>1$ (modulo $m$) then $[E_i^{(r)},E_j^{(s)}]=0=[F_i^{(r)},F_j^{(s)}]$.
\item $R E_i^{(s)}  = E_{i+1}^{(s)} R$ (modulo $m$) and similarly $R F_i^{(s)}  = F_{i+1}^{(s)} R$ (modulo $m$).
\end{enumerate}

\begin{Remark}
The above definition of $ \dULq $ is the integral, idempotent form of Definition 3.1.1 from \cite{Gr}.  Sometimes a slightly smaller algebra is defined without using the $ R $ generator. For a comparison of the two definitions, see section 1.4 of \cite{DG}.
\end{Remark}

We will denote by $(\dULq)^n$ the quotient of $\dULq$ obtained by killing any object not of the form $ (k_1, \dots, k_m) $ with $ 0 \le k_i \le n$. We will also denote by $\dU_{[q^\pm]} \gl_m$ the subcategory generated by $E_i$'s and $F_i$'s for $i=1,\dots,m-1$ and by $(\dU_{[q^\pm]} \gl_m)^n$ the corresponding quotient.

The category $(\dULq)^n$ (or more precisely the direct sum of all Hom spaces in this category) is a special case of an affine generalized $ q$-Schur algebra (or BLN algebra), as defined in section 6.1 of \cite{Mc}.  As explained by McGerty (Prop 6.4 in \cite{Mc}), the following result is a consequence of the work of Beck-Nakajima \cite{BN}.

\begin{Lemma}\label{lem:free1}
All Hom sets in $ (\dULq)^n $ are free over $ \C[q^\pm] $.
\end{Lemma}

\subsection{The braiding functor}\label{sec:braid}

Like the braiding structure (\ref{eq:braiding1}) in $\Sp$ we can define a functor $ \B_m(\{0, \dots, n\}) \rightarrow (\dU_{[q^\pm]} \gl_m)^n$ as follows (following \cite[5.2.1]{Lus}). On objects it takes $\uk \mapsto \uk$. On morphisms it takes
\begin{equation}\label{eq:Ti}
T_i \1_{\uk} \mapsto (-1)^{k_ik_{i+1}} (-q)^{k_i} \sum\limits_{\substack{a,b\geq 0 \\ b-a = k_i-k_{i+1}}} (-q)^{-b}E_i^{(a)} F_i^{(b)} \1_{\uk}.
\end{equation}

\begin{Remark}\label{rem:T}
If $k_i = k_{i+1} = 1$, the formula for the braiding simplifies to give
$$ T_i \1_{\uk} \mapsto  q \1_{\uk} - E_i F_i \1_{\uk}. $$
\end{Remark}

\begin{Lemma}\label{lem:affine}
For $i,j \in \{1, \dots, m-1\}$ we have, inside $(\dU_{[q^\pm]} \gl_m)^n$, $T_j E_i^{(s)} = E_i^{(s)} T_j$ if $|i-j| > 1$ and $T_i T_j E_i^{(s)} = E_j^{(s)} T_i T_j$ if $|i-j|=1$, where the $T$'s are given by (\ref{eq:Ti}).
\end{Lemma}
\begin{proof}
This follows from section 39.2.4 of \cite{Lus} (see also Lemma 6.1.3 of \cite{CKM}).
\end{proof}

\subsection{Representations of $(\dULq)^n$}

The description above gives $\dULq$ in its Kac-Moody form rather than its more common loop form. In the latter form, one also considers generators $E_i \otimes t^r, F_i \otimes t^r$, for $r \in \Z$, but no $E_0$ or $F_0$.

The advantage of the Kac-Moody presentation of $\dULq$ is that it is more closely related to the annular spider category $\ASp$.   However, we will need to define representations of $(\dULq)^n $ which are more naturally defined in the loop presentation. The following technical lemma tells us how to deal with this issue. It shows that it suffices to define the functor on $(\dU_{[q^\pm]} \gl_m)^n$ together with a compatible functor from the annular braid groupoid.

\begin{Lemma}\label{lem:extend}
Let $ \sC $ be any $ \C$-linear category.

Suppose you have a $\C$-linear functor $\phi: (\dU_{[q^\pm]} \gl_m)^n \rightarrow \sC$ and a functor $ \Beta : \AB_m(\{0,\dots,n\}) \rightarrow \sC$
satisfying:
\begin{enumerate}
\item $\Beta(T_i) = \phi(T_i) $ for all $ i = 1, \dots, m-1 $ where in the case of $\phi(T_i)$ we regard $T_i$ as a morphism in $(\dU_{[q^\pm]} \gl_m)^n$ via (\ref{eq:Ti}),
\item $ \Beta(X_1)$ commutes with $\phi(E_i^{(s)})$ and $\phi(F_i^{(s)})$ for $i=2, \dots, m-1$ and
\item $\Beta(X_1 \dots X_m)$ commutes with everything in the image of $\phi$.
\end{enumerate}
Then $\phi$ and $\Beta$ extend to a $\C$-linear functor $\Phi: (\dULq)^n \rightarrow \sC$.
\end{Lemma}
\begin{proof}
 Since $\phi$ already tells us where to map $E_i^{(s)}, F_i^{(s)}$, for $i=1, \dots, m-1$ and all $s \in \N$, we only need to say where $E_0^{(s)}, F_0^{(s)}$ and $R$ are mapped. We define
$$\Phi(R \1_{\uk}) := \Beta(X_1 T_1 \cdots T_{m-1} \1_{\uk})$$
and then $\Phi(E_0^{(s)}) := \Phi(R) \Phi(E_{m-1}^{(s)}) \Phi(R^{-1})$ and similarly for $F_0^{(s)}$.

Now we need to show that $\Phi(R) \Phi(E_i^{(s)}) =  \Phi(E_{i+1}^{(s)}) \Phi(R)$ holds (the analogous results for $F$'s follow similarly). For $i=1, \dots, m-2$, we have
\begin{equation} \label{eq:case1m-2}
\begin{aligned}
\Phi(R) \Phi(E_i^{(s)})
&= \Beta(X_1) \Phi(T_1 \cdots T_{m-1} E_i^{(s)}) \\
&= \Beta(X_1) \Phi(E_{i+1}^{(s)} T_1 \dots T_{m-1}) \\
&= \Phi(E_{i+1}^{(s)}) \Beta(X_1) \Beta(T_1 \dots T_{m-1}) \\
&= \Phi(E_{i+1}^{(s)}) \Phi(R)
\end{aligned}
\end{equation}
where the second equality follows using Lemma \ref{lem:affine} repeatedly, while the third follows by condition (1).

Now, we check the cases $ m = 0, m-1 $.  The case $i=m-1$ holds by definition.  For the case $ i = 0 $, we must show that $ \Phi(R) \Phi(E_0^{(s)}) =  \Phi(E_1^{(s)}) \Phi( R)$.  This follows by writing $\Phi(E_0^{(s)}) = \Phi(R) \Phi(E_{m-1}^{(s)}) \Phi(R^{-1})$ and then using \eqref{eq:case1m-2} repeatedly, together with the fact that $\Phi(R)^m = \Phi(R^m) = \Beta(X_1 \dots X_m)$ commutes with everything (condition 3). This completes the proof.
\end{proof}

\subsection{Embeddings}\label{sec:embed}

Observe that we have $2(m+1)$ faithful functors $ (\dULq)^n \rightarrow (\dU_{[q^\pm]} L \gl_{m+1})^n $, given by adding either a $0$ or an $n$ in some slot. For concreteness let us describe the functor of adding a $0$ at the end. On objects it is given by $(k_1, \dots, k_m) \mapsto (k_1, \dots, k_m, 0)$.  On morphisms it is given by
\begin{align*}
& R \mapsto R, \\
& E_i^{(s)} \mapsto E_i^{(s)} \text{ and } F_i^{(s)} \mapsto F_i^{(s)} \text{ if } i \ne 0, \\
& E_0^{(s)} \mapsto E_m^{(s)} E_0^{(s)} \text{ and } F_0^{(s)} \mapsto F_0^{(s)} F_m^{(s)}.
\end{align*}

\begin{Proposition}\label{prop:embed}
The functor $ (\dU_{[q^\pm]} L \gl_m)^n \rightarrow (\dU_{[q^\pm]} L \gl_{m+1})^n $ is faithful.
\end{Proposition}

\begin{proof}
For the purposes of this proof, we will use $ (U_{[q^\pm]} L \gl_m)^n $ to denote the ``algebra version'' of the category $(\dU_{[q^\pm]} L \gl_m)^n$ (i.e. the algebra obtained as the direct sum of all Hom spaces in the category). Then we must prove that the (non-unital) map of algebras $(U_{[q^\pm]} L \gl_m)^n \rightarrow (U_{[q^\pm]} L \gl_{m+1})^n $ is injective.

For any $ \lambda = (\lambda_1, \dots, \lambda_m) $ with $ \lambda_1 \ge \dots \ge \lambda_m $, let $ V(\lambda) $ denote the extremal weight module for $ U_{[q^\pm]} L\gl_m $.  Let $ P_{n,m} = \{ (\lambda_1, \dots, \lambda_m) : n \ge \lambda_1 \dots \ge \lambda_m \ge 0 \} $.   By Proposition 6.4 of \cite{Mc}, we know that $ (U_{[q^\pm]} L \gl_m)^n $ is the image of $ U_{[q^\pm]} L\gl_m $ inside $ \oplus_{\lambda \in P_{n,m}} \End V(\lambda) $.

If $ \lambda \in P_{n,m} $, then we can add a $0 $ to this sequence to produce $ (\lambda, 0) \in P_{n,m+1} $.  By the lemma below, the restriction of $ V(\lambda, 0) $ to $ U_{[q^\pm]} L\gl_m $ contains $ V(\lambda) $.  Thus the composite map
$$ (U_{[q^\pm]} L \gl_m)^n \rightarrow (U_{[q^\pm]} L \gl_{m+1})^n \rightarrow \oplus_{\lambda \in P_{n,m}} \End(V(\lambda, 0)) $$
is injective. Thus, $ (U_{[q^\pm]} L \gl_m)^n \rightarrow (U_{[q^\pm]} L \gl_{m+1})^n $ is injective as desired.
\end{proof}

\begin{Lemma}
Let  $\lambda \in P_{n,m} $ and let $ v_{\lambda, 0} $ be the generating vector of $ V(\lambda, 0) $.  Then $ v_{\lambda, 0} $ generates $ V(\lambda) $ as a $\dULq$-module.
\end{Lemma}

\begin{proof}
Let us write $ \lambda = \sum_{i=1}^m w_i \omega_i $ where $ \omega_i = (1, \dots, 1, 0, \dots,0) $ has length $ m $.  By Corollary 4.15 of \cite{BN}, we know that there is an injective map $ V(\lambda, 0) \rightarrow \otimes_{i \in I} V(\omega_i,0)^{\otimes w_i} $ taking $ v_{\lambda, 0} $ to the tensor products $ \otimes v_{\omega_i, 0}^{\otimes w_i} $ of the generating vectors.

We also know that $ V(\omega_i,0) $ is isomorphic to $ \Altq{i} \C^{m+1} \otimes \C[t, t^{-1}] $.  Thus, upon restriction to $\dU_{[q^\pm]} L\gl_m $ it is easy to see that the highest weight vector of $ V(\omega_i,0) $ generates a copy of $ V(\omega_i) $.

Thus, we see that as a $\dULq$-module, $ v_{\lambda, 0} $ generates the same module as does $  \otimes v_{\omega_i}^{\otimes w_i} $ inside $ \otimes V(\omega_i)^{\otimes w_i} $ and hence the  Lemma follows.
\end{proof}

\subsection{From quantum loop algebras to annular spiders}

In \cite{CKM} we studied the spider category $\Sp[q^\pm]$ by defining a functor $\Psi_m: (\dU_{[q^\pm]} \gl_m)^n \rightarrow \Sp[q^\pm]$ for each $m$. In exactly the same way one can define a functor
$$A\Psi_m : (\dULq)^n \rightarrow \ASp[q^\pm].$$
More precisely, on objects it takes $\uk$ to $\overline{\uk} $ where $ \overline{\uk} $ is the sequence obtained from $\uk$ by deleting all $0,n$. On morphisms we define $\APsi_m$ by forming ``ladder-like'' webs, as in \cite{CKM}, the only difference being that $ E_0^{(s)}, F_0^{(s)} $ are sent to webs which ``wrap-around'' the annulus crossing the cut line, as illustrated in (\ref{eq:psi}) with $E_0^{(s)}$. Note that this functor was also defined in \cite[Section 2.3]{Q}.
\begin{equation}\label{eq:psi}
A\Psi_3(E_0^{(s)}) = \begin{tikzpicture}[baseline=15,scale=0.8]
\incircle;
\outcircle;
\draw[mid>] (60:1) -- (60:4);

\draw[mid>] (90:1) -- (90:4);

\draw[mid>] (120:1) -- (120:4);

\draw (120:2) to [out=150, in=410] (140:2.5);
\draw[mid>] (140:2.5) arc [radius=2.5, start angle=140, end angle=400];
\draw (40:2.5) to [out=130, in=270] (60:3);

\node[right] at (120:3.5) {\scriptsize $ k_1-s$};
\node[right] at (120:1.5) {\scriptsize $k_1$};
\node[right] at (90:3.5) {\scriptsize $k_2$};
\node[right] at (60:1.5) {\scriptsize $k_3$};
\node[right] at (60:3.2) {\scriptsize $k_3+s$};

\node[right] at (0:2.5) {\scriptsize $s$};
\end{tikzpicture}
\end{equation}

\begin{Theorem}\label{thm:1}
The functor $ \APsi_m : (\dULq)^n \rightarrow \ASp[q^\pm] $ is faithful.

Moreover, suppose that $ \uk, \uk' $ are two sequences drawn from $ \{1, \dots, n-1 \}$.  Given a morphism $ w \in \Hom_{\ASp[q^\pm]}(\uk, \uk') $, there exists some $m$ and a way to add $0$s and $n$s to $ \uk, \uk' $ to form $\ul, \ul'$ of length $m$ such that $w $ lies in the image of
$$ \APsi_m : \Hom_{(\dULq)^n}(\ul, \ul') \rightarrow \Hom_{\ASp[q^\pm]}(\uk,\uk').$$
\end{Theorem}
\begin{proof}
We begin with the surjectivity part of the statment. Note that is suffices to prove this for the case when $ w $ is an annular web (rather than a linear combination). If we have any annular web $ w $ between two objects $ \uk, \uk' $, then we can turn $ w $ into a ladder-like web as follows.

We will call the radial coordinate on the annulus, time.  Let $ t_1, \dots, t_{p-1} $ be the times at which $ w $ crosses the cut line.  Then we can write $ w $ as the composition $ w = w_p R^\pm \cdots R^\pm w_2 R^\pm w_1 $ where $ w_i $ denotes the part of $ w $ between times $ t_{i-1} $ and $ t_{i} $ (we write $ t_0 $ for the inside and $ t_p $ for the outside).  In this composition, we see $ R $ if the cut line is cut with a right-pointing strand and $ R^{-1} $ if the cut line is cut with a left-pointing strand.  Each web $ w_i $ is a web which doesn't cross the cut line and thus is a planar web.

In order to apply Theorem 5.3.1 from \cite{CKM} we need each web $w_i$ to have their starting and ending strands pointing outwards (up in the terminology of \cite{CKM}). To achieve this, we can insert pairs of cancelling tags everywhere we have downward pointing strands at some time $t_i$.  Abusing notation slightly, we continue to use $ w_1, \dots, w_p $ to denote the webs obtained by this procedure.

Now, for each $i$ we can apply \cite[Theorem 5.3.1]{CKM} to find $m_i$ and an element $ a_i $ such that $ \Psi_{m_i}(a_i) = w_i $.  Moreover, note that
$$\Hom((k_1, \dots, k_m), (k_2, \dots, k_m,k_1)) \ni R = \APsi_{m+1}(E_0^{(k_1)})$$
where $E_0^{(k_1)} \in \Hom((k_1, \dots, k_m, 0), (0,k_2, \dots, k_m,k_1))$ and similarly for $R^{-1}$. Combining all this together (by inserting as many $0$s and $n$s as necessary) we get the desired statement.

It remains to show that $\APsi_m$ is faithful. For this, let us note that both categories $ (\dULq)^n $ and $ \ASp[q^\pm] $ are defined by generators and relations.  Let $ \ul, \ul' $ be two objects in $ (\dULq)^n $ and let $ \uk, \uk' $ denote the result of removing all $ 0, n $ from them.  Let $ F_{(\dULq)^n}(\ul, \ul') $ denote the space of morphisms between these two objects in the free category with the same generators as $ (\dULq)^n $ and let $ R_{(\dULq)^n}(\ul, \ul') $ denote the relations, so that $ \Hom_{(\dULq)^n}(\ul, \ul') = F_{(\dULq)^n}(\ul, \ul') / R_{(\dULq)^n}(\ul, \ul') $.  Similarly, we can define $ F_{\ASp[q^\pm]}(\uk, \uk') $ and $ R_{\ASp[q^\pm]}(\uk, \uk') $ and we have $ \Hom_{\ASp[q^\pm]}(\uk, \uk') = F_{\ASp[q^\pm]}(\uk, \uk')/ R_{\ASp[q^\pm]}(\uk, \uk') $.

Now, suppose that $ a \in \Hom_{(\dULq)^n}(\ul, \ul') $ and $ \APsi_m(a) = 0 $ in $ \Hom_{\ASp[q^\pm]}(\uk, \uk') $.  This means that we have an element
$ a \in F_{(\dULq)^n}(\ul, \ul') $ and that $ \APsi_m(a) \in R_{\ASp[q^\pm]}(\uk, \uk')$.  This means that $ a $ can be written as a linear combination of compositions of a generating relation in $ \ASp[q^\pm] $ with elements of $ \ASp[q^\pm] $.  However, by inspection, we see that each generating relation (given in section \ref{se:spider}) in $ \ASp[q^\pm] $ comes from a relation in some $ \dU_{[q^\pm]} L \gl_p $.  This observation, combined with the surjectivity proven above, shows that there exists $ M $ and sequences $ \underline{L}, \underline{L'} $ obtained by  adding $0,n $ to $ \uk, \uk' $, and  $ b \in R_{(\dU_{[q^\pm]} L \gl_M)^n}(\underline{L}, \underline{L'}) $ such that $ \APsi_M(b) =\APsi_m(a) $.  This means that $ a $ becomes 0 in $ \Hom_{(\dU_{[q^\pm]} L \gl_M)^n}(\underline{L}, \underline{L'})$.  However, by Proposition \ref{prop:embed}, the functor $ (\dULq)^n \rightarrow (\dU_{[q^\pm]} L \gl_M)^n $ is faithful, which means that $ a = 0 $ in $ \Hom_{(\dULq)^n}(\ul, \ul')$ as desired.
\end{proof}

\begin{Remark}
Theorem \ref{thm:1} explains that one should think of the category $\ASp[q^\pm]$ as the directed limit as $m \rightarrow \infty$
$$(\dU_{[q^\pm]} L \gl_\infty)^n := \varinjlim (\dU_{[q^\pm]} L \gl_m)^n$$
via the functors discussed in section \ref{sec:embed}. One could make this precise but we do not actually require this stronger result.
\end{Remark}

\section{From the annular spider to KConv}\label{sec:loopaction}
In this section we define a functor
$$ \Phi_m : (\dULq)^n \rightarrow KConv^{\A}(Gr) $$
which follows our earlier constructions from \cite{CKL,Ca,CK3} but descended from derived categories to K-theory. This in turn induces a functor
$$\Phi : \ASp[q^\pm] \rightarrow KConv^{\A}(Gr)$$
which we will prove is an equivalence after tensoring with $ \C(q) $.

\subsection{The varieties} \label{sec:varieties}
We work with the lattice description of $ Gr = Gr_{PGL_n}$.  This means that we identify $ Gr$ with the space of $ \cO$-lattices in the vector space $ \K^n $ (where $ \cO = \C[[z]] $ and $ \K = \C((z))$) modulo the equivalence relation of homothety, i.e. $L_1 \sim L_2$ if and only if $ z^k L_1 = L_2 $ for some $ k \in \Z $.

Note that under this identification, if $ \lambda \in \Lambda_+ \subset \Z^n / \Z(1,\dots, 1) $, then we have
$$
Gr^\lambda = \{ L \in Gr : \cO^n= L_0 \subset L \text{ and } z|_{L/L_0} \text{ has Jordan type } \lambda \}
$$
where we regard $ z|_{L/L_0} $ as a nilpotent linear operator on a finite-dimensional vector space and consider its Jordan type to be the length of blocks in its Jordan form.  (Note that $ \lambda $ is only defined up to shift by a constant sequence; this matches the fact that $ L $ is only defined up to homothety.)

As before, we encode of sequence of minuscule dominant weights $ \ulam $ by a sequence $ \uk= (k_1, \dots, k_m)$ with $ 1 \le k_i \le n-1 $.  Let us introduce the notation $Y(\uk) := Gr^{\ulam}$. The varieties $ Y(\uk) $ can be described explicitly as sequences of lattices (not considered up to homothety).

\begin{Lemma}
We have
\begin{align*}
Y(\uk) = \{ (L_1, \dots, L_m) :& L_i \subset \K^n, \ \cO^n = L_0 \subset L_1 \subset L_2 \subset \dots \subset L_m, \\
& \dim(L_i/L_{i-1}) = k_i \text{ and } zL_i \subset L_{i-1} \}.
\end{align*}
Moreover the projection map $ Y(\uk) \rightarrow Gr $ is given by sending $(L_1, \dots, L_m) \mapsto [L_m] $.
\end{Lemma}

For convenience of notation, we also allow sequences $ \uk $ containing $ 0, n $.  We have a canonical identification $ Y(\uk) = Y(\overline{\uk}) $ where $ \overline{\uk} $ is defined by removing all $ 0,n $.

Assume that $ \uk = (k_1, \dots, k_m) $ and $ \uk' = (k'_1, \dots, k'_{m'}) $ are two sequences with the same sum, $ k_1 + \dots + k_m = k'_1 + \cdots + k'_{m'} $.  From the lemma, we reach a simple description of the varieties
$$ Z(\uk, \uk') := \{ (L_1, \dots, L_m), (L'_1, \dots, L'_{m'}) \in Y(\uk) \times Y(\uk') : L_m = L'_{m'} \}.$$

We consider the action of $ \A $ on $  \K^n = \C^n \otimes \K $ where $ SL_n $ acts on $ \C^n $ and $ \Cx $ acts on $\K $ by $ s \cdot z^k = s^k z^k $.  This gives us an action of $ \A $ on $ Gr $ and on each $ Y(\uk) $.

\begin{Remark}\label{rem:q}
Given a sheaf $\sF$ we denote by $\sF \{1\}$ the same sheaf but shifted equivariantly with respect to the $\Cx$ action. At the level of K-theory we use the convention that $\{1\}$ is multiplication by $-q^{-1}$. This matches our previous conventions from \cite{CKL,Ca}.
\end{Remark}

\subsection{The action of $(\dU_{[q^\pm]} \gl_m)^n$}

We first define a functor
$$\phi_m: (\dU_{[q^\pm]} \gl_m)^n \rightarrow KConv^\A(Gr).$$
On objects $\phi_m$ takes $\uk$ to $ \overline{\uk} $ where $ \overline{\uk} $ is obtained from $\uk$ by removing all $0,n$. To define it on morphisms, we must define maps
$$\phi_m : \Hom_{(\dU_{[q^\pm]} \gl_m)^n}(\uk, \uk') \rightarrow K^\A(Z(\overline{\uk}, \overline{\uk'}))$$
for every $ \uk, \uk' $. Since $(\dU_{[q^\pm]} \gl_m)^n$ is defined by generators and relations we will define $ \phi_m $ on the generators $E_i^{(s)}$ and $F_i^{(s)}$ for $ i = 1, \dots, m-1 $. To do this we use the varieties
$$W_i^s(\uk) = \{(L_\bullet,L'_\bullet) \in Y(\uk) \times Y(\uk-s\alpha_i): L_j = L_j' \text{ if } j \ne i, L_i' \subset L_i \}.$$
Note that these varieties are actually supported on $Z(\uk,\uk-s\alpha_i)$. Then we define
\begin{align}
\label{eq:B} \phi_m(\1_\uk E_i^{(s)}) &= (-q)^{-sk_{i+1}} [\cO_{W_i^s(\uk)} \otimes \det(L_i'/L_i)^{k_{i+1}-k_i+s}] \\
\label{eq:A} \phi_m(F_i^{(s)} \1_\uk) &= (-q)^{-s(k_i-s)} [\cO_{W_i^s(\uk)} \otimes \det(L_{i+1}/L_i')^{-s} \otimes \det(L_i/L_{i-1})^s]
\end{align}
where the prime denotes the corresponding bundle pulled back from the second factor. These kernels live in $K^\A(Z(\uk-s\alpha_i,\uk))$ and $K^\A(Z(\uk,\uk-s\alpha_i))$ respectively). It follows from the calculations in \cite{CKL,Ca} that these satisfy the relations of $\dU_{[q^\pm]} \gl_m$ (see for instance \cite[Theorem 8.2]{Ca}).

\begin{Remark}
These definitions differ from those in \cite{CKL,Ca,CK3} in two ways. First we have reversed the roles of $ E_i, F_i $ and negated weights (this procedure actually defines an involution of $(\dU_{[q^\pm]} \gl_m)^n$). We do this in order to use the standard definition of $ \alpha_i = (0, \dots, 1,-1, \dots, 0)$ and thereby match the previous sections of this paper. Second, the choice of line bundles matches that in \cite{CKL} but not \cite{Ca,CK3}. The difference is just conjugation by a line bundle (denoted $\rho$ in \cite{CK3}). The reason for choosing the line bundles from \cite{CKL} rather than \cite{Ca,CK3} is that they make the diagrams in section \ref{sec:iso1m} commute on the nose without having to conjugate by line bundles.
\end{Remark}

\subsection{Extending to an action of $(\dULq)^n$}

To extend the action above we first define an action of the annular braid group on $KConv^\A(Gr)$ and then use Lemma \ref{lem:extend}.

\begin{Proposition}\label{prop:affine}
There exists a functor
$$\Beta: \AB_m(\{0, \dots, n\}) \rightarrow KConv^\A(Gr)$$
where $\Beta (T_i \1_\uk)$ is defined using equation (\ref{eq:Ti}) and $\Beta (X_i \1_\uk) := \Delta_* \det(L_i/L_{i-1})^\vee$ where $\Delta$ is the diagonal embedding $Y(\uk) \rightarrow Z(\uk,\uk)$.
\end{Proposition}
\begin{proof}
We need to show the last three relations of $AB_m$. The relations $T_iX_j = X_jT_i$ and $X_iX_j=X_jX_i$ are easy. The final relation $T_iX_iT_i = X_{i+1}$ follows from \cite[Cor. A.14]{CK3}.
\end{proof}

In order to apply Lemma \ref{lem:extend} we need to check the following two conditions.
\begin{enumerate}
\item $\Beta(X_1)$ commutes with $\phi_m(E_i^{(s)})$ and $\phi_m(F_i^{(s)})$ for $i=2, \dots, m-1$ and
\item $\Beta(X_1 \dots X_m)$ commutes with everything in the image of $\phi_m$.
\end{enumerate}
The first condition follows since $\Beta(X_1)$ only involves the line bundle $\det(L_1/L_0)$ while $\phi_m (E_i^{(s)})$ only involves changing the flag $L_i$. The second condition follows since $\Beta(X_1 \dots X_m)$ is given by $\Delta_* \det(L_m/L_0)^\vee$  which clearly commutes with the image of every $E_i^{(s)}$ and $F_i^{(s)}$. Thus, by Lemma \ref{lem:extend}, $\phi_m$ extends to a functor
$$\Phi_m : (\dULq)^n \rightarrow KConv^\A(Gr).$$

\begin{Remark}
It is easy to see that the functors $\Phi_m$ are compatible with the functors $(\dU_{[q^\pm]} L \gl_m)^n \rightarrow (\dU_{[q^\pm]} L \gl_{m+1})^n$ discussed in the proof of Theorem \ref{thm:1}.
\end{Remark}

\subsection{The functor $\Phi$}

\begin{Proposition}
There exists a unique functor $$ \Phi: \ASp[q^\pm] \rightarrow KConv^{\A}(Gr) $$  such that for each $ m $ we have $  \Phi_m = \Phi \circ \APsi_m $.
\end{Proposition}

\begin{proof}
Consider the full subcategory of $ \ASp[q^\pm] $ consisting of objects $ (k_1, \dots, k_m) $ with $ k_i \in \{1, \dots, n-1 \} $ (in other words, we are considering webs where the strands point away from the inner circle and into the outer circle).  Because of the tags, every object in $ \ASp[q^\pm] $ is equivalent to an object in this subcategory.  Thus, it suffices to define $ \Phi $ on this subcategory.

On objects,  we define $ \Phi $ to take $ \uk $ to $ \uk $. Now consider objects $ \uk, \uk' $ in $ \ASp[q^\pm] $ (with all $ k_i, k'_i > 0 $) and let $ a \in \Hom_{\ASp[q^\pm]}(\uk, \uk') $. We wish to define $ \Phi(a) $. By Theorem \ref{thm:1}, we know that there exists $ m $ and $ \ul, \ul' $ (such that $ \ul , \ul' $ are obtained by adding to $0$s and $n$s to $ \uk, \uk' $) and $ b \in \Hom_{\dUL}(\ul, \ul') $ with $ \APsi_m(b) = a $.  Then we define $ \Phi(a) := \Phi_m(b) $.

Finally, we must check that this is well-defined. Suppose we have two choices of $ m_1, \ul_1, \ul'_1, b_1$ and $m_2, \ul_2, \ul'_2, b_2 $ as above.  Since both $ \ul_1 $ (resp. $\ul'_1$) and $ \ul_2 $ (resp. $\ul'_2$) are obtained by adding $0,n$ to $ \uk $ (resp. $\uk'$), we can find a third $ \ul_3 $ (resp. $\ul'_3$) which is obtained by adding $0, n $ to both $ \ul_1$ and $ \ul_2 $ (resp. $ \ul'_1, \ul'_2 $).

Using the functors $ (\dULq)^n \rightarrow (\dU_{[q^\pm]} L \gl_{m+1})^n $ defined in the proof of Theorem \ref{thm:1}, we can regard $ b_1, b_2$ inside $\Hom_{(\dU_{[q^\pm]} L\gl_{m_3})^n}(\ul_3, \ul'_3) $.  We have $ \APsi_{m_3}(b_1) = \APsi_{m_3}(b_2) = a $ and thus $ b_1 = b_2 $ by Theorem \ref{thm:1}.  Since the functors $\Phi_m$ are compatible with the inclusion functors  $ (\dULq)^n \rightarrow (\dU_{[q^\pm]} L \gl_{m+1})^n $, this completes the proof of existence and uniqueness is clear by construction.
\end{proof}

The remainder of the paper will be devoted to proving the following result which combined with Corollary \ref{cor:GamEquiv} completes the proof of Conjecture \ref{co:main} in the case $ G = SL_n $.

\begin{Theorem}\label{thm:asp-gr}
The functor $\Phi$ induces equivalences
\begin{align*}
\ASp &\xrightarrow{\sim} KConv^{SL_n}(Gr) \\
\ASp(q) &\xrightarrow{\sim} KConv^{\A}(Gr) \otimes_{\C[q^\pm]} \C(q).
\end{align*}
\end{Theorem}
\begin{proof}
Since the objects on either side are the same it suffices to show that the functor $\Phi$ is fully-faithful. By Lemma \ref{lem:1} this reduces to showing fully-faithfulness for the weight $1^m = (1, \dots, 1)$. This in turn is proven in Theorem \ref{thm:2}.
\end{proof}

\subsection{Compatibility with $ E $-action}
 Recall that inside $\ASp[q^\pm]$ we have the elements $e_j$ corresponding to counterclockwise loops. This was used to define a $ E[q^\pm]$-module structure to the Hom spaces in $ \ASp[q^\pm]$.    The following result determines the image of $e_j$ under $\Phi$.

\begin{Lemma}\label{lem:imek}
For any $\uk$, the image of $e_j \in \End_{\ASp[q^\pm]}(\uk)$ inside $K^{\A}(Z(\uk,\uk))$ is $\Delta_* \wedge^{n-j}(\C^n)$ where $\Delta$ is the embedding of $Y(\uk)$ as the diagonal and $\wedge^{n-j}(\C^n)$ is the natural $SL_n$-equivariant vector bundle on $Y(\uk)$.
\end{Lemma}

\begin{proof}
To simplify notation we prove the case $\uk = \emptyset$ (the general case is no different). Consider the composition
\begin{equation}\label{eq:2A}
\emptyset \xrightarrow{cup} (j,n-j) \xrightarrow{X_1} (j,n-j) \xrightarrow{cap} \emptyset
\end{equation}
inside $\ASp[q^\pm]$ where the first map is a $j$-cup and the last map is a $j$-cap. This composition recovers $e_j$ up to a positive twist involving the $j$-strand. Since such a positive twist can be undone at the cost of a factor of $(-q)^{-j(n-j)}$ the composition in (\ref{eq:2A}) is equal to the image of $(-q)^{j(n-j)} e_j$.

On the other hand, the composition in (\ref{eq:2A}) is equivalent to the composition
\begin{equation}\label{eq:2B}
(0,n) \xrightarrow{E_1^{(j)}} (j,n-j) \xrightarrow{X_1} (j,n-j) \xrightarrow{F_1^{(j)}} (0,n).
\end{equation}
Let us first consider what happens to this composition without the middle map $X_1$. We will need the natural projection and inclusion maps
$$Y(0,n) \xleftarrow{\pi} X(j,n-j) \xrightarrow{i} Y(j,n-j)$$
where
$$X(j,n-j) := \{L_0 \subset L_1 \subset L_2 = z^{-1}(L_0): \dim(L_1/L_0) = j, \dim(L_2/L_1) = n-j\}$$
can be identified with the Grassmannian $\G(j,n)$. Then $F_1^{(j)} E_1^{(j)}$ is given by the composition
$$(-q)^{-j(n-j)} \pi_*(i^* i_* \pi^*(\cdot) \cdot [\det(L_1/L_0)]^{n-j} \cdot [\det(L_2/L_1)]^{-j}).$$
Notice that $\Omega_{X(j,n-j)} \cong (L_2/L_1)^\vee \otimes (L_1/L_0)$ which means that
$$\omega_{X(j,n-j)} \cong \det(L_1/L_0)^{n-j} \otimes \det(L_2/L_1)^{-j}.$$
So the composition above is just
\begin{equation}\label{eq:8}
(-q)^{-d} \pi_*(i^* i_* \pi^*(\cdot) \cdot [\omega_{X(j,n-j)}])
\end{equation}
where $d=j(n-j)= \dim X(j,n-j)$. Now, to compute $i^* i_*$ notice that $X(j,n-j) \subset Y(j,n-j)$ is carved out as the zero section of
\begin{equation}\label{eq:9}
z: L_2/L_1 \rightarrow L_1/L_0 \{2\}
\end{equation}
where the $\{2\}$ shift is because of the $\C^\times$ action on $z$. This means that we have a Koszul resolution
$$ \dots \rightarrow \Alt{2} V \rightarrow V \rightarrow \cO_{Y(j,n-j)} \rightarrow \cO_{X(j,n-j)}$$
where $V = (L_2/L_1) \otimes (L_1/L_0)^\vee \{-2\}$. But $V$ restricted to $X(j,n-j)$ is just $\Omega_{X(j,n-j)}^\vee \{-2\}$ and
$$\Alt{i}(\Omega_{X(j,n-j)}^\vee \{-2\}) \otimes \omega_{X(j,n-j)} \cong \Omega^{d-i}_{X(n-j)} \{-2i\}.$$
Hence the composition in (\ref{eq:8}) simplifies to give
$$(-q)^{-d} \sum_{i \ge 0} \pi_* ((\cdot) \cdot (-1)^i q^{2i}[\Omega_{X(j,n-j)}^{d-i}]) = \sum_{i \ge 0} (-1)^{d+i} q^{d-2i} \pi_* [\Omega_{X(j,n-j)}^i] \cdot (\cdot).$$
Since $X(j,n-j) \cong \G(j,n)$ it is a standard fact that this evaluates to multiplication by $(-1)^d \qbins{n}{j}$.

We now need to reintroduce $X_1$. Since $X_1$ is tensoring with $\det(L_1/L_0)^\vee$ the composition in (\ref{eq:2B}) simplifies to give
$$\sum_{i \ge 0} (-1)^{d+i} q^{d-2i} \pi_* \left( [\Omega^i_{X(j,n-j)}] \cdot [\det(L_1/L_0)^\vee] \right) \cdot (\cdot).$$
Since $\cO_{X(j,n-j)} \otimes \det(L_1/L_0)^\vee \cong \cO_{\G(j,n)}(1)$ all these terms vanish except the lowest degree term $i=0$ which gives
$$(-q)^d \pi_*([\cO_{X(j,n-j)}] \cdot [\det(L_1/L_0)^\vee]) \cdot (\cdot).$$
This is equivalent to multiplication by $(-q)^d [\wedge^{n-j}(\C^n)]$. It follows that $e_j \in \End_{\ASp[q^\pm]}(\emptyset)$ is mapped to multiplication by $[\wedge^{n-j}(\C^n)]$. The result follows.
\end{proof}

\begin{Corollary}
The functor $\Phi$ is a $E[q^\pm]$-linear functor where
\begin{itemize}
\item $e_j \in \ASp[q^\pm]$ acts as a counterclockwise loop labelled $j$,
\item $e_j \in KConv^{\A}(Gr)$ acts by tensoring with the vector bundle $\wedge^{n-j}(\C^n)$.
\end{itemize}
\end{Corollary}

\subsection{Reduction to $1^m $}

\begin{Lemma}\label{lem:1}
To show that $\Phi$ is an equivalence it suffices to check that for each $m$, $\Phi$ induces an isomorphism
$$\End_{\ASp(q)}(1^m) \xrightarrow{\sim} \End_{KConv^{\A}(Gr)}(1^m) \otimes_{\C[q^\pm]} \C(q).$$
\end{Lemma}
\begin{proof}
By a ``fork'' we mean a map $1^k \rightarrow k$ (or its mirror $k \rightarrow 1^k$) depicted in figure \ref{fig:fork} when $ k = 3$.
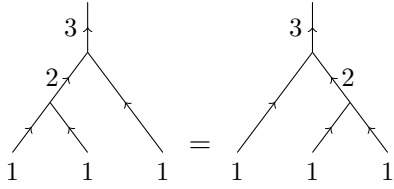
\begin{figure}
\centering
\begin{tikzpicture}[baseline]
\foreach \x/\y in {0/0,1/0,2/0,0/1,1/1,0/2} {
        \coordinate(z\x\y) at (\x+\y/2,\y/1.5);
}
\coordinate (z03) at (1,2);
\draw[mid>] (z00) node[below] {\small $1$} --  (z01);
\draw[mid>] (z01) -- node[left] {\small $2$} (z02);
\draw[mid>] (z10) node[below] {\small $1$} -- (z01);
\draw[mid>] (z20) node[below] {\small $1$} -- (z02);
\draw[mid>](z02) -- node[left] {\small $3$} (z03);
\end{tikzpicture}
 =
\begin{tikzpicture}[baseline]
\foreach \x/\y in {0/0,1/0,2/0,0/1,1/1,0/2} {
        \coordinate(z\x\y) at (\x+\y/2,\y/1.5);
}
\coordinate (z03) at (1,2);
\draw[mid>] (z00) node[below] {\small $1$} --  (z02);
\draw[mid>] (z10) node[below] {\small $1$} -- (z11);
\draw[mid>] (z20) node[below] {\small $1$} -- (z11);
\draw[mid>] (z11) -- node[right] {\small $2$} (z02);
\draw[mid>](z02) -- node[left] {\small $3$} (z03);
\end{tikzpicture}
\caption{The two equal forks $ (1,1,1) \rightarrow 3$}\label{fig:fork}
\end{figure}

By attaching forks to the bottom and top of a morphism we get maps
\begin{align*}
& u: \End_{\ASp(q)}(1^m) \rightarrow \Hom_{\ASp(q)}(\uk,\ul) \ \ {\rm  and } \\
&v: \Hom_{\ASp(q)}(\uk,\ul) \rightarrow \End_{\ASp(q)}(1^m)
\end{align*}
where $m=\sum_i k_i = \sum_i l_i$. Note that we have
\begin{equation}\label{eq:uv}
uv = \prod_i [k_i]! [l_i]! \cdot id
\end{equation}
since the composition (prong to prong) of two forks with $k$ prongs is $[k]!$ times the identity. We also get similar maps for $KConv$ instead of $\ASp$.

To simplify notation we denote
$$\Hom_{KConv(q)}(\uk,\ul) := \Hom_{KConv^{\A}(Gr)}(\uk,\ul) \otimes_{\C[q^\pm]} \C(q).$$
Now consider the following commutative diagram
\begin{equation}\label{eq:commutative}
\xymatrix{
\End_{\ASp(q)}(1^m) \ar[r]^{u} \ar[d]^{\Phi'} & \Hom_{\ASp(q)}(\uk,\ul) \ar[r]^{v} \ar[d]^{\Phi} & \End_{\ASp(q)}(1^m) \ar[d]^{\Phi'} \\
\End_{KConv(q)}(1^m) \ar[r]^{u} & \Hom_{KConv(q)}(\uk,\ul) \ar[r]^{v} & \End_{KConv(q)}(1^m)
}
\end{equation}
We assume that the left and right vertical maps $\Phi'$ (which are the same) are isomorphisms. Now consider $\alpha \in \Hom_{\ASp(q)}(\uk,\ul)$ and suppose $\Phi(\alpha) = 0$. Then $\Phi' \circ v (\alpha) = v \circ \Phi(\alpha) = 0$. Since $\Phi'$ is injective this means $v(\alpha)=0$. It follows from (\ref{eq:uv}) that $\alpha=0$. Thus the middle $\Phi$ is injective.  (Note that we are using here that we are over $ \C(q) $ so the product of the corresponding quantum integers is invertible.)

On the other hand, consider $\beta \in \Hom_{KConv(q)}(\uk,\ul)$. From (\ref{eq:uv}) we know that $u$ is surjective so choose $\beta'$ so that $u(\beta')=\beta$. Since $\Phi'$ is surjective take $\beta''$ so that $\Phi'(\beta'') = \beta'$. Then $\Phi \circ u(\beta'') = u \circ \Phi'(\beta'') = u(\beta') = \beta$ which means that $\Phi$ is also surjective. This completes the proof.
\end{proof}

\section{Algebraic study }\label{sec:technical}
In order to prove that the map
$$ \End_{\ASp(q)}(1^m) \rightarrow \End_{KConv^{\A}(Gr)}(1^m) \otimes_{\C[q^\pm]} \C(q) $$
is an isomorphism, we will need to undertake a more detailed study of each side.  We begin with $ \End_{\ASp(q)}(1^m)$ and in particular $ \End_{\ASp}(1^m)$ which we study using Corollary \ref{cor:GamEquiv} which tells us that
\begin{equation}\label{eq:C}
\End_{\ASp}(1^m) \cong \End_{\OSS}(\OS \otimes (\C^n)^{\otimes m}).
\end{equation}

\subsection{Endomorphisms in equivariant coherent sheaves}

Note that the right hand side of (\ref{eq:C}) is the space of $SL_n$-equivariant endomorphisms of the trivial bundle with fibre $ (\C^n)^{\otimes m}$. Thus we can identify
$$ \End_{\OSS}(\OS \otimes (\C^n)^{\otimes m}) = \Maps_{SL_n}(SL_n, \End((\C^n)^{\otimes m})) $$
where the right side is the space of $ SL_n$-equivariant maps of algebraic varieties from $ SL_n $ to $ \End(\C^n)^{\otimes m} $.

From Proposition \ref{pr:ABaction} at $ q = 1 $, we have an action of the affine symmetric group $ AS_m = \Z^m \rtimes S_m $ on $ \OS \otimes (\C^n)^{\otimes m} $ and thus a map $$ E[AS_m] \rightarrow \Maps_{SL_n}(SL_n, \End((\C^n)^{\otimes m})).$$

We will temporarily study a slight variant of this construction, involving
$$\Maps_{GL_n}(\gl_n, \End((\C^n)^{\otimes m})).$$

Let $\oE = \C[e_1,\dots,e_n] = \cO(\gl_n)^{GL_n} $, where $ e_k(A) $ is the $k$th elementary symmetric function in the eigenvalues of $ A $.  We have a surjection $ \oE \rightarrow E $ sending $e_n \mapsto 1$; this corresponds to the inclusion $ SL_n \rightarrow \gl_n $. Let $ AS_m^+$ be the monoid  $ \N^m \rtimes S_m $.  We write the generators of $ AS_m^+ $ as $ X_i, T_i $, where $ X_i = (0, \dots, 1, \dots 0) \in \N^m $ and $ T_i $ is the usual generator of $ S_m $.

We define a map
$$ \overline{E}[AS_m^+] \rightarrow \Maps_{GL_n}(\gl_n, \End((\C^n)^{\otimes m})) = \End_{\cO(\frgS)\mod}(\cO(\gl_n) \otimes (\C^n)^{\otimes m})) $$
by sending
\begin{itemize}
\item $e_k \in \oE$ to the map which takes $X \in \gl_n$ to multiplication by $e_k(X)$,
\item $X_i \in AS_m^+$ to the map which takes $A  $ to  $I \otimes \dots \otimes A \otimes \dots \otimes I$ (where $A$ occurs on the $i$th tensor factor),
\item $T_i$ to the constant map giving the endomorphism switching the $i$ and $i+1$ tensor factors.
\end{itemize}

This is very similar to the construction from Proposition \ref{pr:ABaction}, except that $ X_i $ are not acting invertibly, so the domain of this map involves the monoid $ AS_m^+$.

Because this construction matches the one from Proposition \ref{pr:ABaction}, the inclusion $ SL_n \rightarrow \gl_n$ gives rise to the commutative diagram
\begin{equation*}
\xymatrix{
\overline{E}[AS_m^+] \ar[r] \ar[d] & \Maps_{GL_n}(\gl_n, \End((\C^n)^{\otimes m})) \ar[d] \\
E[AS_m] \ar[r] & \Maps_{SL_n}(SL_n, \End((\C^n)^{\otimes m}))
}
\end{equation*}
 The following result (essentially due to Kraft-Procesi \cite{KP}) will be quite helpful for us.

\begin{Theorem} \label{th:AllSurjective}
In the above commutative diagram all the arrows are surjective.
\end{Theorem}

\begin{proof}
The surjectivity $ \oE[AS_m^+] \rightarrow E[AS_m^+] $ is obvious.

The surjectivity of
$$ \Maps_{GL_n}(\gl_n, \End((\C^n)^{\otimes m})) \rightarrow  \Maps_{SL_n}(SL_n, \End((\C^n)^{\otimes m})) $$
is a special case of Lemma 6.2 of \cite{KP}.   (Note that the actions of $ GL_n, SL_n$ factor through $ PGL_n $, so on both sides we can consider $ GL_n$-equivariant maps.)

The surjectivity
$$ \overline{E}[AS_m^+] \rightarrow \Maps_{GL_n}(\gl_n, \End((\C^n)^{\otimes m})) $$
is Proposition 6.5 of \cite{KP}.

Finally, the other three surjectivities imply that
$$E[AS_m] \rightarrow  \Maps_{SL_n}(SL_n, \End((\C^n)^{\otimes m}))$$
is surjective.
\end{proof}

\subsection{An algebraic model}
Define the ring
$$\oR_n^m := \oE[x_1, \dots, x_m]/(x_i^n - e_1 x_i^{n-1} + \cdots \pm e_{n-1} x_i \mp e_n)$$
and let $R_n^m = \oR_n^m \otimes_{\oE} E$. In other words, $ R_n^m $ is obtained from $ \oR_n^m $ by setting $e_n = 1$.  In the next section, we will see that $ R_n^m \cong K_{SL_n}((\P^{n-1})^m) $, the equivariant cohomology of the product of $ m $ copies of $ \P^{n-1} $.

Let $ \Sym(\ux) \subset \End_\oE(\oR_n^m) $ denote the $\oE$-subalgebra generated by multiplication by symmetric functions in $ x_1, \dots, x_m $.

\begin{Theorem} \label{th:ApplyKostant}
We have an isomorphism
$$\Maps_{GL_n}(\End(\C^n), \End((\C^n)^{\otimes m})) \xrightarrow{\sim} \End_{\Sym(\ux)}(\oR_n^m). $$
\end{Theorem}

\begin{proof}
Consider the inclusion $ s :Spec(\oE) \rightarrow \End(\C^n)$ given by
$$
s(e_1, \dots, e_n)  =  \begin{bmatrix}
0 & 0 & 0 & \dots & 0 & (-1)^{n-1} e_n \\
1 & 0 & 0 & \dots & 0 & (-1)^{n-2} e_{n-1} \\
0 & 1 & 0 & \dots & 0 & (-1)^{n-3} e_{n-2} \\
\dots & \dots & \dots & \dots & \dots & \dots \\
0 & 0 & 0 &\dots & 0 & -e_2 \\
0 & 0 & 0 &\dots & 1 & e_1
\end{bmatrix}.
$$
This is a variant of the Kostant section and gives us an isomorphism
$$ \End_{\OgS}(\cO(\gl_n) \otimes (\C^n)^{\otimes m} ) \rightarrow \End_\oZ( \oE \otimes_{\cO(\gl_n)} \cO(\gl_n) \otimes (\C^n)^{\otimes m})$$
where $ \oZ $ is the group scheme over $ Spec(\oE) $ whose fibre at a point $ e $ is the centralizer in $ GL_n $ of $ s(e) $.

If $ e \in Spec(\oE) $, then for any $k $, $ s(e)^k $ lies in the centralizer in $ \gl_n $ of $ s(e) $.  It is easy to see that $ I, s(e), \dots, s(e)^{n-1} $ gives a basis for this centralizer.  Thus for each $ k $, we define a map $ P_k : Spec(\oE) \rightarrow Lie(\oZ) $ taking $ e $ to $ s(e)^k$, where $ Lie(\oZ) $ denotes the total space bundle of Lie algebras over $ Spec(\oE) $ coming from the group scheme $ \oZ $.   Thus $ P_0, \dots, P_{n-1} $ give a basis for $ Lie(\oZ) $ as a Lie algebra over $ \oE $.  Hence considering endomorphisms of $\oE \otimes_{\cO(\gl_n)} \cO(\gl_n) \otimes (\C^n)^{\otimes m} $ over $ \oZ $ is the same as considering endomorphisms of this $ \oE$-module over the action of $ P_0, \dots, P_{n-1}$ .

For each $ k $, the action of  $ P_k $ on $ \oE \otimes_{\cO(\gl_n)} \cO(\gl_n) \otimes (\C^n)^{\otimes m}$ is given by $ \sum_{i=1}^{m} X_i^k $, where $ X_i $ is the endomorphism coming from $ X_i \in AS_m^+$ by the construction in the previous section.  Thus, $\End_\oZ( \oE \otimes_{\cO(\gl_n)} \cO(\gl_n) \otimes (\C^n)^{\otimes m})$ is the same thing as endomorphisms over the $ \oE $-algebra generated by symmetric functions in the $ X_1, \dots, X_m $.

We now identify $\oE \otimes_{\cO(\gl_n)} \cO(\gl_n) \otimes (\C^n)^{\otimes m}$ with $\oR_n^m$
 (both are free $\oE$-modules of rank $ n^m $) via the map
$$1 \otimes 1 \otimes  v_{k_1} \otimes \cdots \otimes v_{k_m} \mapsto x_1^{k_1-1} \cdots x_m^{k_m-1}
$$
where $ v_1, \dots, v_n $ denotes the standard basis for $ \C^n $.

We claim that under this identification, $ X_i $ corresponds to multiplication by $ x_i $.  To see this, let us assume for simplicity that $ m = 1$.  Then $ X $ as an element of $  \End_{\OgS}(\cO(\gl_n) \otimes \C^n ) $ is given by
$$ 1 \otimes v_k  \mapsto \phi_k $$
where $ \phi_k $ is the $\C^n $ valued function on $ \gl_n $ given by $ \phi_k(g) = g(v_k)$.
Now, if $ g = s(e_1, \dots, e_n)$, then we have
$$ g(v_k) =  \begin{cases} v_{k+1} \ & \text{ if } k \ne n  \\
(-1)^{n-1}e_n v_1 + (-1)^{n-2}e_{n-1} v_2 + \dots + e_1 v_n \  & \text{ if }  k = n
\end{cases} $$
Thus under the above identification $ \oE \otimes_{\cO(\gl_n)} \cO(\gl_n) \otimes \C^n \rightarrow \oR_n $, we see that the image of $ X $ is given by
$$ x^{k-1} \mapsto \begin{cases} x^k \ & \text{ if } k \ne n \\
(-1)^{n-1}e_n +(-1)^{n-2}e_{n-1} x + \dots + e_1 x^{n-1} \  & \text{ if } k = n
\end{cases}
$$
which matches multiplication by $ x $ acting on $ \oR_n$.

Combining all this, we see that we have an isomorphism
$$ \End_{\OgS}(\cO(\gl_n) \otimes (\C^n)^{\otimes m} ) \cong  \End_{\Sym(\ux)}(\oR_n^m) $$
as desired.
\end{proof}

As a consequence of Theorems \ref{th:AllSurjective} and \ref{th:ApplyKostant} we deduce the following.

\begin{Corollary}\label{cor:technical}
There is a surjective $ \oE$-algebra map
$$ \ochi: \oE[AS_m] \rightarrow \End_{\Sym(\ux)}(\oR_n^m) $$
such that $ \ochi(X_i) $ is multiplication by $ x_i $ and $\ochi(T_i) $ permutes the variables $x_i $ and $ x_{i+1} $.
\end{Corollary}

\subsection{A $q$-deformation} \label{sec:qDefR}
We now introduce a $ q$-deformation of the ring $ \oR_n^m $.

We define
$$ \oR_{n,[q^\pm]}^m := \oE[q^\pm][x_1, \dots, x_m]/ (x_i^n - e_1^{(i)} x_i^{n-1} + e_2^{(i)} x_i^{n-2} - \dots + (-1)^n e_n^{(i)}) $$
where the elements $ e_j^{(i)} $ are defined recursively by $ e_j^{(1)} = e_j $ and
$$ e_j^{(i)} := e_j^{(i-1)} + (q^{2}-1)[x_{i-1} e_{j-1}^{(i-1)} - x_{i-1}^{2} e_{j-2}^{(i-1)} + x_{i-1}^{3} e_{j-3}^{(i-1)} - \dots]. $$
This definition is motivated by computations in $KConv^{\A}(Gr)$ as we will explain in the next section. Again, we consider the $ \oE $ subalgebra $\Sym(\ux)$ of $\oR_{n,[q^\pm]}^m$ generated by multiplication by symmetric functions in the $x_i$. Similarly, we can define $ \oR_{n,(q)}^m $ by tensoring with $ \C(q)$, and we define $ R_{n,[q^\pm]}^m $ and $ R_{n,(q)}^m $ by setting $ e_n = 1$.

\begin{Remark}
It will follow from our later results that there is an isomorphism
$$ \End_{\OqSS}(\OqS \otimes (\C_q^n)^{\otimes m}) \cong \End_{\Sym(\ux)}(R_{n,(q)}^m) $$
but we do not know how to prove this directly.  We also suspect that
$$\End_{\cO_q(\frgS)}(\cO_q(\gl_n) \otimes (\C_q^n)^{\otimes m}) \cong \End_{\Sym(\ux)}(\oR_{n,(q)}^m)$$
where $ \cO_q(\gl_n) $ is the quantum matrix algebra as defined in, for example, Appendix A of \cite{BZ}.
\end{Remark}

Let $ AB_m^+ $ be the submonoid of $ AB_m $ generated by $X_i$, $ i = 1, \dots, m $ and $T_i, T_i^{-1} $ for $ i = 1, \dots, m-1 $.  This is useful for us since it allows us to consider ``representations'' of the annular braid group such that the $ X_i $ do not act invertibly.

We define the $\oE[q^\pm]$-linear map
\begin{equation}\label{eq:ochi}
\ochi: \oE[q^\pm][AB_m^+] \rightarrow \End_{\Sym(\ux)}(\oR_{n,[q^\pm]}^m)
\end{equation}
by taking $X_i$ to multiplication by $x_i$ and $T_i$ to the unique map which is
\begin{enumerate}
\item linear over $\oE[q^\pm][x_1, \dots, x_n]^{s_i}$ where $s_i$ acts by exchanging $x_i$ and $x_{i+1}$,
\item satisfies $T_i(1) = q $ and $T_i(x_i) = q^{-1} x_{i+1}$.
\end{enumerate}

\begin{Remark} \label{re:ABq1}
At $q=1$, the map $\ochi$ from (\ref{eq:ochi}) specializes to $\ochi$ from Corollary \ref{cor:technical}.
\end{Remark}

\begin{Proposition}
The map $\ochi$ from (\ref{eq:ochi}) is well defined.
\end{Proposition}
\begin{proof}
Although not trivial, one can check this directly by hand. Alternatively the result follows from Lemmas \ref{lem:ktheory} and \ref{lem:commutes}. More precisely, in Lemma \ref{lem:ktheory} we identify $R^m_{n,[q^\pm]}$ with $K^{\A}(Y(1^m))$ while Lemma \ref{lem:commutes} shows that there is an action of $E[q^\pm][AB_m]$ on $K^{\A}(Y(1^m))$ which satisfies properties (1) and (2) from above. It then follows that this action induces an action of $\oE[q^\pm][AB_m^+]$ on $\oR_{n,[q^\pm]}^m$ satisfying these same properties.
\end{proof}

\subsection{Surjectivity}

The reason for introducing $ \oE $ and $\oR_n^m$ (rather than working with $E$ and $R_n^m$) is because they carry an $\N$-grading. This means that they can be used to extend surjectivity from $ q = 1 $ to generic $ q $. This grading is defined by setting $ deg(e_j) = j, deg(q) = 0, \deg(x_i)=1$.

\begin{Proposition} \label{pr:ABsurjective}
The map $ \ochi: \oE(q)[AB_m^+] \rightarrow \End_{\Sym(\ux)}(\oR_{n,(q)}^m) $ is surjective. This means that the induced map $\chi: E(q)[AB_m] \rightarrow \End_{\Sym(\ux)}(R_{n,(q)}^m) $ is surjective.
\end{Proposition}
\begin{proof}
Note that $ \End_{\Sym(\ux)}(\oR_{n,[q^\pm]}^m)$ is a finitely generated graded $ \oE[q^\pm] $-module. Define an $\N$-grading on $\oE[AB_m^+]$ with $ deg(X_i) = 1 $ and $deg(T_i) = 0$. Then using Corollary \ref{cor:technical}, the map $\ochi: \oE[q^\pm][AB_m^+] \rightarrow \End_{\Sym(\ux)}(\oR_{n,[q^\pm]}^m) $ fits into the framework of Lemma \ref{lem:technical}. The result follows.
\end{proof}

\begin{Lemma}\label{lem:technical}
Suppose $f: A \rightarrow B$ is a map of graded $\oE[q^\pm]$-modules, and that $B$ is finitely generated over $\oE[q^\pm]$. If $f|_{q=1}$ is surjective then the induced map
$$f \otimes_{\C[q^\pm]} \C(q): A \otimes_{\C[q^\pm]} \C(q) \rightarrow B \otimes_{\C[q^\pm]} \C(q)$$
is surjective.
\end{Lemma}
\begin{proof}
Consider the cokernel $C := \coker(f)$. Since $f|_{q=1}$ is surjective this means $C \otimes_{\C[q^\pm]} \C_1 = 0$. On the other hand, $C$ is a finitely generated, graded $\oE[q^\pm]$-module so we can consider its support
$$\supp(C) \subset \Spec (\C[q^\pm]) \times \Proj(\oE) \xrightarrow{\pi} \Spec(\C[q^\pm]).$$
Since $\pi$ is proper this means that $\pi(\supp(C)) \subset \Spec (\C[q^\pm])$ is a closed subvariety. On the other hand, since $C \otimes_{\C[q^\pm]} \C_1 = 0$ this means that $\pi(\supp(C))$ does not contain the point $q=1$. Thus $\pi(\supp(C))$ must be the union of a finite number of points and thus $C \otimes_{\C[q^\pm]} \C(q) = 0$. This implies that $f \otimes_{\C[q^\pm]} \C(q)$ is surjective.
\end{proof}

\section{Geometric study}

We now study the right hand side of the isomorphism
$$ \End_{\ASp(q)}(1^m) \rightarrow \End_{KConv^{\A}(Gr)}(1^m) \otimes_{\C[q^\pm]} \C(q) $$
which means investigating the K-theory of the variety $Z(1^m,1^m)$.

\subsection{Property $T_A$}

The first step is to prove that $K^\A(Z(1^m,1^m))$ is a free $E[q^\pm]$-module (Corollary \ref{cor:free}). To do this we will prove that this space satisfies a certain property, called property $T_A$, as discussed below.

Let $A$ be an algebraic group and let $ R_A = K^A(\pt) $ be the representation ring of $ A $. Suppose that $X$ is an algebraic variety equipped with an action of $A$. We write $K_{i,top}^A(X)$ for the (higher) $A$-equivariant topological K-theory of $X$. Following \cite[Section 7]{Na1} we say that $X$ has property $T_A$ if
\begin{enumerate}
\item $K_{1,top}^A(X) = 0$ and $K_{top}^A(X) = K_{0,top}^A(X)$ is a free $R_A$-module,
\item the map $K^A(X) \rightarrow K^A_{top}(X)$ is an isomorphism,
\item for any closed algebraic subgroup $A' \subset A$ the $A'$-equivariant K-theories satisfy the two properties above and moreover the map $K^A(X) \otimes_{R_A} R_{A'} \rightarrow K^{A'}(X)$ is an isomorphism.
\end{enumerate}

Now consider a finite collection of $A$-invariant, locally closed subvarieties $\{X_i\}_{i \in I}$ of $X$. Suppose $\cup_{i \in I} X_i = X$ and that $I$ is equipped with a partial order $\le$ so that for each $i \in I$ the union $\cup_{j \le i} X_j$ is closed in $X$. Then we say $\{X_i\}_{i \in I}$ forms an $\alpha$-partition of $X$.

\begin{Lemma}\label{lem:TA1}\cite[Lem. 7.1.3]{Na1}
If $X$ has an $\alpha$-partition $X_1,\dots,X_k$ where each $X_i$ satisfies property $T_A$ then $X$ satisfies property $T_A$.
\end{Lemma}
\begin{Remark}
The actual definition of an $\alpha$-partition from \cite{Na1} assumes that $\le$ is a total ordering of $I$. However, the Lemma above holds because we can always refine a partial order to a total ordering. We use the definition above because in our case a particular partial order shows up and it does not seem natural to refine it arbitrarily.
\end{Remark}

\begin{Lemma} \label{lem:TA3}
Suppose that $ V, W $ are two $A$-equivariant vector bundles over $ X $ equipped with an equivariant vector bundle inclusion $ V \rightarrow W $.  If $ X $ has property $ T_A $, then so does $\P(W) \smallsetminus \P(V) $.
\end{Lemma}
\begin{Remark}
Our stratification of $Z(1^m,1^m)$ has pieces of the form $ \P(W) \smallsetminus \P(V)$ which is why we need this Lemma. Unfortunately this result does not appear in \cite{Na1} so we include a proof based on \cite{CG}.
\end{Remark}
\begin{proof}
Let us denote the inclusion $i: \P(V) \rightarrow \P(W)$ and the projections $\pi_1: \P(V) \rightarrow X$ and $\pi_2: \P(W) \rightarrow X$. First we show that the map $i_*: K_0^A(\P(V)) \rightarrow K_0^A(\P(W))$ is injective. By Theorem \ref{thm:TA} a general element $b \in K_0^A(\P(V))$ is of the form
$$b = \sum_{j=0}^{\rk(V)-1} [\pi_1^*(a_j)] \cdot [\cO_{\P(V)}(-j)]$$
for some $a_j \in K_0^A(X)$. Now suppose $i_*(b)=0$. Then by the projection formula together with the fact that $\pi_1^*(a_j) = i^* \pi_2^*(a_j)$ we get
$$\sum_{j=0}^{\rk(V)-1} [\pi_2^*(a_j)] \cdot [i_* \cO_{\P(V)}(-j)] = 0.$$
On the other hand, we have the Koszul resolution
\begin{equation}\label{eq:koszul}
\dots \rightarrow \cO_{\P(W)}(-2) \otimes \pi_2^* {\wedge^2 Q} \rightarrow \cO_{\P(W)}(-1) \otimes \pi_2^* Q \rightarrow \cO_{\P(W)} \rightarrow \cO_{\P(V)}
\end{equation}
where $Q = (W/V)^\vee$. Substituting this into the relation above we get
$$\sum_{j=0}^{\rk(V)-1} [\pi_2^*(a_j)] \cdot \left( \sum_{k=0}^{\rk(W)-\rk(V)} (-1)^k [\cO_{\P(W)}(-j-k)] \cdot [\pi_2^* {\wedge^k Q}] \right) = 0$$
or equivalently
$$\sum_{l=0}^{\rk(W)-1} [\cO_{\P(W)}(-l)] \cdot \left( \sum_{j+k=l} (-1)^k [\pi_2^*(a_j \otimes {\wedge^k Q})] \right) = 0.$$
Since $ [\cO_{\P(W)}(-l)] $ are linearly independent for $ l= 0, \dots, \rk(W) -1 $, we see that $ \sum_{j+k=l} (-1)^k [\pi_2^*(a_j \otimes {\wedge^k Q})] = 0 $ for all $ l $.  From this one can show inductively that $\pi_2^*(a_j)=0$ for all $j$ (starting with $j=0$). Hence $b=0$ which shows that $i_*$ is injective.

The same argument shows that $i_*$ is also injective on topological K-theory. Now, by Corollary \ref{cor:TA4} we know that $K_{1,top}^A(\P(V)) = K_{1,top}^A(\P(W)) = 0$. It follows from the long exact sequence
$$K_{1,top}^A(\P(W)) \rightarrow K_{1,top}^A(\P(W) \smallsetminus \P(V)) \rightarrow K_{0,top}^A(\P(V)) \xrightarrow{i_*} K_{0,top}^A(\P(W))$$
that $K_{1,top}^A(\P(W) \smallsetminus \P(V)) = 0$.

Next, consider the commutative diagram
\begin{equation*}
\xymatrix{
K_0^A(\P(V)) \ar[r]^{i_*} \ar[d] & K_0^A(\P(W)) \ar[r] \ar[d] & K_0^A(\P(W) \smallsetminus \P(V)) \ar[r] \ar[d] & 0 \\
K_{0,top}^A(\P(V)) \ar[r]^{i_*} & K_{0,top}^A(\P(W)) \ar[r] & K_{0,top}^A(\P(W) \smallsetminus \P(V)) \ar[r] & 0
}
\end{equation*}
The first two vertical maps are isomorphisms by Corollary \ref{cor:TA4}. It follows that the right most vertical map is also an isomorphism.

Finally, it remains to show that $K_{0,top}^A(\P(W) \smallsetminus \P(V))$ is a free $R_A$-module.  Consider the short exact sequence
$$0 \rightarrow K_{0,top}^A(\P(V)) \xrightarrow{i_*} K_{0,top}^A(\P(W)) \rightarrow K_{0,top}^A(\P(W) \smallsetminus \P(V)) \rightarrow 0.$$
The image of $i_*$ are elements of the form $b \cdot [i_* \cO_{\P(V)}]$ for $b \in K_{0,top}^A(\P(W))$. Using the Koszul resolution (\ref{eq:koszul}) for $\cO_{\P(V)}$ it is not hard to see that this means $K_{0,top}^A(\P(W))$ is generated, as a $K_{0,top}^A(X)$-module, by the image of $i_*$ and by $\cO_{\P(W)}(-k)$ for $k= \rk(V), \dots, \rk(W)-1$. Since $K_{0,top}^A(\P(V))$ and $K_{0,top}^A(\P(W))$ are free $K_{0,top}^A(X)$-modules of rank $\rk(V)$ and $\rk(W)$ respectively this means that $K_{0,top}^A(\P(W) \smallsetminus \P(V))$ is freely generated over $K_{0,top}^A(X)$ by $\cO_{\P(W)}(-k)$ for $k= \rk(V), \dots, \rk(W)-1$. This completes the proof.
\end{proof}

\begin{Theorem}\cite[Theorem 5.2.31]{CG}\label{thm:TA}
Consider an $A$-equivariant vector bundle $E \rightarrow X$ and let $\pi: \P(E) \rightarrow X$ be the associated vector bundle. Then for $j \ge 0$ the groups $K_j^A(\P(E))$ are freely generated over $K_j^A(X)$ by the classes $[\cO_{\P(E)}(-k)]$, $k=0,1, \dots, \rk(E)-1$.
\end{Theorem}
\begin{Remark} The actual statement in \cite{CG} uses $[\cO_{\P(E)}(k)]$ but one could equally well use $[\cO_{\P(E)}(-k)]$.
\end{Remark}

The argument in \cite{CG} can also be used to prove the same result for topological K-theory. The following is then immediate.

\begin{Corollary}\label{cor:TA4}
Suppose $E \rightarrow X$ is an $A$-equivariant vector bundle on $X$. If $X$ satisfies property $T_A$ then $\P(E)$ satisfies property $T_A$.
\end{Corollary}

\subsection{Property $T_A$ for $Z(1^m,1^m)$}\label{sec:TA}

Our current goal is to establish the following result.

\begin{Theorem}\label{prop:TA}
The varieties $Y(1^m)$ and $Z(1^m,1^m)$ satisfy property $T_\A$.
\end{Theorem}

For $Y(1^m)$ this is immediate from Lemma \ref{lem:TA3}, since it is an iterated bundle of projective spaces. For $Z(1^m, 1^m)$ the proof is more involved and will occupy the rest of this section.

The main idea is to partition the variety $ Z = Z(1^m, 1^m)$ into pieces and then apply Lemma \ref{lem:TA3}. This partition is a special case of the one considered in section 4.1 of \cite{FKK} which in turn was inspired by the proof of Theorem 3.1 from \cite{H}.

Let $ T $ be a standard Young tableau with $ m $ boxes and at most $ n $ rows.  Then we write $ \lambda^{(i)}(T) = T|_{1,\dots, i} $ for the shape made by using only the boxes filled with $ 1, \dots, i $.  We regard $ \lambda^{(i)}(T) $ as an element of $ \Lambda_+ $.

Let $ \sP $ be the set of pairs $ (T,T') $ of standard Young tableaux (SYT) of the same shape, each with $m $ boxes and at most $ n $ rows.  We define a partial order on $ \sP $ by
$$ (U, U') \le (T, T') \text{ if, for all $ i$, } \lambda^{(i)}(U) \le \lambda^{(i)}(T) \text{ and } \lambda^{(i)}(U') \le \lambda^{(i)}(T')
$$
With this partial order, the pair $ (T_0, T_0) $ where $ T_0 $ is the unique SYT with one row and $ m $ boxes, is the maximal element. We define
$$ Z(T, T') = \{ (L_\bullet, L'_\bullet) \in Z : L_i \in Gr^{\lambda^{(i)}(T)}, L'_i \in Gr^{\lambda^{(i)}(T')}, \text{ for } i = 1, \dots, m \} $$
where we recall (from section \ref{sec:varieties}) that for $ \lambda = (\lambda_1, \dots, \lambda_n) \in \Lambda_+$,  $ Gr^\lambda $ consists of those lattices such that $ z|_{L/L_0} $ has Jordan type $ \lambda $.

\begin{Proposition}
The closures $\{ \overline{Z(T,T')} \}_{(T,T') \in \sP} $ are the irreducible components of $ Z $. Moreover, the collection $\{Z(T,T') \}_{(T,T') \in \sP} $ forms an $\alpha$-partition of $ Z $.
\end{Proposition}
\begin{proof}
The first claim is a special case of Theorem 4.1 of \cite{FKK}. To see the second claim it is clear that each $ Z(T,T') $ is $\A$-invariant and that they give a partition of $ Z $. From the closure relation among the strata of the affine Grassmannian, we see that, for each $ (T,T') $, we have
$$\overline{Z(T,T')} \subset \bigcup_{(U,U') \le (T,T')} Z(U,U').$$
This implies the necessary closedness property for having an $\alpha$-partition.
\end{proof}

Theorem \ref{prop:TA} now follows from the Proposition above and the following result.

\begin{Lemma}\label{lem:Z(T)}
For each $ (T,T') \in \sP$, $ Z(T, T') $ satisfies property $T_\A $.
\end{Lemma}
\begin{proof}
For each $k = 1, \dots, 2m-1$, we define varieties $Z(T,T')_k$ which contain only part of the data of $Z(T,T')$. Increasing $k$ remembers more data about the flags $(L_\bullet,L_\bullet')$. Roughly speaking, we remember the lattices $L_i$ one by one starting with $L_1$ and working up until $L_m$ ($k=m$). We then remember the $L_i'$ starting with $L'_{m-1}$ and working back down to $L'_1$.

More precisely, for $k = 1, \dots, m$ we define
$$ Z(T, T')_k = \{ (L_1, \dots, L_k) : L_0 \subset L_1 \subset \cdots \subset L_k, zL_i \subset L_{i-1}, L_i \in Gr^{\lambda^{(i)}(T)} \text{ for }  i \}.$$
On the other hand, for $k = m, \dots, 2m-1$, we define $Z(T,T')_k$ as
\begin{align*}
\{ &(L_1, \dots, L_m), (L'_{2m-k}, \dots, L'_m): L_m = L_m' \text{ and } \\
& L_{i-1} \subset L_i, \dim L_i/L_{i-1} = 1, zL_i \subset L_{i-1}, L_i \in Gr^{\lambda^{(i)}(T)} \text{ for } i = 1, \dots ,m \\
& L_{i-1}' \subset L'_i, \dim L'_i/L'_{i-1}=1, zL'_i \subset L'_{i-1}, L'_i \in Gr^{\lambda^{(i)}(T')} \text{ for } i = 2m-k+1, \dots, m \}
\end{align*}
Notice that $Z(T,T')=Z(T,T')_{2m-1}$.

We will now prove that the Lemma by induction on $k$. Let us suppose the conclusions of the Lemma are true for $ Z(T,T')_{k-1}$. We have a map $\pi_k: Z(T,T')_k \rightarrow Z(T,T')_{k-1} $ given by forgetting the last piece of data. We claim that this map satisfies the hypotheses of Lemma \ref{lem:TA3}.

To see this, we examine the fibres of $\pi_k$. If $ k \le m $ the the fibre over of point $ L_\bullet \in Z(T, T')_{k-1} $ is
$$\pi_k^{-1}(L_\bullet) = \{ L_{k-1} \subset L_k \subset z^{-1} L_{k-1}, z|_{L_i/L_0} \text{ has Jordan type } \lambda^{(k)}(T) \}.$$
Now the condition on the Jordan type of $ z|_{L_i/L_0} $ is equivalent to a condition on the intersections of $ L_k $ with the subspaces $ z^{-j}L_0 $ for various $ j$.  More precisely, if the box labelled $k$ is on the $ j$th column of $ T $, then the condition on the Jordan type is equivalent to the condition that
$$ \dim (L_k \cap z^{-j}L_0) = \dim (L_{k-1} \cap z^{-j} L_0) + 1, \text{ and } L_k \cap z^{-(j-1)}L_0 = L_{k-1} \cap z^{-(j-1)} L_0. $$
Thus we can identify the fibre with the space
$$ \P( (z^{-1} L_{k-1} \cap z^{-j}L_0) / (L_{k-1} \cap z^{-j} L_0) ) \smallsetminus \P( (z^{-1} L_{k-1} \cap z^{-(j-1)}L_0) / (L_{k-1} \cap z^{-(j-1)} L_0))
$$
and thus the hypotheses of Lemma \ref{lem:TA3} are satisfied.  The case $ k > m $ is similar. This proves that $ Z(T,T')_k $ satisfies property $ T_\A $.
\end{proof}

\begin{Corollary}\label{cor:free}
$K^\A(Z(1^m,1^m))$ is a free $E[q^\pm]$-module.
\end{Corollary}

\subsection{Localization}

Suppose $T$ is an abelian reductive group and $Y$ an algebraic variety equipped with an action of $T$. As before, we denote $R_T = K^T(\pt)$ and by $F_T$ the fraction field of $R_T$. For $t \in T$ we denote by $Y^t$ the fixed point locus of $t$.

It was proved by Thomason \cite{Th1,Th2} that the inclusion $i^t: Y^t \rightarrow Y$ induces an isomorphism
$$i^t_*: K^T(Y^t) \otimes_{R_T} F_T \xrightarrow{\sim} K^T(Y) \otimes_{R_T} F_T.$$

\begin{Lemma}\label{lem:ZYY}
If $i: Z(1^m,1^m) \rightarrow Y(1^m) \times Y(1^m)$ is the natural inclusion then $i_*: K^\A(Z(1^m,1^m)) \rightarrow K^\A(Y(1^m) \times Y(1^m))$ is injective.
\end{Lemma}
\begin{proof}
We write $Y \times Y$ for $Y(1^m) \times Y(1^m)$ and $Z$ for $Z(1^m,1^m)$. Choose $T \subset A := \A$ to be the maximal torus in $SL_n$ cross $\C^\times$ and $t$ a generic element in $T$. This means that $(Y \times Y)^t$ consists of a finite number of isolated points.

Next, consider the commutative diagram
$$\xymatrix{
K^T(Z^t) \otimes_{R_T} F_T \ar[d] \ar[rr]^{i^t_*} & & K^T((Y \times Y)^t) \otimes_{R_T} F_T \ar[d] \\
K^T(Z) \otimes_{R_T} F_T \ar[rr]^{i_*} & & K^T(Y \times Y) \otimes_{R_T} F_T }$$
where the vertical maps are the natural inclusions. The vertical maps are isomorphisms by Thomason's result while $i^t_*$ is induced by an inclusion of isolated fixed points and thus injective. This means that the map
\begin{equation}\label{eq:5}
i_*: K^T(Z) \otimes_{R_T} F_T \rightarrow K^T(Y \times Y) \otimes_{R_T} F_T
\end{equation}
is injective.

On the other hand, since $Z$ and $Y \times Y$ have property $T_A$ both $K^T(Z)$ and $K^T(Y \times Y)$ are free $R_T$-modules. Thus we have the commutative diagram
$$\xymatrix{
K^T(Z) \ar[r] \ar[d] & K^T(Y \times Y) \ar[d] \\
K^T(Z) \otimes_{R_T} F_T \ar[r] & K^T(Y \times Y) \otimes_{R_T} F_T }$$
where the vertical maps are injective since $R_T$ embeds into $F_T$ and the bottom horizontal map is injective by (\ref{eq:5}). This means that the map
\begin{equation}\label{eq:7}
i_*: K^T(Z) \rightarrow  K^T(Y \times Y)
\end{equation}
is injective. But $K^T(Z) \cong K^A(Z) \otimes_{R_A} R_T$ and likewise for $K^T(Y \times Y)$. Since $R_T$ is a free $R_A$ module this implies that $i_*: K^A(Z) \rightarrow K^A(Y \times Y)$ is injective as a consequence of the injectivity of (\ref{eq:7}).
\end{proof}

\subsection{K-theory of $Y(1^m)$}
Recall that $K^\A(\pt) \cong E[q^\pm]$ where $e_j \in E$ denotes $\wedge^{n-j}(\C^n)$ and $-q^{-1}$ keeps track of the shift $\{1\}$ (c.f. Remark \ref{rem:q}). Note that $e_n=1$. Since $Y(1^m)$ is an iterated $\P^{n-1}$-bundle $K^\A(Y(1^m))$ is generated, over $E[q^\pm]$, by $x_1, \dots, x_m$ where $x_i$ is the line bundle $[(L_i/L_{i-1})^\vee]$. We will also denote $e_j^{(i)} := [\Alt{j}(z^{-1}L_{i-1}/L_{i-1})^\vee]$. Note that $e_j^{(1)} = e_j$ and, by convention, $e_0^{(i)} = 1$ and $e_j^{(i)}=0$ if $j < 0$.

\begin{Lemma}\label{lem:5}
The $e_j^{(i)}$ are related to each other via the relations
\begin{align}
\label{eq:e1} & e_j^{(i+1)} = e_j^{(i)} + (q^{2}-1)[x_{i} e_{j-1}^{(i)} - x_{i}^{2} e_{j-2}^{(i)} + x_{i}^{3} e_{j-3}^{(i)} - \dots] \\
\label{eq:e2} & e_j^{(i+1)} + x_i e_{j-1}^{(i+1)} = e_j^{(i)} + q^{2} x_i e_{j-1}^{(i)}.
\end{align}
\end{Lemma}
\begin{proof}
We will prove this by induction on $j$. The base case $j=0$ is obvious.

Now, let us denote $V_i = (z^{-1}L_{i-1}/L_{i-1})^\vee$. From the following two standard exact sequences
$$0 \rightarrow L_i/L_{i-1} \rightarrow z^{-1}L_{i-1}/L_{i-1} \rightarrow z^{-1}L_{i-1}/L_i \rightarrow 0$$
$$0 \rightarrow z^{-1}L_{i-1}/L_i \rightarrow z^{-1}L_i/L_i \rightarrow z^{-1}L_i/z^{-1}L_{i-1} \rightarrow 0$$
we get that
$$[V_i] = x_i + [V_{i+1}] - [(z^{-1}L_i/z^{-1}L_{i-1})^\vee].$$
On the other hand, we have the isomorphism
$$z^{-1}L_i/z^{-1}L_{i-1} \xrightarrow{\sim} L_i/L_{i-1} \{2\}$$
induced by multiplication by $z$ (the $\{2\}$ shift is for the same reason as in (\ref{eq:9})). This implies that $[(z^{-1}L_i/z^{-1}L_{i-1})^\vee] = q^{2} x_i$ and so we get
$$[V_{i+1}] + x_i = [V_i] + q^{2} x_i.$$

Applying $\Alt{j}$ we get
$$[\Alt{j} V_{i+1}] + x_i [\Alt{j-1} V_{i+1}] = [\Alt{j} V_i] + q^{2} x_i [\Alt{j-1} V_i]$$
which is relation (\ref{eq:e2}).

Now, we can rewrite (\ref{eq:e1}) as
$$e_j^{(i+1)} = e_j^{(i)} + (q^2-1)x_{i}e_{j-1}^{(i)} - x_{i}(q^2-1)(x_{i}e_{j-2}^{(i)} - x_{i}^2e_{j-3}^{(i-1)} + \dots).$$
Assuming (\ref{eq:e1}) is true for $j-1$ then we know
$$(q^2-1)(x_{i}e_{j-2}^{(i)} - x_{i}^2e_{j-3}^{(i-1)} + \dots) = e_{j-1}^{(i+1)} - e_{j-1}^{(i)}.$$
So it remains to show that
$$e_j^{(i+1)} = e_j^{(i)} + (q^2-1)x_{i}e_{j-1}^{(i)} - x_{i}(e_{j-1}^{(i+1)} - e_{j-1}^{(i)}).$$
This easily simplifies to equation (\ref{eq:e2}) which we proved above.
\end{proof}

Recall the definition of the $E[q^\pm]$-algebra $ R^m_{n,[q^\pm]}$ from section \ref{sec:qDefR}.

\begin{Lemma}\label{lem:ktheory}
We have $K^\A(Y(1^m)) \cong R^m_{n,[q^\pm]}$ as $E[q^\pm]$-modules.
\end{Lemma}
\begin{proof}
Inside $K^{\A}(Y(1^m))$ we have the standard exact sequence
$$0 \rightarrow L_i/L_{i-1} \rightarrow z^{-1}L_{i-1}/L_{i-1} \rightarrow z^{-1}L_{i-1}/L_i \rightarrow 0$$
which implies that $\Alt{n} [(z^{-1}L_{i-1}/L_{i-1})^\vee - (L_i/L_{i-1})^\vee] = 0$. The defining relation of $R^m_{n,[q^\pm]}$ now follows using that $e_j^{(i)} = [\Alt{j}(z^{-1}L_{i-1}/L_{i-1})^\vee]$, $x_i = [\det(L_i/L_{i-1})^\vee]$ and the general identity
$$\Alt{r}(A-B) = \Alt{r}(A) - \Alt{r-1}(A) \cdot B + \Alt{r-2}(A) \cdot \Sym^2(B) - \Alt{r-3}(A) \cdot \Sym^3(B) + \dots$$
in K-theory. There are no further relations since $Y(1^m)$ is an iterated $\P^{n-1}$-bundle and
$$K^\A(Y(1)) = K^\A(\P^{n-1}) \cong E[q^\pm][x]/(x^n - e_1x^{n-1} + e_2 x^{n-2} - \dots + (-1)^n)$$
where $x = \cO_{\P^{n-1}}(1)$.
\end{proof}

Recall the $ E[q^\pm]$-algebra $\Sym(\ux)$ consisting of symmetric functions in the $x_i$ which acts on $ R^m_{n,[q^\pm]} $ as discussed in section \ref{sec:qDefR}.

\begin{Corollary} \label{cor:ZYY}
The natural inclusion $i: Z(1^m,1^m) \rightarrow Y(1^m) \times Y(1^m)$ induces an injective map
$$ S: K^\A(Z(1^m,1^m)) \rightarrow \End_{\Sym(\ux)}(R^m_{n,[q^\pm]}).$$
\end{Corollary}
\begin{proof}
Since $Y(1^m) \times Y(1^m)$ is an iterated projective bundle it has property $T_A$ where $A=\A$. It follows that
\begin{align*}
K^\A(Y(1^m) \times Y(1^m))
&\cong K^\A_{0,top}(Y(1^m) \times Y(1^m)) \\
&\cong \End_{K^\A_{0,top}(pt)}(K^\A_{0,top}(Y(1^m)) \\
&\cong \End_{K^\A(pt)}(K^\A(Y(1^m)).
\end{align*}
Subsequently, by Lemma \ref{lem:ZYY}, we have an inclusion
$$ K^\A(Z(1^m,1^m)) \rightarrow \End(K^\A(Y(1^m)))$$
while, by the previous Lemma, $K^\A(Y(1^m)) \cong R^m_{n,[q^\pm]}$. On the other hand, since $x_i = [\det(L_i/L_{i-1})^\vee]$, we see that the symmetric functions in the $x_i$ are generated by $\wedge^j(L_m/L_0)^\vee$. These clearly commute with any kernel supported on $Z(1^m,1^m)$ since $Z(1^m,1^m)$ is defined by the condition $L_m=L_m'$. The result follows.
\end{proof}

\section{Isomorphism for $1^m$}\label{sec:iso1m}

We are now in a position to prove the isomorphism for $ 1^m $.
\begin{Theorem}\label{thm:Isoq1}
The functor $\Phi$ gives an isomorphism
\begin{equation}\label{eq:ASpZ}
\End_{\ASp}(1^m) \xrightarrow{\sim} K^{SL_n}(Z(1^m,1^m))
\end{equation}
making the following diagram of isomorphisms commute
\begin{equation}
\xymatrix{
\End_{\ASp}(1^m) \ar^{\Phi}[r] \ar^{A\Gamma}[d] & K^{SL_n}(Z(1^m, 1^m)) \ar^S[d] \\
\End_{\OSS}(\OS \otimes (\C^n)^{\otimes m}) \ar[r]^>>>>>>{\gamma} & \End_{\Sym(\ux)}(R^m_n)
}
\end{equation}
where $\gamma$ comes from setting $e_n = 1$ in Theorem \ref{th:ApplyKostant}.
\end{Theorem}

\begin{Theorem} \label{thm:2}
The functor $ \Phi $ also gives an isomophism
$$ \End_{\ASp(q)}(1^m) \xrightarrow{\sim} K^\A(Z(1^m, 1^m)) \otimes_{\C[q^\pm]} \C(q) $$
\end{Theorem}

\subsection{The big diagram}
The proof relies on the following diagram of $ E[q^\pm]$-algebras
\begin{equation} \label{eq:BigDiagram}
\xymatrix{
E[q^\pm][AB_m] \ar^\beta[r] \ar^{\chi}[d]& \End_{ASp_n[q^\pm]}(1^m) \ar^{\Phi}[d]  \\
\End_{\Sym(\ux)}(R_{n,[q^\pm]}^m) & \ar_{S}[l]  K^\A(Z(1^m,1^m))
}
\end{equation}
where $\chi$ is obtained (by setting $e_n=1$) from $\ochi$ which was defined in section \ref{sec:qDefR}.

\begin{Lemma}\label{lem:commutes}
This diagram commutes.
\end{Lemma}

\begin{proof}
It suffices to check this on the generators $ X_i, T_i $ of $ AB_m $.

By definition $\Phi(\Beta(X_i)) = [\Delta_* \det(L_i/L_{i-1})^\vee] $ and, following the isomorphism from Lemma \ref{lem:ktheory}, $S(\Delta_* \det(L_i/L_{i-1})^\vee)$ acts on $R_{n, [q^\pm]}^m$ as multiplication by $x_i$. On the other hand, $\chi(X_i)$ is by definition multiplication by $x_i$.

Now consider the correspondence
$$Y(1^{i-1},2,1^{m-i-1}) \xleftarrow{\pi} W = \{(L_0 \subset \dots \subset L_m: zL_{i+1} \subset L_{i-1}\} \xrightarrow{i} Y(1^m)$$
where $\pi$ is the projection which forgets $L_i$ and $i$ is an inclusion. Recall that $(\Phi \circ \beta)(T_i) = q \cdot id - E_i F_i$ (c.f. Remark \ref{rem:T}) where $E_i$ and $F_i$ act on $K^{\A}(Y(1^m)) \cong R^m_{n,[q^\pm]}$ as follows
\begin{align*}
E_i(\cdot) &= -q^{-1} [i_* \pi^* (\cdot) \otimes \det(L_i/L_{i-1})] \\
F_i(\cdot) &= \pi_* i^* ((\cdot) \otimes \det(L_{i+1}/L_i)^{-1}).
\end{align*}
From this description it is clear that for any $y \in R^m_{n,[q^\pm]}$ we have
\begin{align*}
(\Phi \circ \beta) (T_i)(y x_j) &= x_j (\Phi \circ \beta) (T_i) (y) \ \text{ if } j \ne i,i+1, \\
(\Phi \circ \beta) (T_i)((x_i+x_{i+1})y) &=  (x_i+x_{i+1}) (\Phi \circ \beta) (T_i)(y) \\
(\Phi \circ \beta) (T_i)(x_ix_{i+1}y) &= x_i x_{i+1} (\Phi \circ \beta)(T_i)(y).
\end{align*}
Since the same is true for $\chi$ in place of $(\Phi \circ \beta)$ it suffices to show that the actions of $(\Phi \circ \beta)(T_i)$ and $\chi(T_i)$ agree on $1$ and $x_i$.

Now $F_i(1) = 0$ which means that $(\Phi \circ \beta)(T_i)(1) = q = \chi(T_i)(1)$.

On the other hand, $E_i F_i(x_i) = -q^{-1} x_{i+1} [\cO_W]$ so we need to determine the class of $\cO_W$ inside $K^\A(Y(1^m))$. To do this consider the map of bundles on $Y(1^m)$
$$z: L_{i+1}/L_i \rightarrow L_i/L_{i-1} \{2\}$$
where the shift by $\{2\}$ is because of the $\C^\times$ action on $z$ (c.f. the map in (\ref{eq:9})). This map vanishes exactly along $W$ and hence $[\cO_W] = 1 - q^2 x_i x_{i+1}^{-1}$. Putting this together we find that
\begin{align*}
(\Phi \circ \beta)(T_i)(x_i)
&= q x_i - E_i F_i (x_i) \\
&= qx_i + q^{-1} x_{i+1} (1 - q^{2} x_i x_{i+1}^{-1}) \\
&= q^{-1} x_{i+1} = \chi(T_i)(x_i).
\end{align*}
This concludes the proof.
\end{proof}

\subsection{The proofs}

\begin{proof}[Proof of Theorem \ref{thm:Isoq1}]
We take the diagram (\ref{eq:BigDiagram}) and tensor over $\C[q^\pm]$ with $ \C $ (at $ q=1$) to obtain
\begin{equation*}
\xymatrix{
E[AB_m] \ar^{\beta}[r] \ar^{\chi}[d]& \End_{ASp_n}(1^m) \ar^{\Phi}[d]  \\
\End_{\Sym(\ux)}(R_{n}^m) & \ar[l]^>>>>>{S} K^{SL_n}(Z(1^m,1^m))
}
\end{equation*}
Now by Theorem \ref{th:ApplyKostant} and Theorem \ref{cor:GamEquiv}, we obtain isomorphisms
$$  \End_{ASp_n}(1^m) \xrightarrow{A\Gamma}  \End_{\OSS}(\OS \otimes (\C^n)^{\otimes m}) \xrightarrow{\sim} \End_{\Sym(\ux)}(R_{n}^m) $$
so that our diagram becomes
\begin{equation}
\xymatrix{
E[AB_m] \ar^\beta[r] \ar^{\chi}[d]& \End_{ASp_n}(1^m) \ar^{\Phi}[d] \ar[dl]^{A\Gamma}  \\
\End_{\Sym(\ux)}(R_{n}^m) & \ar[l]^>>>>>{S}  K^{SL_n}(Z(1^m,1^m))
}
\end{equation}
Now the top left triangle commutes by Remark \ref{re:ABq1}.  Then the commutativity of the bottom right triangle follows from the commutativity of the whole square and the surjectivity of $ \beta $. Note that the bottom right triangle is rewritten as a square in the statement of the theorem.

Since $ A\Gamma $ is an isomorphism and $ S $ is injective, we see that $\Phi $ is an isomorphism.
\end{proof}

\begin{proof}[Proof of Theorem \ref{thm:2}]
We take the diagram (\ref{eq:BigDiagram}) and tensor over $\C[q^\pm]$ with $ \C(q) $ to obtain
\begin{equation*}
\xymatrix{
E(q)[AB_m] \ar^\beta[r] \ar^{\chi}[d]& \End_{ASp_n(q)}(1^m) \ar^{\Phi}[d]  \\
\End_{\Sym(\ux)}(R_{n,(q)}^m) & \ar[l]^>>>>>>{S} K^\A(Z(1^m,1^m)) \otimes_{\C[q^\pm]} \C(q)
}
\end{equation*}

From Proposition \ref{pr:ABsurjective} we see that $\chi$ is surjective.  On the other hand, from Corollary \ref{cor:ZYY} we see that $S$ is injective. Thus $\Phi$ is surjective.

Now we will prove that $ \Phi $ is injective. Let $ \ul, \ul' $ be two sequences of length $ m' $ which are each obtained from $ 1^m $ by adding some $0$s and $n$s. Consider the map
\begin{equation}\label{eq:temp}
\Phi_{m'} : \Hom_{(\dU_{[q^\pm]} L \gl_{m'})^n}(\ul, \ul') \rightarrow K^{\A}(Z(1^m, 1^m))
\end{equation}
and recall that this map factors as
$$\Hom_{(\dU_{[q^\pm]} L \gl_{m'})^n}(\ul, \ul') \xrightarrow{\Psi_{m'}} \End_{\ASp[q^\pm]}(1^m) \xrightarrow{\Phi} K^{\A}(Z(1^m, 1^m))$$
where the first map is injective by Theorem \ref{thm:1}.

Now $\Phi$ is injective at $ q = 1$ by Theorem \ref{thm:Isoq1} and hence $ \Phi_{m'}$ is injective at $ q = 1 $.  Since both sides of (\ref{eq:temp}) are free $\C[q^\pm]$-modules (the left hand side by Lemma \ref{lem:free1} and the right hand side by Corollary \ref{cor:free}), Lemma \ref{le:FreeInjective} implies that $ \Phi_{m'} $ is injective.

Since, by Theorem \ref{thm:1}, every element of $ \End_{\ASp[q^\pm]}(1^m) $ is in the image of $ \Psi_{m'} $ for some $ m', \ul, \ul' $ this implies that $ \Phi$ is injective.
\end{proof}

\begin{Lemma} \label{le:FreeInjective}
Let $ f : M \rightarrow N $ be a map between free $ \C[q^\pm] $-modules such that $ f|_{q=1} $ is injective.  Then $ f $ is injective.
\end{Lemma}

\begin{proof}
Since the submodule of a free $\C[q^\pm] $-module is free, we can assume that $ f $ is surjective (by replacing $ N $ with the image of $ f $).  Thus we have a short exact sequence
$$
0 \rightarrow K \rightarrow M \xrightarrow{f} N \rightarrow 0
$$
where $ K $ is the kernel of $ f $.  Applying $ \otimes_{\C[q^\pm]} \C_1 $, we obtain the long exact sequence
$$
Tor_1^{\C[q^\pm]}(N, \C_1) \rightarrow K \otimes_{\C[q^\pm]} \C_1 \rightarrow M \otimes_{\C[q^\pm]} \C_1 \xrightarrow{f|_{q=1}} N \otimes_{\C[q^\pm]} \rightarrow 0
$$
Since $ N $ is free, $ Tor_1^{\C[q^\pm]}(N, \C_1) = 0 $.  On the other hand, since $ f|_{q=1} $ is injective, the map $Tor_1^{\C[q^\pm]}(N, \C_1) \rightarrow K \otimes_{\C[q^\pm]} \C_1 $ is surjective.  Thus $ K \otimes_{\C[q^\pm]} \C_1 = 0 $.  Since $ K $ is a submodule of a free module, it is free, and thus $ K = 0 $.
\end{proof}

\section{The structure of $\End_{\ASp[q^\pm]}(1^m)$ and cyclotomic Hecke algebras} \label{se:cyclotomic}

This section stands somewhat separate from the main results in the paper but may be of independent interest. As we noted earlier, the weight $1^m$ plays an important role in the categories we have studied. We will now discuss in a little more detail the structure of the $E[q^\pm]$-algebra $\End_{\ASp[q^\pm]}(1^m)$. 

\subsection{Affine Hecke algebras}

The annular braid group $AB_m$ has a natural quotient which is the affine Hecke algebra. We will consider a slight variant, denoted $\hH_m[q^\pm]$, which is the quotient of $E[q^\pm][AB_m]$ by the relation
$$T_i^2 = (q-q^{-1})T_i + 1 \ \ \text{ for } i=1,\dots,m-1.$$
As usual, we denote by $\hH_m(q)$ the same algebra over $\C(q)$ and by $\hH_m$ the specialization to $q=1$. Note that $\hH_m \cong E[AS_m]$ where $AS_m$ is the affine symmetric group.

\begin{Remark}
The usual generators of the affine Hecke algebra use $qT_i$ instead of $T_i$ which means that the Hecke relation becomes the more familiar $T_i^2 = (q^2-1)T_i + q^2$ while the relation $T_iX_iT_i = X_{i+1}$ becomes $T_iX_iT_i = q^2X_{i+1}$.
\end{Remark}

Now consider the composition
$$\hH_m[q^\pm] \rightarrow \End_{\ASp[q^\pm]}(1^m) \rightarrow \End_{\Sym(\ux)}(R^m_{n,[q^\pm]}).$$
By Proposition \ref{pr:ABsurjective} this composition, denoted $\chi$, is surjective if we tensor over $\C(q)$. Moreover, the second map is an isomorphism over $\C(q)$. We conjecture that both these results also hold over $\C[q^\pm]$. So to understand $\End_{\ASp[q^\pm]}(1^m)$ we need some description of the kernel of $\chi$.

\subsection{When $q=1$}

\begin{Conjecture}\label{conj:2}
The kernel of $\chi: E[AS_m] \rightarrow \End_{\Sym(\ux)}(R^m_n)$ is generated by
\begin{equation}\label{eq:conj2}
e_{sgn}(T_1, \dots, T_k)(h_{n-k} - e_1 h_{n-k-1} + e_2 h_{n-k-2} - \dots + (-1)^{n-k} e_{n-k})
\end{equation}
for $k=0, \dots, \min(n,m-1)$. Here $e_{sgn}(T_1, \dots, T_k) \in \C[S_k] \subset \C[S_m]$ is the sign idempotent and $h_j=h_j(X_1, \dots, X_{k+1})$ denotes the homogeneous symmetric function of degree $j$.
\end{Conjecture}

This conjecture interpolates between the Cayley-Hamilton theorem and Schur-Weyl duality. This is because when $k=0$ the relation in (\ref{eq:conj2}) becomes
$$X_1^{n} - e_1 X_1^{n-1} + \dots \pm e_{n-1} X_1 \mp 1 = 0.$$
However, under the isomorphism
$$R^m_n \cong \Maps_{SL_n}(SL_n, \End((\C^n)^{\otimes m}))$$
the image of $\chi(X_1^{n} - e_1 X_1^{n-1} + \dots \pm e_{n-1} X_1 \mp 1)$ is the map $g \mapsto p(g) \otimes I \otimes \cdots \otimes I$, where $p(y)$ is the characteristic polynomial of $ g $. Thus, this relation holds by Cayley-Hamilton theorem.

On the other hand, when $k=n$ (assuming $m > n$), then the relation in (\ref{eq:conj2}) becomes just $e_{sgn}(T_1, \dots, T_n)$ which, by Schur-Weyl duality, generates the kernel of $\C[S_m] \rightarrow \End((\C^n)^{\otimes m})$ and thus also holds in $Maps_{SL_n}(SL_n, \End((\C^n)^{\otimes m})$.

\subsection{Over $\C(q)$}

Define $e_j^{(i)} \in \hH_m[q^\pm]$ recursively using
\begin{equation}\label{eq:e'}
e_j^{(i)} := e_j^{(i-1)} + (q^{2}-1)[X_{i-1} e_{j-1}^{(i-1)} - X_{i-1}^{2} e_{j-2}^{(i-1)} + X_{i-1}^{3} e_{j-3}^{(i-1)} - \dots].
\end{equation}
This follows the definition of $e_j^{(i)} \in R^m_{n,[q^\pm]}$ from section \ref{sec:qDefR}.

\begin{Conjecture}\label{conj:1}
The kernel of $\chi: \hH_m(q) \rightarrow \End_{\Sym(\ux)}(R^m_{n,(q)})$ is generated by
\begin{equation}\label{eq:1}
X_i^{n} - e_1^{(i)} X_i^{n-1} + e_2^{(i)} X_i^{n-2} - \dots + (-1)^n e_n^{(i)} = 0
\end{equation}
for $i=1, \dots, m$.
\end{Conjecture}

Some remarks are in order. If one conjugates, for instance, the relation in (\ref{eq:1}) when $i=1$ by $T_1$ then one obtains the relation in (\ref{eq:1}) when $i=2$ plus an extra term. This extra term contains a factor of $q-1$ but after dividing by $q-1$ gives us another relation. One can repeat this argument to get new relations, such as the ones in (\ref{eq:conj2}), but it seems only as long as $q$ is not a root of unity.

Thus, Conjecture \ref{conj:1} is not true over $\C[q^\pm]$ and, in particular, when $q=1$. Moreover, although it might be tempting to only impose relation (\ref{eq:1}) when $i=1$ (like for the usual cyclotomic affine Hecke algebras), this is also incorrect.

\subsection{Over $\C[q^\pm]$}

\begin{Conjecture}\label{conj:4}
The kernel of $\chi: \hH_m[q^\pm] \rightarrow \End_{\Sym(\ux)}(R^m_{n,[q^\pm]})$ is generated by
\begin{equation}\label{eq:conj4}
[k+1]! \cdot e_{sgn}(T_i, \dots, T_{i+k-1}) \left(h_{n-k} - e_1^{(i)} h_{n-k-1} + \dots + (-1)^{n-k} e_{n-k}^{(i)} \right)
\end{equation}
for $k=0, \dots, \min(n,m-1)$ and $ i = 1, \dots, m-k$.
\end{Conjecture}
\begin{Remark}
Here $[k+1]! \cdot e_{sgn}(T_i, \dots, T_{i+k-1}) \in \H_m[q^\pm]$ is the sign quasi-idempotent inside the finite Hecke algebra where the sign representation corresponds to $T_i$ acting by $-q^{-1}$. Moreover, $h_j=h_j(X_i, \dots, X_{i+k})$ is the homogeneous symmetric function of degree $j$.
\end{Remark}

Let us explain why the term in (\ref{eq:conj4}) belongs to the kernel of
$$ \chi :  \hH_m[q^\pm] \rightarrow \End_{\Sym(\ux)}(R^m_{n,[q^\pm]}).$$
The map $\chi$ can be factored as
$$\hH_m[q^\pm] \xrightarrow{\delta} K^{\A}(Z(1^m,1^m)) \xrightarrow{S} \End_{\Sym(\ux)}(R^m_{n,[q^\pm]}).$$
Thus it is enough to show that the image of (\ref{eq:conj4}) under $\delta$ is zero.

In \cite{CKL,Ca} we constructed a categorical $\sl_m$ action on $\oplus_\uk Y(\uk)$ where $\uk = (k_1, \dots, k_m)$ with $\sum_i k_i = m$. Consider the composition
\begin{equation}\label{eq:6}
E_{i+k-1}^{(k)} \dots E_{i+1}^{(2)} E_i: K^\A(Y(1^m)) \rightarrow K^\A(Y(1^{i-1},k+1,1^{m-k})).
\end{equation}
On the one hand, it is fairly straight-forward to compute this composition. More precisely, consider the subvariety $X_i^k(1^m) \subset Y(1^m)$ given by the locus where $zL_{i+k} \subset L_{i-1}$ and the associated diagram
\begin{equation}\label{eq:4}
\begin{CD}
X_i^k(1^m) @>i>> Y(1^m)  \\
@VqVV \\
Y(1^{i-1},k+1,1^{m-k})
\end{CD}
\end{equation}
Here $i$ is inclusion and $q$ the projection which forgets $L_i,L_{i+1}, \dots, L_{i+k-1}$. Thus we can consider $X_i^k(1^m) \subset Y(1^m) \times Y(1^{i-1},k+1,1^{m-k})$ via the embedding  $i \times q$. Then the functor in (\ref{eq:6}) is induced by the kernel
$$[\sP] := [\cO_{X_i^k(1^m)} \otimes \sL] \in K^\A(Y(1^m) \times Y(1^{i-1},k+1,1^{m-k}))$$
where $\sL$ is some line bundle (this line bundle can be explicitly identified but we do not need to know it precisely).

On the other hand, we have
\begin{align*}
K^\A(Y(1^m)) & \cong (\C^n)^{\otimes(i-1)} \otimes (\C^n)^{\otimes k} \otimes (\C^n)^{\otimes(m-k)} \\
K^\A(Y(1^{i-1},k+1,1^{m-k})) &\cong (\C^n)^{\otimes(i-1)} \otimes \Alt{k}(\C^n) \otimes (\C^n)^{\otimes(m-k)}
\end{align*}
and the main result of \cite{CKM} implies that the composition in (\ref{eq:6}) is induced by the (quasi)projection $(\C^n)^{\otimes k} \rightarrow \Alt{k}(\C^n)$ corresponding to the sign representation. It then follows that (up to a power of $q$)
$$[\sP^L] * [\sP] \cong \delta([k+1]! \cdot e_{sgn}(T_i, \dots, T_{i+k-1})) \in K^\A(Y(1^m) \times Y(1^m))$$
where $\sP^L$ denotes the left adjoint of $\sP$.

This gives a geometric interpretation of $\delta([k+1]! \cdot e_{sgn})$. Thus relation (\ref{eq:conj4}) follows if we can show that
\begin{equation}\label{eq:3}
i^* \delta \left(h_{n-k} - e_1^{(i)} h_{n-k-1} + \dots + (-1)^{n-k} e_{n-k}^{(i)} \right) = 0
\end{equation}
inside $K^\A(X_i^k(1^m))$. Using the exact sequence
$$0 \rightarrow L_{i+k}/L_{i-1} \rightarrow z^{-1} L_{i-1}/L_{i-1} \rightarrow z^{-1} L_{i-1}/L_{i+k} \rightarrow 0$$
we see that $[(z^{-1} L_{i-1}/L_{i+k})^\vee] = e_1^{(i)} - [(L_{i+k}/L_{i-1})^\vee]$. On the other hand, $$\bigwedge^{n-k}(z^{-1} L_{i-1}/L_{i+k})^\vee = 0,$$ since $\dim (z^{-1} L_{i-1}/L_{i+k}) = n-k-1$. Moreover, $[(L_{i+k}/L_{i-1})^\vee] = X_i + \dots + X_{i+k}$. Then (\ref{eq:3}) follows from the general fact that
$$\Alt{r}(A-B) = \Alt{r}(A) - \Alt{r-1}(A) \cdot B + \Alt{r-2}(A) \cdot \Sym^2(B) - \Alt{r-3}(A) \cdot \Sym^3(B) + \dots$$
and that $\Sym^j(X_i + \dots + X_{i+k}) = h_j$.

\section{Comparison with the Steinberg variety}

Denote by $\B$ the full flag variety of $SL_n$ and $Z$ the associated Steinberg variety. Since the Steinberg variety is the fibre product of $T^* \B$ with itself over its affinization its equivariant K-theory $K^{\A}(Z)$ is equipped with a convolution product. The following result is due to Ginzburg \cite[Theorem 7.2.5]{CG} and Kazhdan-Lusztig \cite{KL}.

\begin{Proposition}\label{prop:CG}
There exists an isomorphism of algebras between $K^\A(Z)$ and the affine Hecke algebra $\bhH_n[q^\pm]$.
\end{Proposition}

In the result above $\bhH_n[q^\pm]$ denotes the affine Hecke algebra generated over $\C[q^\pm]$ by the $T_i$'s and $X_i$'s with the usual relations but without the $e_i$'s. In other words, $\bhH_n[q^\pm] = \hH_n[q^\pm] \otimes_{E} \C$. We would like to explain now the relation between Proposition \ref{prop:CG} and our results.

We will need to consider only the special case $m=n$. In this case we have varieties $Y(1^n)$ and $Z(1^n,1^n)$.

\begin{Lemma}
We have natural $\A$-equivariant open embeddings $i: T^* \B \rightarrow Y(1^n)$ and $j: Z \rightarrow Z(1^n,1^n)$.
\end{Lemma}
\begin{proof}
Let $\C^n = z^{-1}L_0/L_0$ and consider the space $\C^{n^2} = \C^n \otimes {\rm span}\{z^{-1}, \dots, z^{-n}\}$. Then
$$Y(1^n) = \{L_0 \subset \dots \subset L_n \subset \C^{n^2}: \dim(L_i/L_{i-1})=1, zL_i \subset L_{i-1}\}.$$
Now denote by $P: \C^{n^2} \rightarrow \C^n$ the projection which takes $\C^n \otimes z^{-i}$ to zero if $i > 1$ and onto $\C^n$ if $i=1$. Consider inside $Y(1^n)$ the open locus $U$ consisting of $L_\bullet$ such that $P(L_n) = \C^n$.

On $U$ we have an isomorphism $P: L_n \xrightarrow{\sim} \C^n$. Using this isomorphism we can transfer the map $z$ to a nilpotent endomorphism of $\C^n$ and we can transfer $ L_1, \dots, L_{n-1} $ to a flag in $ \C^n$. Thus we get a map $U \rightarrow T^* \B$. One can check that this map is actually an isomorphism.  This idea is originally due to Ngo \cite{Ngo} and it is also a special case of Theorem 3.2 in \cite{MVy}.

Under this isomorphism the map $Y(1^m) \rightarrow Gr$ which takes $L_\bullet \mapsto L_n$ corresponds to the affinization map of $T^* \B \rightarrow \sN$ to the nilpotent cone. It follows that $Z = T^* \B \times_\sN T^* \B$ naturally embeds into $Z(1^n,1^n) = Y(1^n) \times_{Gr} Y(1^n)$.
\end{proof}

Thus we have the following maps
$$\hH_n[q^\pm] \rightarrow K^\A(Z(1^n,1^n)) = \End_{\ASp[q^\pm]}(1^n) \xrightarrow{j^*} K^\A(Z) \cong \bhH_n[q^\pm]$$
where $j^*$ denotes the restriction on K-theory using $j$.

\begin{Proposition}
The composition map $\hH_n[q^\pm] \rightarrow \bhH_n[q^\pm]$ sends $T_i \mapsto T_i, X_i \mapsto X_i$ and
$$e_i^{(j)} \mapsto e_i(q^{2} X_1, q^{2} X_2, \dots, q^{2} X_{j-1}, X_j, \dots, X_n)$$
where the latter are elementary symmetric functions in the variables indicated.
\end{Proposition}
\begin{proof}
Suppose $\B = \{0=V_0 \subset V_1 \subset \dots \subset V_n = \C^n\}$. Then it is clear $L_i/L_{i-1}$ is mapped to the natural bundle $V_i/V_{i-1}$ on $T^* \B$ which explains why $X_i \mapsto X_i$.

Next, notice that $T_i$ can be identified on the one hand with the structure sheaf of the fibre product $T^* \B \times_{(T^* \B)_i} T^* \B$ where $(T^* \B)_i$ is obtained from $T^* \B$ by forgetting $V_i$. On the other hand the image of $T_i$ in $K(Z(1^n,1^n))$ is identified with the structure sheaf of $Y(1^n) \times_{Y(1^n)_i} Y(1^n)$ where $Y(1^n)_i$ is obtained from $Y(1^n)$ by forgetting $L_i$. It is then easy to check that
$$j^{-1}((1^n) \times_{Y(1^n)_i} Y(1^n)) = T^* \B \times_{(T^* \B)_i} T^* \B.$$
This explains why $T_i \mapsto T_i$.

Finally, on $T^* \B$ we have $\sum_i X_i = [\C^n]^\vee$ and more generally $e_i(X_1, \dots, X_n) = [\Alt{i}(\C^n)^\vee]$. This shows that $e_i \mapsto e_i(X_1, \dots, X_n)$. The general expression for the image of $e_i^{(j)}$ is then a nontrivial but elementary calculation using the recursive relations
$$e_j^{(i+1)} + X_i e_{j-1}^{(i+1)} = e_j^{(i)} + q^{2} X_i e_{j-1}^{(i)}$$
which follow from (\ref{eq:e'}) (c.f. relations (\ref{eq:e1}) and (\ref{eq:e2})).
\end{proof}

It is worth noting what happens to all the conjectured relations (\ref{eq:conj4}). They are all mapped to zero since
\begin{equation}\label{eq:ehrel}
h_{n-k} - e_1^{(i)} h_{n-k-1} + e_2^{(i)} h_{n-k-2} - \dots + (-1)^{n-k} e_{n-k}^{(i)} = 0
\end{equation}
by the standard relations between homogeneous and elementary symmetric functions. Thus Conjecture \ref{conj:4} recovers Proposition \ref{prop:CG}.

\begin{Remark}
Note that in relation (\ref{eq:ehrel}) the $e$'s are functions of $X_1, \dots, X_n$ while the $h$'s are homogeneous symmetric functions in $X_i, \dots, X_{k+i}$.
\end{Remark}

\end{document}